\documentclass[12pt,a4paper,twoside]{article}
\usepackage{cite}
\usepackage{comment}
\usepackage{amsmath,bm}
\usepackage{amscd}
\usepackage{amssymb}
\usepackage[pdftex]{color,graphicx,hyperref}
\usepackage{enumitem}
\usepackage[utf8]{inputenc}
\usepackage{latexsym}
\usepackage{cases}
\usepackage{xfrac}
\usepackage[normalem]{ulem}
\usepackage{environ}
\usepackage{xcolor}
\usepackage{mathtools}
\usepackage{stackengine}
\usepackage{geometry}
\usepackage{esint}
\usepackage[T1]{fontenc}
\usepackage[english]{babel}
\usepackage{tikz}
\usetikzlibrary{positioning}
\usepackage{framed}
\usepackage{fancyhdr}

\usepackage{color}

\newcommand{\Div}{\operatorname{div}}

\newcommand{\Ov}[1]{\overline{#1}}

\newcommand{\vr}{\varrho}

\newcommand{\vu}{\vc{u}}

\newcommand{\vv}{\vc{v}}

\newcommand{\vc}[1]{{\bf #1}}

\newcommand{\Grad}{\nabla}

\newcommand{\dx}{{\,\rm d} {x}}

\newcommand{\dt}{{\,\rm d} t }

\newcommand{\lr}[1]{\left( #1 \right)}

\newcommand{\dd}{\delta}
\newcommand{\R}{\mathbb{R}}
\newcommand{\Erel}{\mathcal{E}_{\mathrm{rel}}}
\newcommand{\Cv}{C_{\vv}}
\newtheorem{thm}{Theorem}
\newtheorem{lemma}[thm]{Lemma}
\newtheorem{prop}[thm]{Proposition}
\newtheorem{df}{Definition}
\newtheorem{rmk}{Remark}

\numberwithin{equation}{section}

\renewcommand{\dd}{{\rm d}}

\newcommand{\T}{\mathbb{T}}

\newcommand{\Om}{\Omega}

\newcommand{\vm}{\vc{m}}

\newcommand{\dz}{{\rm d} z}

\newcommand{\Gradx}{\nabla_x}
\newcommand{\Gradxz}{\nabla_{x,z}}
\newcommand{\Divx}{\operatorname{div}_x}
\newcommand{\Lambdaz}{\Lambda_0}

\newcommand{\tcbb}{\textcolor{black!30!blue}}

\begin{document}
	
	\title{Generalized solution and Weak-Strong uniqueness for a barotropic Euler-Riesz system}

	\author{Nilasis Chaudhuri\footnote{\texttt{nchaudhuri@mimuw.edu.pl}}}
	\date{}     
	
	%
	%
	%
	%
	%
	%
	%
	
	\maketitle
	\vspace{-5mm}
	{
		\small
		\centerline{Institute of Applied Mathematics and Mechanics, University of Warsaw,}
		\centerline{ul. Banacha 2 -- 02-097 Warsaw, Poland}
	}

	\begin{abstract}
We study the Euler--Riesz system on the torus $\mathbb T^d$, $d=2,3$: the compressible Euler
equations with barotropic pressure $p(\varrho)=a\varrho^\gamma$ ($\gamma>1$, $a>0$), coupled to a
repulsive nonlocal force ($\approx \varrho \Grad_x K \ast \varrho$) with Riesz kernel
$K(x)\propto|x|^{\beta-d}$ of order $\beta\in(0,2)$. Since $K$ is the kernel of the inverse
fractional Laplacian $(-\Delta)^{-\beta/2}$, we recast the force through the Caffarelli--Silvestre
extension as the trace of a local stress tensor, replacing the nonlocal interaction by a local
identity in one extra variable. For a repulsive kernel the total energy is coercive, and we use
this to introduce a notion of global-in-time \emph{dissipative solution} for arbitrarily large finite-energy
data. Our main result is weak (measure-valued)--strong uniqueness, for every order
$\beta\in(0,2)$ and every $\gamma>1$ independently: on any interval on which a strong solution
exists, every dissipative solution with the same initial data coincides with it and all defects
vanish. The proof rests on a suitable adaptation of relative energy.
	\end{abstract}

	
	
	\tableofcontents

	\section{Introduction}
	
We consider the isentropic Euler--Riesz system in the space-time domain
$(0,T)\times\mathbb T^d$, $\mathbb T^d=(\mathbb R/2\pi\mathbb Z)^d$: 
\begin{align}
	& \partial_t \vr + \Div_x( \vr \vu ) = 0, \label{eq:ERc:1}\\
	& \partial_t(\vr \vu) + \Div_x (\vr \vu \otimes \vu ) + \Grad_x p(\vr) + \vr \Grad_x \Phi = 0,\label{eq:ERm:1}
\end{align}
where $\vr(t,x)\ge0$ is the fluid density, $\vu(t,x)$ the velocity, and the pressure obeys
the isentropic law $p(\vr)=a\vr^\gamma$ with $a>0$, $\gamma>1$. The forcing $\vr\Grad_x\Phi$
is driven by a potential $\Phi$ determined by a kernel $K$ via $\Phi=K\ast\vr$ (made precise
below after neutralizing the background).

\medskip 

\noindent\textbf{Riesz kernel.}
In general $K$ is, up to sign, the classical \emph{Riesz kernel of order $\beta$},
\begin{equation}\label{riesz-kernel}
	K(x) = \pm\, C_{d,\beta}\, |x|^{\beta - d}, \qquad 0 < \beta < d,
\end{equation}
for a normalizing constant $C_{d,\beta}>0$, so that $K$ agrees (up to sign) with the kernel of
$(-\Delta)^{-\beta/2}$ on $\mathbb R^d$. We record the singularity exponent
\begin{equation}\label{alpha-beta}
	\alpha := d - \beta,
\end{equation}
so that \eqref{riesz-kernel} reads $K(x)=\pm C_{d,\beta}|x|^{-\alpha}$. The sign encodes the
nature of the interaction: $K=+C_{d,\beta}|x|^{-\alpha}$ is \emph{repulsive},
$K=-C_{d,\beta}|x|^{-\alpha}$ \emph{attractive}. We are interested in the singular regime of
$\alpha$, in which the coupling between \eqref{eq:ERc:1}--\eqref{eq:ERm:1} and $\Phi$ is most
delicate.

On $\mathbb T^d$ the total mass $M=\int_{\mathbb T^d}\vr\,dx$ is conserved; we neutralize the
uniform background so that the convolution is well defined on the torus, setting
\begin{equation*}
	\Phi = K \ast (\vr - \bar{\vr}), \qquad \bar{\vr} = |\mathbb{T}^d|^{-1} M .
\end{equation*}
This convolution is intimately linked to the inverse fractional Laplacian. From now on we
restrict to the fractional regime $\beta = 2s$ with $s\in(0,1)$, i.e.\ $\beta\in(0,2)$, so that
$(-\Delta)^{-\beta/2} = (-\Delta)^{-s}$. Absorbing the normalizing constant $C_{d,\beta}$ into
$\Phi$, we take $K$ to be the kernel of $(-\Delta)^{-s}$,\footnote{On $\mathbb T^d$ the
	normalization is $\widehat{K\ast f}(k)=|k|^{-2s}\hat f(k)$ for $k\neq0$.} so that the Riesz
potential solves the fractional Poisson equation
\begin{equation*}
	(-\Delta)^s \Phi = \vr - \bar{\vr} \quad \text{on } \mathbb{T}^d,
\end{equation*}
where $\alpha = d - 2s$ by \eqref{alpha-beta}. Thus, the system \eqref{eq:ERc}-\eqref{eq:ERr} with repulsive Riesz kernel \eqref{riesz-kernel} now reads as
	\begin{align}
		& \partial_t \vr + \Div_x( \vr \vu ) = 0, \label{eq:ERc}\\
		& \partial_t(\vr \vu) + \Div_x (\vr \vu \otimes \vu ) + \Grad_x p(\vr) + \vr \Grad_x \Phi = 0,\label{eq:ERm} \\
		&(-\Delta)^s \Phi = \vr - \bar{\vr}. \label{eq:ERr}
	\end{align}
	The system is supplemented with initial data 
	\begin{align}\label{id}
		\vr(0,\cdot) = \vr_0 (\cdot) \text{ and } \mathbf{m}(0,\cdot)=(\vr \vu) (0,\cdot) =(\vr \vu)_0 \text{ in } \T^d,
	\end{align}
	where we denote momentum by $\vm=\vr \vu$.
	
	The Euler--Riesz system \eqref{eq:ERc}--\eqref{eq:ERm} describes a compressible fluid whose particles interact through a singular Riesz kernel, and provides a unified framework encompassing several classical models of mathematical physics. The Coulomb case in $d=3$, where $\alpha=d-2=1$ and $K\propto|x|^{-1}$ is the Newtonian
	kernel, corresponds to the endpoint $s=1$ (formal, since our analysis is carried out for
	$s\in(0,1)$) and recovers the Euler--Poisson system,, which models both self-gravitating gaseous stars in the \emph{attractive} case $K = -C_{d}|x|^{-1}$ and charged plasmas or electron gases in the \emph{repulsive} case $K = +C_{d}|x|^{-1}$; the general Riesz exponent further captures mean-field descriptions of large interacting particle systems. It thus serves as a canonical model for the interplay between compressible transport, pressure, and nonlocal interaction forces.
	
	\medskip
	The sign of the kernel governs the variational structure of the system. We put our focus on the \emph{repulsive} case $K = +C_{d,\beta}|x|^{-\alpha}$, for which the system possesses a coercive energy. Indeed, a formal computation shows that smooth solutions conserve the total energy
	\begin{equation}\label{eq:energy}
		\frac{d}{dt}\, \mathcal{E}(t) = 0, \qquad
		\mathcal{E}(t) := \int_{\mathbb{T}^d} \left( \frac{1}{2}\vr|\vu|^2 + P(\vr) \right) dx
		+ \frac{1}{2}\int_{\mathbb{T}^d} (\vr-\bar\vr)\, K\ast(\vr-\bar\vr)\,dx,
	\end{equation}
	where $P(\vr) = \tfrac{a}{\gamma-1}\vr^\gamma$ is the pressure potential. The derivation of this in the formal level is quite standard, but for self containment, we give a derivation in the Appendix \ref{sec:eng-formal}. \par
	 In the repulsive case the interaction term is nonegative, and it is given by
	\begin{equation*}
		\frac{1}{2}\int_{\mathbb{T}^d} (\vr-\bar\vr)\, K\ast(\vr-\bar\vr)\,dx
		= \frac{1}{2}\, \big\| (-\Delta)^{s/2}\Phi \big\|_{L^2(\mathbb{T}^d)}^2 \ge 0.
	\end{equation*}
	Therefore, $\mathcal{E}(t)$ in \eqref{eq:energy} controls all three contributions -- kinetic, internal, and interaction -- as a genuinely positive energy, providing the natural a priori estimate for the analysis.
	\medskip
	By contrast, in the \emph{attractive} (self-gravitating) case $K = -C_{d,\alpha}|x|^{-\alpha}$ the interaction term in \eqref{eq:energy} carries the opposite sign,
	\begin{equation*}
		\frac{1}{2}\int_{\mathbb{T}^d} (\vr-\bar\vr)\, K\ast(\vr-\bar\vr)\,dx
		= -\frac{1}{2}\, \big\| (-\Delta)^{s/2}\Phi \big\|_{L^2(\mathbb{T}^d)}^2 \le 0.
	\end{equation*}
	Keeping our focus on repulsive case, goal is to the study of \emph{generalized solutions} -- weak and measure-valued -- of the Euler--Riesz system \eqref{eq:ERc}--\eqref{eq:ERr}. The main difficulty is already visible in the nonlocal force $\vr\Grad_x\Phi$: it is a product of
	two quantities, and the bounds coming from the energy are, on their own, not enough to make sense
	of this product\footnote{The energy controls $\Phi$ only in $\dot H^{s}$, so the field
		$\Grad_x\Phi$ lies in $\dot H^{s-1}$; for $s<1$ this is a space of distributions, and the pointwise
		product $\vr\,\Grad_x\Phi$ cannot be defined by integrability of $\vr$ alone.}. Any theory of
	generalized solutions therefore has to find some extra structure that turns this product into
	something one can actually control.
	\medskip
	
	The cleanest place to see what such a structure looks like is the Coulomb case
	$\alpha=d-2$ ($s=1$, $\beta=2$), where \eqref{eq:ERr} becomes the classical Poisson equation
	and the system reduces to the Euler--Poisson system
	\begin{align}
		& \partial_t \vr + \Div_x( \vr \vu ) = 0, \label{eq:EPc-final}\\
		& \partial_t(\vr \vu) + \Div_x (\vr \vu \otimes \vu ) + \Grad_x p(\vr) + \vr \Grad_x \Phi = 0,\label{eq:EPm-final} \\
		& -\Delta \Phi = \vr - \bar{\vr}, \label{eq:EPp}
	\end{align}
	with $\Phi=K\ast(\vr-\bar\vr)$ the Newtonian potential: $K=C_{3}\,|x|^{-1}$ in $d=3$, while
	in $d=2$ the kernel is logarithmic, $K=-\tfrac{1}{2\pi}\log|x|$ (up to a smooth periodic
	correction on $\mathbb T^d$). \par 
	For small, smooth, irrotational data the Euler--Poisson
	system is globally well posed in the \emph{repulsive} (plasma) regime, where the electric
	field is dispersive: Guo~\cite{G1998} for electron dynamics and Guo--Pausader~\cite{GP2011}
	for ion dynamics. For large data, by contrast, smooth solutions may form shocks in finite
	time. At the weak level the siutuatoion is bit delicate. Adapting the
	convex-integration method of De Lellis and Sz\'ekelyhidi~\cite{DS2009}, Feireisl~\cite{F2016b}
	showed that the Euler--Poisson system admits infinitely many weak solutions satisfying the energy inequality in $\R^d$ with $d=2,3$. Global finite-energy solutions with large spherically symmetric data, for both gaseous stars and plasmas, were constructed by Chen, He, Wang and Yuan~\cite{CHWY2024} as the inviscid limit of Navier--Stokes--Poisson approximations.

	For Euler--Poison system(\eqref{eq:EPc-final}, \eqref{eq:EPm-final} and \eqref{eq:EPp}), using $-\Delta\Phi = \vr-\bar\vr$ one key feature is the following pointwise identity:
	\begin{equation}\label{ms-P}
		(\vr-\bar\vr)\,\Grad_x\Phi
		= \Div_x\!\left( \tfrac12 |\Grad_x\Phi|^2\,\mathbb{I}_d - \Grad_x\Phi\otimes\Grad_x\Phi \right):=-\Div_x \mathbb{T}_{c}[\Grad_x \Phi],
	\end{equation}
	that makes the nonlocal force is a pure divergence form and at the level of the energy it holds
	\begin{align*}
		\tfrac12\int_{\mathbb T^d}(\vr-\bar\vr)\Phi \dx =\tfrac12\int_{\mathbb T^d}|\Grad_x\Phi|^2 \dx .
	\end{align*}
	The identity \eqref{ms-P}, however, holds in either
	dimension ($d=2 \text{ or } 3$), since it uses only \eqref{eq:EPp}.
	The gain from \eqref{ms-P} is that the force no longer needs joint control of $\vr$ and $\Grad_x\Phi$: as a
	divergence tested against a velocity field it becomes
	$\int_{\mathbb T^d}\mathbb T_c[\Phi]:\Grad_x\pmb\psi\,dx$, bounded by
	$\|\Grad_x\pmb\psi\|_{L^\infty}$ times the field energy, with no derivative on $\Phi$.
	Building on this identity, Carrillo et al. ~\cite{CDGS2024} carried the
	measure-valued framework for pressureless Euler--Poisson.	For $0<s<1$ no such pointwise identity holds on $\mathbb T^d$: the force is genuinely nonlocal
	and is not the divergence of any local stress built from $\Phi$.
	\medskip

For the compressible Euler--Riesz system, strong solutions have been studied by Choi and
Jeong~\cite{CJ2022}, Choi, Jung and Lee~\cite{CJL2025}, Danchin and Ducomet~\cite{DD2022}, and
Danchin and Mucha~\cite{DM2024} for general nonlocal interaction potentials. Following the
compensated-integrability framework of Serre~\cite{S2019}, Alves, Grafakos and
Tzavaras~\cite{AlGT2026} write the nonlocal force as a divergence-form tensor and analyze it
through a bilinear fractional integral operator. In the large-data regime,
Carrillo et al.~\cite{CCCY2025} recently established global existence and nonlinear stability of
finite-energy solutions with spherically symmetric initial data, for both attractive and
repulsive Riesz potentials, via a vanishing-viscosity limit. \par

To our knowledge, existence and uniqueness of measure-valued or dissipative solutions has not yet
been addressed. The obstruction is structural: for the fractional operator $(-\Delta)^s$ with
$s\in(0,1)$ there is no pointwise analogue of the identity \eqref{ms-P} rendering the nonlocal
force a local divergence, as in the Euler--Poisson case. This is precisely the gap the present
work fills: the Caffarelli--Silvestre extension supplies the missing divergence structure, at the
cost of one extra variable. \par

	Instead of the nonlocal stress tensor of \cite{AlGT2026}, we use the Caffarelli--Silvestre
	extension~\cite{CaS2007}\footnote{The realization of $(-\Delta)^s$ as a Dirichlet-to-Neumann
		operator for a degenerate local problem goes back to Molchanov and Ostrovskii~\cite{MO1969} and
		Spitzer~\cite{S1958} in the probabilistic setting; we use the analytic formulation and
		regularity theory of Caffarelli and Silvestre~\cite{CaS2007}.}, by now a standard device for the
	fractional Laplacian. One adds a variable $z\in\mathbb R_+$ and lifts $\Phi$ to $U(t,x,z)$ on the half-cylinder $\Omega=\mathbb T^d\times\mathbb R_+$, where $U$ is the unique decaying solution of
	\begin{equation}\label{CS-ext}
		\begin{cases}
			\Div_{x,z}\!\left( z^{1-2s}\,\Grad_{x,z} U \right) = 0 & \text{in } \Omega,\\[2pt]
			-c_{d,s}\,\displaystyle\lim_{z\to 0} z^{1-2s}\,\partial_z U = \vr - \bar\vr & \text{on } \mathbb{T}^d\times\{0\},
		\end{cases}
		\qquad \lim_{z\to\infty} U(x,z) = 0,
	\end{equation}
	where $c_{d,s}>0$ with explicit value in \eqref{constCDS}.
	The equation in $\Omega$ is local: the nonlocal operator $(-\Delta)^s$ on the boundary is replaced by a weighted divergence-form equation in the extended space $\Omega$, and the datum $\vr-\bar\vr$ becomes the weighted normal flux of $U$ at $z=0$. Since $U|_{z=0}=\Phi$, formally we have $\Grad_x\Phi=\Grad_x U|_{z=0}$, and the momentum equation \eqref{eq:ERm} reads
	\begin{equation}\label{eq:ERm-ext}
		\partial_t(\vr\vu) + \Div_x(\vr\vu\otimes\vu) + \Grad_x p(\vr) + \vr\,\Grad_x U|_{z=0} = 0 .
	\end{equation}
	Testing \eqref{CS-ext} with $U$ itself shows that the Dirichlet energy of the extension is the interaction energy of the system: 
	\begin{align*}
		\tfrac{c_{d,s}}2\iint_\Omega z^{1-2s}|\Grad_{x,z}U|^2 \dx \, \dz = \tfrac12\int_{\mathbb T^d}(\vr-\bar\vr)\Phi \dx  = \tfrac12\|\vr-\bar\vr\|_{\dot H^{-s}}^2. 
	\end{align*}

	As a solution of the extension equation $\Div_{x,z}(z^{1-2s}\Grad_{x,z}U)=0$, $U$ carries a
	stress tensor in $\Omega$, namely the $(d+1)\times(d+1)$ tensor
	\begin{equation}\label{Tfull}
		\mathbb{T}_{\mathrm{ext}}[U] = z^{1-2s}\,\Grad_{x,z} U \otimes \Grad_{x,z} U -
		\tfrac12\, z^{1-2s}\, |\Grad_{x,z} U|^2\, \mathbb{I}_{d+1} .
	\end{equation}
	A detailed computation from \eqref{CS-ext}, carried out in Section~\ref{sec:stress}
	(see \eqref{div-Tfull}), shows that the $d$ horizontal rows of $\mathbb T_{\mathrm{ext}}[U]$
	are divergence-free, while the $z$-row is not unless $s=\tfrac12$; the obstruction is the
	$z$-dependence of the weight.
	
	The momentum equation is posed on $\mathbb T^d$ and we work with generalized solutions, so the test functions are $\pmb\psi=\pmb\psi(t,x)$. Extended to $\Omega$ they are horizontal and
	$z$-independent, $\pmb\Psi_{\mathrm{ext}}(t,x,z)=(\pmb\psi(t,x),0)^T$. Testing \eqref{Tfull} with $\Grad_{x,z}\pmb\Psi_{\mathrm{ext}}$, only the upper-left $d\times d$
	block is paired, and \eqref{Tfull} reduces to
	\begin{equation}\label{T}
		\mathbb{T}_{\rm ER}[U] = z^{1-2s}\left( \Grad_x U \otimes \Grad_x U - \tfrac12\,|\Grad_{x,z} U|^2\,\mathbb{I}_d \right).
	\end{equation}
	Note that $\mathbb{T}_{\rm ER}[U]$ retains the normal derivative through its trace --- the
	off-diagonal part sees only $\Grad_x U$, the trace the full $\Grad_{x,z}U$ --- so it is not
	the stress of a $d$-dimensional problem. Integrating by parts in $\Omega$ against the
	horizontal rows and using the flux condition in \eqref{CS-ext} yields the identity we are
	after
	\begin{equation}\label{eq:nonlocMS-intro}
		c_{d,s} \iint_{\Omega} \mathbb{T}_{\rm ER}[U] : \Grad_x \pmb{\psi} \,dx\,dz
		= \int_{\mathbb{T}^d} (\vr-\bar\vr)\,\Grad_x U|_{z=0} \cdot \pmb{\psi} \,dx ,
	\end{equation}
	the mean $\bar\vr$ contributing the extra term $ -\bar\vr\int_{\mathbb T^d}\Phi\,\Div_x\pmb\psi \dx $, which is harmless. The details are in Section~\ref{sec:stress}. The gain is that the nonlocal force no longer appears in the energy estimates through $\Grad_x\Phi$, but only through $\mathbb T_{\rm ER}[U]$ paired with $\Grad_x\pmb\psi$. Since $|\mathbb T_{\rm ER}[U]|\le z^{1-2s}|\Grad_{x,z}U|^2$ pointwise, such a term is controlled by $\|\Grad_x\pmb\psi\|_{L^\infty}$ times the interaction energy. Equation \eqref{eq:nonlocMS-intro} is the fundamental structural identity underlying the entire paper.

	\medskip
	This strategy is also motivated by the work of Serfaty and collaborators (see, \cite{Serfaty2020}, \cite{NRS2022}), Duerinckx \cite{Du2016}, Choi and Jung \cite{CJ2024b} and references therein, where the \emph{modulated energy} method is used to derive mean-field (and, in some regimes, fluid) limits for large systems of particles interacting through Coulomb and Riesz kernels. A central device there is a stress-tensor reformulation of the singular interaction energy, which for Riesz kernels is handled through the very same Caffarelli--Silvestre extension: the extension localizes the singular energy and supplies the coercivity needed to control the modulated energy along the limit. We borrow this extension viewpoint, and adapt it here to the question of generalized solutions and weak--strong uniqueness for the Euler--Riesz system.\par 
	\medskip

Following the work of DiPerna~\cite{D1985} and DiPerna and Majda~\cite{DM1987}, a theory of
measure-valued solutions has been developed for compressible fluids and related models (see
\cite{BDS11, FjMT2012, BF2018a, Bd2020, FL2018, CGZ2023}). In this this article is to
provide a concept of generalized solution based on \eqref{T} and \eqref{eq:nonlocMS-intro}.
Rather than defining it through Young measures, we work with the notion of a \emph{dissipative
	solution}, introduced for the compressible Euler system by Breit, Feireisl and
Hofmanov\'a~\cite{BeFH2020d} and by Feireisl, Luk\'a\v{c}ov\'a-Medvi\v{d}ov\'a and
Mizerov\'a~\cite{FLM2020i}. For compressible Euler the core idea rests on the convexity and
weak lower semicontinuity of the energy in the conservative variables: $(\vr,\vm)\longmapsto \frac{1}{2}\frac{|\vm|^2}{\vr}+P(\vr). $
Along a sequence bounded only in energy, the weak formulation and the energy inequality then
acquire \emph{defect measures} --- correction terms that absorb the concentrations and oscillations the nonlinear terms may develop in the limit, appearing in both the momentum balance and the energy inequality. Crucially, the defect in the momentum equation is controlled by the defect in the energy inequality that is retained as part of the definition. In our setting the stress \eqref{T} is not sign-definite, so its defect measures carry no sign. What replaces the sign is domination: the momentum defect stays controlled by the energy defect, and this is precisely what makes the dissipative formulation available for the fractional system.
	\medskip
	
Two fundamental properties justify the introduction of a generalized notion of solution,
both of which we establish in this work:
\begin{enumerate}[leftmargin=*]
	\item[(i)] \emph{Compatibility.} Every sufficiently smooth dissipative
	solution is a classical solution; 
	\item[(ii)] \emph{Weak--strong uniqueness.} Given the same initial data, a dissipative
	solution coincides with the strong solution for as long as the latter exists; 
\end{enumerate}
The proof of weak--strong uniqueness relies on the relative energy method, which
originated in the celebrated work of Dafermos~\cite{D1979} on hyperbolic conservation
laws and has since become the standard route to weak--strong uniqueness across fluid
models: for the incompressible Euler system at the measure-valued level by Brenier,
De Lellis and Sz\'ekelyhidi~\cite{BDS11}, and for the compressible Navier--Stokes and
Euler systems and their measure-valued and dissipative formulations by Feireisl and
collaborators~\cite{FJN2012, FNS2011, FPAW2016, C2024, BF2018a}. For a complete picture
we refer to the survey of Wiedemann~\cite{W2018} and the references therein. \par 
Closest to the present setting are the works of Alves and collaborators~\cite{AlCC2024,
		Al2024, AlGT2026} and Lattanzio and Tzavaras~\cite{LT2017}, who establish weak--strong
	uniqueness for Euler--Riesz systems with nonlocal interaction, including the Euler--Poisson
	case, at the level of weak solutions (assuming the existence of a weak solution) and
	$\gamma\ge 2-\tfrac{\beta}{d}$ (respectively $\gamma\ge 2-\tfrac{2}{d}$ in the Poisson case) --- via Hardy--Littlewood--Sobolev inequalities on $\mathbb R^d$ for $0<\beta <d$.
We highlight that, Alves et al.~\cite{AlGT2026} obtain weak--strong uniqueness for over the full Riesz range $\beta\in(0,d)$, where as in this particular extension approach restricts us for $\beta=2s\in(0,2)$, covering the whole range of $\gamma$ (See Remark \ref{compref}). For pressureless Euler--Poisson we further refer to Carrillo et al.~\cite{CDGS2024}; in the
absence of pressure their result yields a (partial) weak--strong uniqueness principle, and it is
their dissipative approach that we adapt to the present fractional setting. \par 


Accordingly, to establish weak--strong uniqueness we work
in the extended variable, with relative energy
\begin{equation}\label{RE:Def}
	\Erel(\tau)
	:= \int_{\mathbb T^d}\Big( \tfrac12\,\vr\, \Big|\frac{\vm}{\vr} - \vv\Big|^2
	+ P(\vr\,|\,r) \Big)(\tau)\,dx
	\;+\; \frac{c_{d,s}}{2}\iint_\Omega z^{1-2s}\,\big|\Gradxz (U-V)\big|^2(\tau)\,dx\,dz,
\end{equation}
where $(\vr,\vm,U)$ denotes the weak (generalized) solution and $(r,\vv,V)$ the strong solution.

\paragraph{Plan of the article.}
After fixing notation, we introduce in Section~\ref{sec:defn} the notion of dissipative weak solution and state the main result. Section~\ref{sec:justification} justifies the definition: we derive the stress tensor through which the nonlocal force is expressed, and show that the defect measures arise as the limits of a consistent approximation. Section~\ref{sec:relative} is devoted to the weak--strong uniqueness principle, via a relative energy inequality. 
 
\subsection{Notation}	\label{sec:notation}

We work on the flat torus $\mathbb T^d$, $d\ge1$. For $1\le p\le\infty$ and $k\in\mathbb N$, $L^p(\mathbb T^d)$, $W^{k,p}(\mathbb T^d)$, $H^k(\mathbb T^d)$ denote the usual Lebesgue and Sobolev spaces, $C^k(\mathbb T^d)$ the $k$-times continuously differentiable functions, and $C_c^\infty$ the smooth compactly supported ones; homogeneous fractional spaces $\dot H^\sigma(\mathbb T^d)$ with $0<\sigma <1$ are defined in \eqref{frac-sobolev} below. For a Banach space $X$  and $1\leq p\leq \infty$, we write $L^p_{t,X}=L^p(0,T;X)$, and abbreviate
$L^p_x=L^p(\mathbb T^d)$, $L^p_{t,x}=L^p((0,T)\times\mathbb T^d)$,
$C^k_{t,x}=C^k([0,T]\times\mathbb T^d)$, with subscripts indicating the variables.

 We write $A:B=A_{jk}B_{jk}$ for the Frobenius product of matrices. For $A\in\mathbb R^{d\times d}$, $\|A\|_{\mathrm{op}}:=\sup_{|\xi|=1}|A\xi|$ is the operator
(spectral) norm; for symmetric $A$ it equals $\max_i|\lambda_i(A)|$,
the largest eigenvalue in modulus.
For $\mathbf{a},\mathbf{b}\in\mathbb R^d$, $\mathbf{a}\otimes \mathbf{b}\in\mathbb R^{d\times d}$ is the matrix $(a\otimes b)_{jk}=a_j b_k$; in particular $A:(\xi\otimes\xi)=A\xi\cdot\xi$ for symmetric $A$ and $\xi \in \R^d$.

\subsubsection*{Scaler and Matrix valued Radon Measures}

$\mathcal M(\mathbb T^d)$ is the space of finite signed Radon measures --- the dual of $C(\mathbb T^d)$ --- with the total variation norm; $\mathcal M(\mathbb T^d;\mathbb R^d)$ and $\mathcal M(\mathbb T^d;\mathbb R^{d\times d}_{\mathrm{sym}})$ are the vector- and symmetric-matrix-valued analogues. We use two notions of non-negativity, and distinguish them by the ordering symbol:
\begin{itemize}
	\item $\mathcal M^+(\mathbb T^d)$ is the cone of non-negative scalar measures, $\mu\ge0$;
	\item $\mathcal M^+(\mathbb T^d;\mathbb R^{d\times d}_{\mathrm{sym}})$ is the cone of \emph{positive semidefinite} matrix-valued measures, i.e.\ $\mathcal N$ with $\mathcal N:(\xi\otimes\xi)\in\mathcal M^+(\mathbb T^d)$ for every $\xi\in\mathbb R^d$; we write $\mathcal N\succeq0$.
\end{itemize}

For $T>0$, $L^\infty_{\mathrm{weak}\text{-}*}(0,T;\mathcal M(\mathbb T^d))$ denotes the essentially bounded weak-$*$ measurable maps $(0,T)\to\mathcal M(\mathbb T^d)$; since $C(\mathbb T^d)$ is separable, this is the dual of $L^1(0,T;C(\mathbb T^d))$. By definition of $L^\infty_{\mathrm{weak}\text{-}*}\big(0,T;\mathcal M(\overline\Omega)\big)$, any
	$\boldsymbol\mu$ in this space is represented by a weakly-$*$ measurable family
	$\{\mu(t)\}_{t\in(0,T)}\subset\mathcal M(\overline\Omega)$ i.e.\ $t\mapsto\langle\mu(t),\varphi\rangle$
	is measurable for every $\varphi\in C(\overline\Omega)$ with $ \operatorname*{ess\,sup}_{t\in(0,T)}\|\mu(t)\|_{\mathcal M(\overline\Omega)}<\infty ,$ and acts on $C_c((0,T)\times\overline\Omega)$ by $ \langle\boldsymbol\mu,\Phi\rangle=\int_0^T\!\!\int_{\overline\Omega}\Phi(t,x,z)\,d\mu(t)(x,z)\,dt $. We abbreviate this as $d\boldsymbol\mu(t,x,z)=d\mu(t)(x,z)\,dt$, and write $\mu(t)$ for the slice at
	time $t$.\par
Let $\mathcal O$ = $\mathbb T^d$ or $\overline\Omega=\T^d\times [0,\infty)$. Fix a finite-dimensional inner-product space $V$ (below $V=\mathbb R$, $\mathbb
R^{d}$, or $\mathbb R^{d\times d}_{\mathrm{sym}}$) and let $\lambda$ be a
$V$-valued Radon measure on $\mathcal O$, with total variation
$|\lambda|\in\mathcal M^+(\mathcal O)$, $ |\lambda|(E):=\sup\Big\{\textstyle\sum_k|\lambda(E_k)|_V:\ \{E_k\}\ \text{a finite
	Borel partition of }E\Big\}$.
Given $\mu\in\mathcal M^+(\mathcal O)$, we write $\lambda\ll\mu$ iff
$|\lambda|\ll\mu$, equivalently $\mu(E)=0\Rightarrow\lambda(E)=0$ for all Borel
$E$. Radon--Nikod\'ym theorem then yields a density
$\tfrac{d\lambda}{d\mu}\in L^1_{\mathrm{loc}}(\mu;V)$, unique $\mu$-a.e., with
$\lambda=\tfrac{d\lambda}{d\mu}\,\mu$.
For $V=\mathbb R$ this is the signed scalar case (and a
positive measure is the sub-case $\lambda\ge0$, where $|\lambda|=\lambda$ and the
density is $\ge0$ a.e.). For $V=\mathbb R^{d\times d}_{\mathrm{sym}}$,
$\lambda\ll\mu$ holds iff $\lambda_{ij}\ll\mu$ for all $i,j$, iff $\lambda{:}(\xi\otimes\xi)\ll\mu$ for all $\xi\in\mathbb R^d$; the
density is symmetric a.e., and if $\lambda$ is positive semidefinite
($\lambda{:}(\xi\otimes\xi)\ge0$ for all $\xi$) then
$\tfrac{d\lambda}{d\mu}\succeq0$ a.e.. The above is collection of statements from Ambrossio, Fusco  and Pallara \cite[Chapter 1]{AFP2000}.

\subsubsection*{Weak limits of nonlinear quantities}
Passing to weak limits in nonlinear terms is the central difficulty of the compactness
argument, and requires some care. Let $Z_n$ be a sequence on $\mathbb T^d$ and
$F:\mathbb R\to\mathbb R$ continuous. Two objects must be distinguished.
\begin{itemize}[leftmargin=*]
	\item If $F(Z_n)$ is bounded in $L^1(\mathbb T^d)$, then along a subsequence it converges
	weakly-$*$ to a finite Radon measure, which we denote $\displaystyle \Ov{F(Z)} := \underset{n\to\infty}{w^*\text{-}\lim}\ F(Z_n)\ \in\ \mathcal M(\mathbb T^d)$.
	This limit exists under the $L^1$ bound \emph{alone}; no convergence of $Z_n$ itself is
	needed, and $\Ov{F(Z)}$ is a measure, not in general a function.
	
	\item Separately, if $Z_n\rightharpoonup Z$ in $L^p(\mathbb T^d)$ for some $1<p<\infty$
	(equivalently, $Z_n\rightharpoonup^{\ast} Z$ in $L^\infty(\mathbb T^d)$), the weak limit $Z$
	is a function and the pointwise composition $F(Z)$ is defined a.e. The two need not agree:
	the \emph{defect} $	\Ov{F(Z)} - F(Z)\ \in\ \mathcal M(\mathbb T^d) $, and encodes both the
	oscillation and the concentration lost in the limit.
	\item When $F$ is convex and non-negative, lower semicontinuity fixes the sign:
	\begin{equation}\label{defect1}
		\Ov{F(Z)}\ \ge\ F(Z)
		\qquad\text{as measures,}\qquad\text{i.e.}\qquad
		\Ov{F(Z)} - F(Z)\ \in\ \mathcal M^+(\mathbb T^d).
	\end{equation}
	The matrix-valued analogue, where for each fixed $\xi$ the scalar map
	$\vr\mapsto F(Z):(\xi\otimes\xi)$ is convex, yields $\Ov{F(Z)}-F(Z)\succeq0$; this is examined
	in detail in \S\ref{sec:stability-defect}.
\end{itemize}

\begin{rmk}[Detailed Description of Defect measures]
	The rigorous analysis of $\Ov{F(Z)}$ for nonegative, convex, weakly lower semicontinuous $F$ and of the sign in \eqref{defect1} is carried out through the theory of Young measures. The sequence $Z_n$ generates a Young measure $\{\nu_y\}$; the weak-$*$ limit is $\Ov{F(Z)} = \langle\nu;F(\lambda)\rangle + \mathfrak R^{\mathrm{cd}}$, splitting into the barycentric value and a concentration defect, and the oscillation defect $\langle\nu;F\rangle - F(\vr)$ together with $\mathfrak R^{\mathrm{cd}}$ is non-negative for convex $F$ by Jensen's inequality. We refer to \cite[Chapter~1]{C2021} for a self-contained treatment. 
\end{rmk}
\paragraph{Weighted Sobolev spaces and the trace operator}
\label{sec:weighted-spaces}

Throughout, $s \in (0,1)$ and we work on the extended (half-cylinder) domain $ \Omega = \mathbb{T}^d \times \mathbb{R}_+, \, \overline{\Omega} = \mathbb{T}^d \times [0,\infty),$
with points $(x,z)$, $x \in \mathbb{T}^d$, $z > 0$, and $\Grad_{x,z} = (\Grad_x, \partial_z)$. The
extension is governed by the degenerate weight $z^{1-2s}$, and the relevant spaces are the following.
\begin{itemize}[leftmargin=*]
	\item The natural energy space is the homogeneous weighted Sobolev space
	\begin{equation}\label{weighted-space}
		\dot{H}^1(\Omega, z^{1-2s})
		= \overline{\big\{ U \in C_c^\infty(\overline{\Omega}) \big\}}^{\,[\,\cdot\,]},
		\qquad
		[U]^2 := \iint_{\Omega} z^{1-2s}\, |\Grad_{x,z} U|^2 \,dx\,dz,
	\end{equation}
	the completion of smooth functions on $\overline{\Omega}$ under the weighted Dirichlet
	seminorm $[\,\cdot\,]$, together with the far-field normalization $U(x,z) \to 0$ as
	$z \to \infty$.
	
	\item  On $\mathbb{T}^d \times \{0\} \cong \mathbb{T}^d$ we use the homogeneous
	fractional Sobolev spaces defined through Fourier series: for
	$u = \sum_{k \in \mathbb{Z}^d} \hat u_k\, e^{ik\cdot x}$ with $\hat u_0 = 0$,
	\begin{equation}\label{frac-sobolev}
		\| u \|_{\dot{H}^{\sigma}(\mathbb{T}^d)}^2 := \sum_{k \neq 0} |k|^{2\sigma}\, |\hat u_k|^2,
		\qquad \sigma \in \mathbb{R},
	\end{equation}
	and $(-\Delta)^{\sigma}$ acts as the Fourier multiplier $|k|^{2\sigma}$. All boundary data below
	are mean-zero, so restricting to $\hat u_0 = 0$ is harmless.
\end{itemize}
The two spaces above are linked by the (Dirichlet) trace, which realizes the boundary value of an
extended function. On smooth functions it is the pointwise limit onto $\mathbb T^d\times\{0\}$,
\begin{equation}\label{trace-ext}
	\Lambda_0 U(x) := \lim_{z\to0^+} U(x,z), \qquad U \in C_c^\infty(\overline\Omega),
\end{equation}
and, by the Caffarelli--Silvestre extension theory \cite{CaS2007} (Fourier form on $\mathbb T^d$ in Roncal and Stinga \cite{RS2016}), it extends to a continuous operator
$\Lambda_0:\dot H^1(\Omega,z^{1-2s})\to\dot H^s(\mathbb T^d)$; further properties are recorded in
\S\ref{sec:extension-trace}.

\begin{rmk}[Muckenhoupt structure]\label{rem:A2}
	For $s\in(0,1)$ one has $1-2s\in(-1,1)$, so the weight $\omega(z)=|z|^{1-2s}$ lies in the
	Muckenhoupt class $A_2$ --- the structural hypothesis underlying the weighted Sobolev theory of
	Fabes, Kenig and Serapioni \cite{FKS1982} used throughout.
\end{rmk}




\section{Definition of Dissipative Solution and Main Result}\label{sec:defn}

We now formulate the notion of a \emph{dissipative solution} studied in this paper. First we collect the properties of the extension needed to state the definition and state the main result.

	\subsection{The Caffarelli--Silvestre extension}
	\label{sec:extension-trace}
	
We now turn to the weak form of the extension \eqref{CS-ext}, recalling that $\Lambda_0$ denotes the Dirichlet trace introduced in Remark \ref{trace-ext}. Given mean-zero data $f=\vr-\bar\vr\in\dot H^{-s}(\mathbb T^d)$, the weak form of extension \eqref{CS-ext} of $f$ is the unique $U\in\dot H^1(\Omega,z^{1-2s})$ with
	\begin{equation}\label{CS_weak}
		\iint_{\Omega} z^{1-2s}\,\Grad_{x,z} U\cdot\Grad_{x,z}\Psi\,dx\,dz
		= c_{d,s}^{-1}\big\langle f,\Lambda_0\Psi\big\rangle_{\dot H^{-s},\dot H^{s}}
		\qquad\text{for all }\Psi\in\dot H^1(\Omega,z^{1-2s}),
	\end{equation}
	with $c_{d,s}>0$ the explicit constant fixed in \eqref{constCDS} (Appendix \ref{sec:extension-appendix}), where we discuss the explicit extension on $\T^d$ with the help of the Fourier representation. Next, we introduce the notation 
	\begin{align}\label{Not:CSext}
			\mathfrak E[f]:=U,
	\end{align}
	that says $U$ is extension of $f$ that satisfies \eqref{CS_weak}. 
	
	Now we state the important properties of the extension that we will use in next sections: 
	\begin{prop}[Existence, trace, and energy isometry]\label{prop:CS_weak}
		Let $s\in(0,1)$ and let $f\in\dot H^{-s}(\mathbb T^d)$ be mean-zero and $c_{d,s}$ is given by \eqref{constCDS}. Then \eqref{CS_weak} has a unique solution $U\in\dot H^1(\Omega,z^{1-2s})$, and:
		\begin{enumerate}[leftmargin=*]
			\item \textup{(Trace.)} The Dirichlet trace $\Lambda_0 U\in\dot H^s(\mathbb T^d)$ is well defined and solves $(-\Delta)^s\Lambda_0 U=f$ on $\mathbb T^d$ in the distributional sense\footnote{On $\mathbb T^d$, $(-\Delta)^s\Lambda_0 U=f$ is understood in the sense of Fourier series, $|k|^{2s}\widehat{\Lambda_0 U}(k)=\hat f(k)$ for $k\neq0$; since $f$ and $\Lambda_0 U$ are both mean-zero, this is equivalent to $\langle\Lambda_0 U,(-\Delta)^s\varphi\rangle=\langle f,\varphi\rangle$ for all mean-zero $\varphi\in C^\infty(\mathbb T^d)$, using that $(-\Delta)^s$ preserves $C^\infty(\mathbb T^d)$ and annihilates constants.}; equivalently, $\Lambda_0 U=\Phi$, the Riesz potential of $f$.
			\item \textup{(Energy isometry.)} The weighted energy equals the fractional boundary energy, and hence the negative norm of the source:
			\begin{equation}\label{CS-prop1}
				c_{d,s}\iint_{\Omega} z^{1-2s}\,|\Grad_{x,z} U|^2\,dx\,dz
				= \|\Lambda_0 U\|_{\dot H^s(\mathbb T^d)}^2
				= \|f\|_{\dot H^{-s}(\mathbb T^d)}^2 .
			\end{equation}
		\end{enumerate}
		For smooth $U$ the identity \eqref{CS_weak} is equivalent to the pointwise problem
		\[
		\Div_{x,z}\!\left(z^{1-2s}\,\Grad_{x,z} U\right)=0\ \text{ in }\Omega,
		\qquad
		-c_{d,s}\lim_{z\to0^+} z^{1-2s}\,\partial_z U=f\ \text{ on }\mathbb T^d\times\{0\}.
		\]
	
	\end{prop}
	Proposition~\ref{prop:CS_weak} is classical; see Caffarelli and Silvestre \cite{CaS2007} for the extension characterization of $(-\Delta)^s$, Roncal and Stinga \cite{RS2016} for the spectral form adapted to the torus, and references therein. 	
\begin{rmk}\label{rem:trace-mean-zero}
	In particular $\Lambda_0 U\in\dot H^s(\mathbb T^d)$ has zero mean,
	$\int_{\mathbb T^d}\Lambda_0 U\,dx=0$, since $\Lambda_0 U=(-\Delta)^{-s}(\vr-\bar\vr)$ and
	$(-\Delta)^{-s}$ preserves the mean-zero class.
\end{rmk}

	\subsubsection*{The energy in extended variables}
	\label{sec:energy-extended}
		In formal setup, we observe $\Grad_x \Phi = \Grad_x \Lambda_0 U$, so the fluid feels the extended field only through its trace, and the momentum equation \eqref{eq:ERm} (equivalently \eqref{eq:ERm-ext}) takes the localized form
	\begin{equation}\label{eq:ERm-2}
		\partial_t(\vr \vu) + \Div_x (\vr \vu \otimes \vu ) + \Grad_x p(\vr) + \vr\, \Grad_x \Lambda_0 U = 0 .
	\end{equation}
	
	The gain of the extension is most transparent at the level of the energy. Using $(-\Delta)^s \Phi = \vr - \bar\vr$ and $\Phi = \Lambda_0 U$, the nonlocal interaction energy is a fractional Dirichlet energy of the potential, and it reads as
	
	\begin{equation*}
		\frac{1}{2} \int_{\mathbb{T}^d} (\vr - \bar\vr)\, \Phi \,dx
		= \frac{1}{2} \big\| \Phi \big\|_{\dot{H}^s(\mathbb{T}^d)}^2 ,
	\end{equation*}
	and, by the isometry \eqref{CS-prop1}, this equals the weighted Dirichlet energy of the extension:
	\begin{equation*}
		\frac{1}{2} \int_{\mathbb{T}^d} (\vr - \bar\vr)\, \Phi \,dx
		= \frac{c_{d,s}}{2} \iint_{\Omega} z^{1-2s}\, |\Grad_{x,z} U|^2 \,dx\,dz .
	\end{equation*}
	Thus the singular, nonlocal interaction energy is realized as a \emph{local} (weighted) energy in the bulk. Accordingly, the total energy of the system reads
	\begin{equation}\label{en_def_ext}
		\mathcal{E}(t)
		= \int_{\mathbb{T}^d} \left( \frac{1}{2} \vr |\vu|^2 + P(\vr) \right) dx
		+ \frac{c_{d,s}}{2} \iint_{\Omega} z^{1-2s}\, |\Grad_{x,z} U(t)|^2 \,dx\,dz ,
	\end{equation}
	with $P(\vr) = \tfrac{a}{\gamma-1}\vr^\gamma$. 
	
	\subsection{Definition of Dissipative Solution}
	Now we are in a position to state the definition
	
	\begin{df}[Dissipative Solution]\label{def:dissipative_weak_sol}
		Let $T>0$ and $\bar{\vr}>0$ be given, and let $\Omega=\mathbb T^d\times(0,\infty)$ with closure
		$\overline\Omega=\mathbb T^d\times[0,\infty)$. Assume the initial data \eqref{id} have finite energy,
		\begin{align}\label{eq:E0}
			\mathcal E_0 := \int_{\mathbb T^d}\Big(\tfrac12\tfrac{|\vm_0|^2}{\vr_0}+P(\vr_0)\Big)dx
			+ \tfrac{c_{d,s}}2\iint_\Omega z^{1-2s}|\Grad_{x,z}U_0|^2\,dx\,dz <\infty .
		\end{align}
		A septuplet $(\vr,\vm,U,\mathcal R_{fluid},\mathcal R_{ext},\mathcal E_{fluid},\mathcal E_{ext})$ consisting of
		\begin{itemize}[leftmargin=*]
			\item the fluid density $\vr\ge0$ and momentum $\vm$, and the extended potential $U$;
			\item the symmetric matrix-valued \emph{turbulent defect measures}
			\[
			\mathcal R_{fluid}\in L^\infty_{\mathrm{weak}\text{-}*}\big(0,T;\mathcal M^+(\mathbb T^d;\mathbb R^{d\times d}_{\mathrm{sym}})\big),
			\qquad
			\mathcal R_{ext}\in L^\infty_{\mathrm{weak}\text{-}*}\big(0,T;\mathcal M(\overline\Omega;\mathbb R^{d\times d}_{\mathrm{sym}})\big);
			\]
			\item the nonnegative scalar \emph{energy defect measures}
			\[
			\mathcal E_{fluid}\in L^\infty_{\mathrm{weak}\text{-}*}\big(0,T;\mathcal M^+(\mathbb T^d)\big),
			\qquad
			\mathcal E_{ext}\in L^\infty_{\mathrm{weak}\text{-}*}\big(0,T;\mathcal M^+(\overline\Omega)\big);
			\]
		\end{itemize}
		is a \emph{dissipative weak solution} to the extended Euler--Riesz system on $[0,T]$ with initial
		data $(\vr_0,\vm_0,U_0)$ if the state variables lie in the regularity class
		\begin{align}\label{regclass}
			0\leq \vr,\,
			\vr\in C_{\mathrm{weak}}([0,T];L^\gamma(\mathbb T^d)),\,
			\vm\in C_{\mathrm{weak}}([0,T];L^{\frac{2\gamma}{\gamma+1}}(\mathbb T^d;\mathbb R^d)),\,
			U\in L^\infty(0,T;\dot H^1(\Omega,z^{1-2s})),
		\end{align}
		and the following hold.
		\begin{enumerate}[leftmargin=*]
			\item \textbf{Continuity equation.} For all $\phi\in C_c^\infty([0,T)\times\mathbb T^d)$, for a.e. $\tau \in (0,T)$ the identity 
			\begin{equation}\label{eq:w-ERc}
				\int_0^{\tau}\!\!\int_{\mathbb T^d}\big(\vr\,\partial_t\phi+\vm\cdot\Grad_x\phi\big)\,dx\,dt
				+\int_{\mathbb T^d}\vr_0\,\phi(0,\cdot)\,dx- \int_{\mathbb T^d}\vr(\tau)\,\phi(\tau,\cdot)\,dx=0.
			\end{equation}
		\item \textbf{Momentum equation.} For all $\pmb\psi\in C^\infty([0,T]\times\mathbb T^d;\mathbb R^d)$,
		for a.e.\ $\tau\in(0,T)$,
		\begin{equation}\label{eq:w-ERm}
			\begin{aligned}
				&\int_0^\tau\!\!\int_{\mathbb T^d}\Big(\vm\cdot\partial_t\pmb\psi
				+\big(\tfrac{\vm\otimes\vm}{\vr}+p(\vr)\mathbb I_d\big):\Grad_x\pmb\psi\Big)\,dx\,dt\\
				&\quad -\,c_{d,s}\int_0^\tau\!\!\iint_\Omega \mathbb T_{\rm ER}[U]:\Grad_x\pmb\psi\,dx\,dz\,dt
				+\,\bar\vr\int_0^\tau\!\!\int_{\mathbb T^d}\Lambda_0(U)\,\Div_x\pmb\psi\,dx\,dt\\
				&\quad +\int_0^\tau\!\!\int_{\mathbb T^d}\Grad_x\pmb\psi:d\mathcal R_{fluid}
				-\,c_{d,s}\int_0^\tau\!\!\iint_{\overline\Omega}\Grad_x\pmb\psi:d\mathcal R_{ext}\\
				&\quad +\int_{\mathbb T^d}\vr_0\vu_0\cdot\pmb\psi(0,\cdot)\,dx
				-\int_{\mathbb T^d}\vm(\tau)\cdot\pmb\psi(\tau,\cdot)\,dx
				= 0 .
			\end{aligned}
		\end{equation}
			with $\mathbb T_{\rm ER}[U]=z^{1-2s}\big(\Grad_xU\otimes\Grad_xU-\tfrac12|\Grad_{x,z}U|^2\mathbb I_d\big)$.
			\item \textbf{Extension PDE.} For a.e.\ $t\in(0,T)$,
			\begin{equation}\label{eq:w-ext}
				\iint_\Omega z^{1-2s}\Grad_{x,z}U\cdot\Grad_{x,z}\Psi\,dx\,dz
				=c_{d,s}^{-1}\int_{\mathbb T^d}(\vr-\bar\vr)\Lambda_0(\Psi)\,dx
				\qquad\forall\,\Psi\in\dot C_c^\infty(\overline\Omega).
			\end{equation}
			\item \textbf{Energy inequality.} For a.e.\ $\tau\in(0,T)$,
			\begin{equation}\label{eq:ei}
				\begin{aligned}
					&\int_{\mathbb T^d}\Big(\tfrac12\tfrac{|\vm|^2}{\vr}+P(\vr)\Big)(\tau)\,dx
					+\tfrac{c_{d,s}}2\iint_\Omega z^{1-2s}|\Grad_{x,z}U(\tau)|^2\,dx\,dz
					+\int_{\mathbb T^d}d\mathcal E_{fluid}(\tau)+\iint_{\overline\Omega}d\mathcal E_{ext}(\tau)\\
					&\le \int_{\mathbb T^d}\Big(\tfrac12\tfrac{|\vm_0|^2}{\vr_0}+P(\vr_0)\Big)\,dx
					+\tfrac{c_{d,s}}2\iint_\Omega z^{1-2s}|\Grad_{x,z}U_0|^2\,dx\,dz .
				\end{aligned}
			\end{equation}
		\item \textbf{Defect measure compatibility.} 
		\begin{enumerate}
			\item[\rm(a)] \emph{(Fluid.)} For a.e.\ $t\in(0,T)$, there exist constants $\lambda_1,\lambda_2>0$, independent of $t$, with
			\begin{equation}\label{dm-compf}
				\lambda_1 \operatorname{tr}\big(\mathcal R_{fluid}(t)\big) \ \le\ \mathcal E_{fluid}(t) \ \le\ \lambda_2 \operatorname{tr}\big(\mathcal R_{fluid}(t)\big)
				\qquad\text{on } \mathbb{T}^d .
			\end{equation}
		\item[\rm(b)] \emph{(Extension.)} For a.e.\ $t\in(0,T)$, $\mathcal R_{ext}(t)\ll\mathcal E_{ext}(t)$ entrywise, with symmetric
		Radon--Nikod\'ym density\footnote{For a finite nonnegative Radon measure $\mu$ on $\overline\Omega$ and $1\le p<\infty$, the space
				$L^p(\overline\Omega,\mu;\mathbb R^{d\times d}_{\mathrm{sym}})$ carries the norm
				$\|\mathbb D\|_{L^p(\overline\Omega,\mu)}:=\Big(\int_{\overline\Omega}\|\mathbb D(x,z)\|_{\mathrm{op}}^{\,p}\,d\mu\Big)^{1/p}$;
				for $p=\infty$,
				$\|\mathbb D\|_{L^\infty(\overline\Omega,\mu)}:=\mu\text{-}\hspace{-0.15em}\operatorname*{ess\,sup}_{(x,z)\in\overline\Omega}\|\mathbb D(x,z)\|_{\mathrm{op}}$,
				the essential supremum of the spectral operator norm.}
		\begin{equation}\label{dm-compE}
		\mathbb D(t):=\frac{d\big(c_{d,s}\,\mathcal R_{ext}(t)\big)}{d\,\mathcal E_{ext}(t)}
		\in L^\infty\big(\overline\Omega,\mathcal E_{ext}(t);\mathbb R^{d\times d}_{\mathrm{sym}}\big),
		\qquad \|\mathbb D(t)\|_{L^\infty( \Ov{\Om},\mathcal E_{ext}(t))}\le 1 .
		\end{equation}
	\end{enumerate}
		Condition \eqref{dm-compE} states that the extension defect is dominated by the energy defect it carries; equivalently,
		\begin{equation}\label{defm-do}
		-\,\mathcal E_{ext} \ \le\ \big(c_{d,s}\,\mathcal R_{ext}\big):(\xi\otimes\xi) \ \le\ \mathcal E_{ext}
		\, \text{as scalar measures, for every } \xi\in\mathbb R^d \text{ with } |\xi|=1 .
		\end{equation}
		\end{enumerate}

	\end{df}

	\begin{rmk}\label{rem:no-sign}
		Compared with \cite{BeFH2020d}, the extension defect $\mathcal R_{ext}$ lacks a
		sign: it is only symmetric, not positive semidefinite. What replaces the sign is the
		absolute continuity \eqref{dm-compE} along with the explict structure that we discuss in Lemma \ref{lem:Rext-structure} in Section \ref{sec:stability-defect}: $\mathcal R_{ext}$ is dominated by the energy defect
		$\mathcal E_{ext}$ through a density of operator norm at most one. 
	\end{rmk}

	\subsection{Main Result}\label{sec:main-result}
	
We can now state the main result precisely.
	
	\begin{thm}[Weak--strong uniqueness]\label{thm:ws}
		Let $d=2,3$, $\gamma>1$, $s\in(0,1)$ (equivalently, $\beta \in (0,2)$) and $T>0$. Let
		\[
		\big(\vr,\ \vm,\ U,\ \mathcal R_{fluid},\ \mathcal R_{ext},\ \mathcal E_{fluid},\ \mathcal E_{ext}\big)
		\]
		be a dissipative weak solution of the extended Euler--Riesz system on $[0,T]$ in the sense of Definition~\ref{def:dissipative_weak_sol}, and let $(r,\vv)$ solve \eqref{eq:ERc}--\eqref{eq:ERm} classically, with
		\begin{equation}\label{strong-class}
			r,\ \vv\in C^1\big([0,T]\times\mathbb T^d\big),
			\qquad
			\underline r:=\min_{[0,T]\times\mathbb T^d} r>0 ,
		\end{equation}
		and $V$ the extension of $r-\bar r$( i.e. $V= \mathfrak{E}[r-\Ov{r}]$). If initial data for a weak and the strong solution coincides 
		\[
		\vr(0,\cdot)=r(0,\cdot),\qquad \vm(0,\cdot)=(r\vv)(0,\cdot),
		\]
		then $\bar\vr=\bar r$ and, for a.e. $\tau \in (0,T)$ the following folds 
		\[
		\vr(\tau)=r(\tau),\qquad \vm(\tau)=r(\tau)\vv(\tau),\qquad U(\tau)=V(\tau),
		\qquad
		\mathcal E_{fluid}=\mathcal E_{ext}=\mathcal R_{fluid}=\mathcal R_{ext}=0.
		\]
		In particular the dissipative solution coincides with the strong one, and the energy is conserved on $[0,T]$.
	\end{thm}
	\begin{rmk}\label{compref}
		It is worth noting that in the repulsive setting Alves et al.~\cite{AlGT2026} obtain
		weak--strong uniqueness for $\gamma\ge 2-\tfrac{\beta}{d}$ over the full Riesz range
		$\beta\in(0,d)$, coupling the pressure exponent to the singularity order. In our setting the
		two decouple: the result holds for every $\gamma>1$ and every $\beta=2s\in(0,2)$,
		independently. This decoupling is the point of the extension formulation --- the nonlocal
		force is controlled through the extension energy rather than the pressure --- while the
		restriction $\beta\in(0,2)$ is its price: the reformulation of the force as the divergence of
		the first-order stress \eqref{T} relies on the Caffarelli--Silvestre extension being of first
		order. For $d=2$ the ranges coincide ($\beta\in(0,2)=(0,d)$) and the improvement is entirely
		in $\gamma$; for $d\ge3$ the restriction $\beta<2$ excludes part of the \cite{AlGT2026} range.
	\end{rmk}
	\begin{rmk}
		We donot put any separate regularity assumption for the extended variable of strong solution.  As noted in a later proposition (Proposition~\ref{prop:CS_regularity}), we deduce certain regularity of $\mathfrak{E}(r-\Ov{r})$. In the proof of the Theorem \ref{thm:ws} in Section \ref{sec:relative}, we note that the choice of class cann be slightly relaxed by taking 
		\begin{align*}
	&	0 <	\underbar{r}\leq r \in L^\infty\lr{(0,T)\times \T^d}, \, \Gradx r \in L^1(0,T; L^\infty(\T^d)) \\ &\text{ and }  \vv \in L^\infty\lr{(0,T)\times \T^d}, \, \Gradx \vv \in L^1(0,T; L^\infty(\T^d)) 
		\end{align*}
			Since both continuity equations conserve mass, so $\bar\vr$ and $\bar r$ are constant in time, and the shared initial datum gives $\bar\vr=\bar r$ throughout. Writing $\bar\rho$ for the common value, $U$ and $V$ are the extensions of $\vr-\bar\rho$ and $r-\bar\rho$, so that $U-V$ is the extension of $\vr-r$.
	\end{rmk} 
	
	\medskip

The class \eqref{strong-class} is non-empty on a short time interval; for our system this
follows from the local existence theory of Choi and Jeong~\cite{CJ2022}. Their interaction
kernel is $|x|^{-\alpha}$, normalized so that $c_{\alpha,d}\,|x|^{-\alpha}\ast\,\cdot
=(-\Delta)^{(\alpha-d)/2}$. Matching this with our kernel gives $\alpha=d-2s$, so
$\alpha\in(d-2,d)$ for $s\in(0,1)$, and our repulsive interaction is their case $c_K<0$. We
always take a genuine pressure, $\gamma>1$, so only the pressure case of \cite{CJ2022} is
relevant; the pressureless and attractive regimes treated there are not needed. Translating
their result into our notation:

\begin{prop}[Local existence, {\cite[Thm.~4.1 and Rmk.~4.3]{CJ2022}}]\label{prop:local-existence}
	Let $m>\tfrac d2+1$ and $s\in(0,1)$. Given $r_0-\bar r\in H^m(\mathbb T^d)$ and
	$\vv_0\in H^m(\mathbb T^d;\mathbb R^d)$ with $\min_{\mathbb T^d} r_0>0$, there exist $T^*>0$
	and a unique solution of \eqref{eq:ERc}--\eqref{eq:ERm} with
	\[
	r-\bar r,\ \vv\ \in\ C\big([0,T^*);H^m\big)\cap C^1\big([0,T^*);H^{m-1}\big),
	\qquad \min_{[0,T]\times\mathbb T^d} r>0 \ \text{ for every } T<T^*.
	\]
\end{prop}

\begin{rmk}[Strict positivity of the density in \eqref{strong-class}]\label{rem:strong-class}
	If $\vv\in L^1(0,T;W^{1,\infty}(\T^d))$, the continuity equation
	$\partial_t r+\Divx(r\vv)=0$ propagates positivity along the flow of $\vv$ (DiPerna--Lions
	\cite{DiL1989}):
	\[
	\min_{\mathbb T^d} r(t)\ \ge\ \Big(\min_{\mathbb T^d} r_0\Big)
	\exp\!\Big(-\!\int_0^t\|\Divx\vv\|_{L^\infty}\,d\tau\Big)>0 ,
	\]
	and $r$ stays bounded above. Thus $\min_{\T^d} r_0>0$ already forces
	$\underline r:=\inf_{[0,T]\times\T^d} r>0$, so the uniform lower bound in \eqref{strong-class}
	is automatic.
\end{rmk}

\begin{rmk}[Applicability of Proposition~\ref{prop:local-existence}]\label{rem:CJ-scope}
	Two features of \cite{CJ2022} matter here. First, the pressure theory in general carries the
	restriction $1<\gamma\le\tfrac53$, but this is lifted whenever the density is bounded below or
	the potential is repulsive \cite[Rmk.~4.3]{CJ2022}; on $\mathbb T^d$ with $\min r_0>0$ both hold,
	so no upper bound on $\gamma$ is imposed. Second, since $m>\tfrac d2+1$, Sobolev embedding gives
	$r,\vv\in C^1([0,T]\times\mathbb T^d)$ with $\Gradx\vv\in L^\infty((0,T)\times\T^d)$; together with
	Remark~\ref{rem:strong-class} this places $(r,\vv)$ in \eqref{strong-class} for every $T<T^*$.
	Hence Theorem~\ref{thm:ws} applies on $[0,T^*)$: any dissipative weak solution with the same data
	coincides there with $(r,\vv)$.
\end{rmk}

\section{Soundness of the Definition}
\label{sec:justification}

Definition~\ref{def:dissipative_weak_sol} contains three ingredients that do not appear in the classical formulation: the stress tensor $\mathbb T[U]$, through which the nonlocal force is written; the defect measures, which record what is lost in the nonlinear and nonlocal terms; and the compatibility conditions binding the two. This section accounts for all three. We first derive $\mathbb T[U]$ and the identity \eqref{T} by direct computation. We then show that the definition is consistent: any reasonable approximation scheme produces, in the limit, a solution in the sense of Definition~\ref{def:dissipative_weak_sol}, with the defect measures arising precisely as the weak limits the nonlinear terms fail to attain. Finally, we verify that the defects are genuine: a dissipative solution that happens to lie in the strong class \eqref{strong-class} has all four defect measures equal to zero, and is therefore a classical solution. The notion thus extends the classical one without enlarging it among smooth fields.

\subsection{Structure of the tensor $\mathbb{T}_{\rm ER}[\cdot ]$ }
\label{sec:stress}

We recover an identity similar to the stress identity in $\Omega=\mathbb{T}^d\times\mathbb{R}_+$, where the operator is local. Write $(x,z)\in\mathbb{R}^{d+1}$ with $x_{d+1}=z$, and let indices $i,j$ range over $\{1,\dots,d+1\}$, so that $\Grad_{x,z}=(\partial_{x_1},\dots,\partial_{x_d},\partial_z)$.

\subsubsection*{Formal computation}
\label{sec:stress-formal}

We recall the full extended $(d+1)\times(d+1)$ tensor introduced in \eqref{Tfull}:
\begin{equation*}
	\mathbb{T}_{\mathrm{ext}}[U] = z^{1-2s} \left( \Grad_{x,z} U \otimes \Grad_{x,z} U - \tfrac12 |\Grad_{x,z} U|^2\, \mathbb{I}_{d+1} \right),
\end{equation*}
whose entries are
\begin{equation}\label{Tfull-entries}
	\big(\mathbb{T}_{\mathrm{ext}}[U]\big)_{ij}
	= z^{1-2s}\, \partial_{x_i} U\, \partial_{x_j} U
	- \tfrac12\, z^{1-2s}\, |\Grad_{x,z} U|^2\, \delta_{ij},
	\qquad i,j \in \{1,\dots,d+1\}.
\end{equation}
The goal is to compute $\Div_{x,z}\mathbb{T}_{\mathrm{ext}}[U]$, whose $j$-th component is $\sum_{i} \partial_{x_i} \big(\mathbb{T}_{\mathrm{ext}}[U]\big)_{ij}$. A direct computation shows 
\begin{align}\label{div-dyadic}
	\sum_{i} \partial_{x_i}\!\left( z^{1-2s}\, \partial_{x_i} U\, \partial_{x_j} U \right)
	&= \left( \sum_{i} \partial_{x_i}\!\left( z^{1-2s}\, \partial_{x_i} U \right) \right) \partial_{x_j} U
	+ \sum_{i} z^{1-2s}\, \partial_{x_i} U\, \partial_{x_i} \partial_{x_j} U \nonumber \\
	&= \Div_{x,z}\!\left( z^{1-2s}\, \Grad_{x,z} U \right) \partial_{x_j} U
	+ \tfrac12\, z^{1-2s}\, \partial_{x_j}\!\left( \sum_{i} \big(\partial_{x_i} U\big)^2 \right) \nonumber \\
	&= \tfrac12\, z^{1-2s}\, \partial_{x_j} |\Grad_{x,z} U|^2 .
\end{align}
In the second line we used the symmetry of second derivatives to write $\sum_i \partial_{x_i} U\, \partial_{x_i}\partial_{x_j} U = \tfrac12 \partial_{x_j} \sum_i (\partial_{x_i} U)^2$, and in the third line the bulk equation $\Div_{x,z}\big( z^{1-2s} \Grad_{x,z} U \big) = 0$, which annihilates the first term.
On the other hand,  the Kronecker delta collapses the sum to the single index $i=j$, and the product rule gives
\begin{equation}\label{div-iso}
	\sum_i \partial_{x_i}\!\left( -\tfrac12\, z^{1-2s}\, |\Grad_{x,z} U|^2\, \delta_{ij} \right)
	= -\tfrac12\, \partial_{x_j}\!\left( z^{1-2s} \right) |\Grad_{x,z} U|^2
	- \tfrac12\, z^{1-2s}\, \partial_{x_j} |\Grad_{x,z} U|^2 .
\end{equation}
Thus, adding \eqref{div-dyadic} and \eqref{div-iso}, we obtain
\begin{equation}\label{eq:div-j}
	\big( \Div_{x,z}\mathbb{T}_{\mathrm{ext}}[U] \big)_j
	= -\tfrac12\, \partial_{x_j}\!\left( z^{1-2s} \right) |\Grad_{x,z} U|^2 .
\end{equation}
The weight depends on $z = x_{d+1}$ alone, so $\partial_{x_j}(z^{1-2s}) = 0$ for $j \le d$, while
\[
\partial_{x_{d+1}}\!\left( z^{1-2s} \right) = \partial_z\!\left( z^{1-2s} \right) = (1-2s)\, z^{-2s}.
\]
Hence the divergence has no horizontal component, and \eqref{eq:div-j} reduces to
\begin{equation}\label{div-Tfull}
	\Div_{x,z}\mathbb{T}_{\mathrm{ext}}[U]
	= \begin{pmatrix} 0 \\[2pt] -\dfrac{1-2s}{2}\, z^{-2s}\, |\Grad_{x,z} U|^2 \end{pmatrix}.
\end{equation}
The horizontal component of the divergence vanishes identically; only a vertical part survives. Since the field equations \eqref{eq:ERm} is confined only in $\T^d$, the vertical component \eqref{div-Tfull} will be eliminated when we choose of test function in depending only on $(t,x)$

\subsubsection*{Weak identity and horizontal projection}
\label{sec:stress-weak}

Testing \eqref{Tfull} against a full field $\pmb{\Psi}=(\pmb{\Psi}_x,\Psi_z)^T$ and applying Green's theorem on $\Omega$, it holds
\begin{equation}\label{green}
	\iint_{\Omega} \mathbb{T}_{\mathrm{ext}}[U] : \Grad_{x,z}\pmb{\Psi}\,dx\,dz
	= \int_{\mathbb{T}^d} \lim_{z\to0}\big(\mathbb{T}_{\mathrm{ext}}[U]\,\mathbf{n}\big)\cdot\pmb{\Psi}\,dx
	- \iint_{\Omega} \big(\Div_{x,z}\mathbb{T}_{\mathrm{ext}}[U]\big)\cdot\pmb{\Psi}\,dx\,dz,
\end{equation}
where $\mathbf{n}=(0,\dots,0,-1)^T$ is the outward normal to $\Omega$ at $z=0$; the far-field decay of $U$ kills the contribution at $z=\infty$.

\medskip
\noindent\emph{The boundary term.} Splitting into horizontal and vertical components, and using $|\Grad_{x,z}U|^2 = |\Grad_x U|^2 + (\partial_z U)^2$, we have
\begin{equation}\label{Tn}
	\mathbb{T}_{\mathrm{ext}}[U]\,\mathbf{n}
	= \begin{pmatrix}
		-\, z^{1-2s}\, \partial_z U\, \Grad_x U \\[4pt]
		-\, z^{1-2s} \Big( (\partial_z U)^2 - \tfrac12 |\Grad_{x,z} U|^2 \Big)
	\end{pmatrix}
	= \begin{pmatrix}
		-\, z^{1-2s}\, \partial_z U\, \Grad_x U \\[4pt]
		-\, \tfrac12\, z^{1-2s} \Big( (\partial_z U)^2 - |\Grad_x U|^2 \Big)
	\end{pmatrix}.
\end{equation}
The horizontal component is precisely the object controlled by the extension: by the trace condition
\[
-c_{d,s} \lim_{z\to0} z^{1-2s}\, \partial_z U = \vr - \bar\vr,
\]
we obtain
\begin{equation*}
	\lim_{z\to0} \Big( -\, z^{1-2s}\, \partial_z U\, \Grad_x U \Big)
	= c_{d,s}^{-1}\, (\vr - \bar\vr)\, \Grad_x U|_{z=0} .
\end{equation*}

\medskip
\noindent\emph{The extension term.} By \eqref{div-Tfull} the divergence is purely vertical, so it pairs only with $\Psi_z$:
\begin{equation}\label{bulk-term}
	- \iint_{\Omega} \big( \Div_{x,z}\mathbb{T}_{\mathrm{ext}}[U] \big)\cdot\pmb{\Psi}\,dx\,dz
	= \iint_{\Omega} \tfrac{1-2s}{2}\, z^{-2s}\, |\Grad_{x,z} U|^2\, \Psi_z \,dx\,dz .
\end{equation}

\medskip
\noindent Substituting \eqref{Tn}--\eqref{bulk-term} into \eqref{green} gives the general identity
\begin{equation}\label{w-Tfull}
	\begin{aligned}
		\iint_{\Omega} \mathbb{T}_{\mathrm{ext}}[U] : \Grad_{x,z}\pmb{\Psi}\,dx\,dz
		&= c_{d,s}^{-1}\int_{\mathbb{T}^d} (\vr-\bar\vr)\,\Grad_x U|_{z=0}\cdot\pmb{\Psi}_x\,dx \\
		&\quad - \int_{\mathbb{T}^d} \lim_{z\to0}\Big[ \tfrac12\, z^{1-2s} \big((\partial_z U)^2 - |\Grad_x U|^2\big)\, \Psi_z \Big]dx \\
		&\quad + \iint_{\Omega} \tfrac{1-2s}{2}\, z^{-2s}\,|\Grad_{x,z}U|^2\, \Psi_z \,dx\,dz .
	\end{aligned}
\end{equation}
We note that, in the RHS of \eqref{w-Tfull} only the first term is well-defined at finite energy. The remaining two are driven entirely by the vertical component $\Psi_z$: the second term in \eqref{w-Tfull} requires the boundary traces of $z^{1-2s}(\partial_z U)^2$ and $z^{1-2s}|\Grad_x U|^2$, and the third term is of $z^{-2s}|\Grad_{x,z}U|^2$. These are not controlled by the energy norm $\dot H^1(\Omega, z^{1-2s})$, which bounds only $z^{1-2s}|\Grad_{x,z}U|^2$ in $L^1(\Omega)$.\par 
The momentum equation is confined only on $\mathbb{T}^d$, so the test velocity is horizontal. This imposes that the test field to have no vertical component, more precisely
\begin{equation*}
	\pmb{\Psi}_{\mathrm{ext}}(t,x,z) = \big(\pmb{\psi}(t,x),\, 0\big)^T,
	\qquad\text{i.e.}\qquad \Psi_z \equiv 0 \text{ and } \Grad_{x,z}\pmb{\Psi}_{\mathrm{ext}}
	= \begin{pmatrix} \Grad_x \pmb{\psi} & 0 \\ 0 & 0 \end{pmatrix}.
\end{equation*}
So in in \eqref{w-Tfull} only the upper-left $d\times d$ block of $\mathbb{T}_{\mathrm{ext}}[U]$ survives; the vertical row and column, which carries the ill-defined terms , are annihilated. From \eqref{Tfull-entries} for $i,j \le d$, we have the following identity
\begin{equation*}
	\mathbb{T}_{\mathrm{ext}}[U] : \Grad_{x,z}\pmb{\Psi}_{\mathrm{ext}} = \mathbb{T}_{\rm ER}[U] : \Grad_x \pmb{\psi},
	\qquad
	\mathbb{T}_{\rm ER}[U] := z^{1-2s}\left( \Grad_x U \otimes \Grad_x U - \tfrac12 |\Grad_{x,z}U|^2\, \mathbb{I}_d \right),
\end{equation*}
exactly as defined in \eqref{T}. Hence \eqref{w-Tfull} reduces to the identity
\begin{equation}\label{w-T}
	c_{d,s}\iint_{\Omega} \mathbb{T}_{\rm ER}[U] : \Grad_x\pmb{\psi}\,dx\,dz
	= \int_{\mathbb{T}^d} (\vr-\bar\vr)\,\Grad_x U|_{z=0}\cdot\pmb{\psi}\,dx .
\end{equation}

\subsection{Justification of Definition via Consistent Approximation}
\label{sec:stability}

\emph{Sequential stability} means that the solution class is closed under weak limits of approximations. Suppose a sequence obeys the equations up to a vanishing error, and satisfies the energy inequality uniformly. Then, along a subsequence, it converges to a solution of the same class.

This is what makes the weak formulation useful rather than arbitrary. The class is not invented to fit a definition; it is the class that arises as the limit of objects one can actually construct --- vanishing viscosity limits, Galerkin approximations, energy-stable numerical schemes. 

\subsubsection{Approximating sequences and uniform bounds}
\label{sec:stability-bounds}

The defect measures are not ad hoc: they are exactly what survives the limit in an energy-bounded family of approximations. We first make the hypothesis precise.

\begin{df}[Consistent approximating sequence]\label{def:consistent}
	A family $(\vr_n, \vu_n)_{n\in\mathbb{N}}$ of smooth functions on $[0,T]\times\mathbb{T}^d$, with $\vr_n > 0$, is a \emph{consistent approximating sequence} for \eqref{eq:ERc}--\eqref{eq:ERr} if:
	\begin{enumerate}
		\item[(A1)] \textbf{(Bounded initial data.)} The initial data $(\vr_{0,n}, \vu_{0,n})$ satisfy $\vr_{0,n}\to\vr_0$ in $L^\gamma(\mathbb{T}^d)$, $\vr_{0,n}-\Ov{ \vr}_{0,n}\to\vr_0-\Ov{\vr}_0$ in $\dot H^{-s}(\T^d)$, $\vr_{0,n}\vu_{0,n}\to \vm_0$ in $L^{\frac{2\gamma}{\gamma+1}}(\mathbb{T}^d;\mathbb{R}^d)$, and
		\begin{equation}\label{app:id}
			\limsup_{n\to\infty} \mathcal{E}\big[\vr_{0,n},\vu_{0,n}\big] \;\le\; \mathcal{E}_0 := \mathcal{E}\big[\vr_0,\vu_0\big] \;<\; \infty,
			\end{equation}
		where $\displaystyle \mathcal{E}[\vr,\vu] = \int_{\mathbb{T}^d}\big(\tfrac12\vr|\vu|^2 + P(\vr)\big)dx + \tfrac12\|\vr-\bar\vr\|_{\dot H^{-s}(\mathbb{T}^d)}^2$ is the energy \eqref{en_def_ext}.
		\item[(A2)] \textbf{(Energy inequality.)} For all $\tau\in[0,T]$,
		\begin{equation}\label{app_ie}
			\mathcal{E}\big[\vr_n(\tau),\vu_n(\tau)\big] + e_n^{e}(\tau)\;\le\; \mathcal{E}\big[\vr_{0,n},\vu_{0,n}\big].
		\end{equation}
		\item[(A3)] \textbf{(Consistency.)} The pair $(\vr_n,\vu_n)$ solves \eqref{eq:ERc}--\eqref{eq:ERm} up to errors $e_n^{\,c} \in L^1((0,T)\times \T^d), \mathbf{e}_n^{\,m} \in L^1((0,T)\times \T^d;\R^d) $ and $ e_n^{e} \in L^{\infty}(0,T)$ that vanish in the sense of distributions: More precisely, we assume that for all $\phi\in C_c^\infty([0,T)\times\mathbb{T}^d)$, $\pmb\psi\in C_c^\infty([0,T)\times\mathbb{T}^d;\mathbb{R}^d)$ and for all $\tau \in (0,T)$, the following holds:
		\[
		\big\langle e_n^{\,c}, \phi \big\rangle_{L^1_{t,x}, C_{t,x}} \to 0,
		\qquad
		\big\langle \mathbf{e}_n^{\,m}, \pmb\psi \big\rangle_{L^1_{t,x}, C_{t,x}} \to 0, \qquad  e_n^{e}(\tau) \rightarrow 0
		\qquad\text{as } n\to\infty .
		\]
	\end{enumerate}
\end{df}
The goal is to show that if $(\vr_n,\vu_n)$ is a consistent approximation, then we obtain a dissipative solution following the definition as a limiting procedure.\par 

\textbf{{Uniform Bounds.}}
\medskip
Combining (A1) and (A2),
\begin{equation*}
	\sup_{n\in\mathbb{N}}\ \sup_{t\in[0,T]} \left[ \int_{\mathbb{T}^d}\!\Big( \tfrac12 \vr_n|\vu_n|^2 + P(\vr_n) \Big)dx + \tfrac12\big\|\vr_n-\bar\vr\big\|_{\dot H^{-s}(\mathbb{T}^d)}^2 \right] \;\le\; \mathcal{E}_0 + C,
\end{equation*}
the constant $\mathcal{E}_0$ depending only on the limiting data $(\vr_0,\vu_0)$, consistency condition ensures that the term $ e_n^{e}(t) $ is uniformly bounded in $L^\infty(0,T)$ . using $P(\vr)=\tfrac{a}{\gamma-1}\vr^\gamma$ and each term being nonnegative, we obtain
\begin{equation}\label{app:bd1}
	\sup_n \big\| \vr_n \big\|_{L^\infty(0,T; L^\gamma)} +
	\sup_n \big\| \sqrt{\vr_n}\,\vu_n \big\|_{L^\infty(0,T; L^2)} +
	\sup_n \big\| \vr_n - \bar\vr \big\|_{L^\infty(0,T; \dot H^{-s})} \leq C,
\end{equation}
where $C$ is independent of $n$. As a simple consequence we have $ 	\big\| \vr_n\vu_n \big\|_{L^\infty(0,T;L^{\frac{2\gamma}{\gamma+1}}(\mathbb{T}^d))}
\leq C$.

\medskip
The third bound in \eqref{app:bd1} is exactly what makes the extension \eqref{CS-ext} (More precisely, use of Proposition \ref{prop:CS_weak}) available for $\nabla \Phi_n$. For each $n$ and a.e.\ $t$, let $U_n(t)$ be the unique solution of \eqref{CS_weak} with datum $\vr_n(t)-\bar\vr \in \dot H^{-s}(\mathbb{T}^d)$; by the isometry \eqref{CS-prop1},
\begin{equation}\label{app-bd-ext1}
	\frac{c_{d,s}}{2}\iint_{\Omega} z^{1-2s}\big|\Grad_{x,z}U_n(t)\big|^2\,dx\,dz
	= \frac12 \big\| \vr_n(t)-\bar\vr \big\|_{\dot H^{-s}(\mathbb{T}^d)}^2 ,
\end{equation}
so \eqref{app:bd1} transfers verbatim to the extended energy:
\begin{equation}\label{app:ba1-ext}
	\sup_n\ \sup_{t\in[0,T]} \left[ \int_{\mathbb{T}^d}\!\Big( \tfrac12 \vr_n|\vu_n|^2 + P(\vr_n) \Big)dx + \tfrac{c_{d,s}}{2}\iint_{\Omega} z^{1-2s}\big|\Grad_{x,z}U_n\big|^2\,dx\,dz \right] \;\leq C .
\end{equation}

\medskip
\noindent

\textbf{{Convergence.} } From the bound \eqref{app:ba1-ext}, using  Banach--Alaoglu, along a subsequence (not relabelled),
\begin{align}\label{app:conv}
	\begin{split}
			\vr_n &\overset{*}{\rightharpoonup} \vr \qquad \text{ in } L^\infty_{\text{weak-(*)}}\big(0,T;L^\gamma(\mathbb{T}^d)\big), \\
		\vr_n\vu_n &\overset{*}{\rightharpoonup} \mathbf{m} \qquad \text{ in } L^\infty_{\text{weak-(*)}}\big(0,T;L^{\frac{2\gamma}{\gamma+1}}(\mathbb{T}^d;\mathbb{R}^d)\big), \\
		U_n &\overset{*}{\rightharpoonup} U \qquad \text{ in } L^\infty_{\text{weak-(*)}}\big(0,T;\dot H^1(\Omega, z^{1-2s})\big). 
	\end{split}
\end{align}

Let us make two observations :

\begin{itemize}[leftmargin=*]
	\item First, \eqref{app:conv} produces a momentum field $\mathbf m$, not a velocity. We observe that the map
	\[
	[0,\infty)\times\mathbb R^d\ni(\vr,\mathbf m)\longmapsto \frac{|\mathbf m|^2}{\vr}
	\quad\Big(=0 \text{ if }\mathbf m=0,\ =+\infty \text{ if }\vr=0\neq\mathbf m\Big)
	\]
	is convex and lower semicontinuous. Hence, we have
	\[
	\int_{\mathbb T^d}\frac{|\mathbf m|^2}{\vr}\,\dx \ \le\ \liminf_{n\to\infty}\int_{\mathbb T^d}\frac{|\vr_n\vu_n|^2}{\vr_n}\,\dx \ \leq C,
	\]
	so the limit functional is finite. Finiteness forces $\mathbf m$ to vanish wherever $\vr=0$, i.e.\ $\mathbf m\ll\vr$; the Radon--Nikod\'ym density $\vu:=\dfrac{d\mathbf m}{d\vr}$ is then defined $\vr$-a.e.\ . 
	
	\item Second, \eqref{app:conv} gives convergence weak-$*$ in time, whereas
	Definition~\ref{def:dissipative_weak_sol} requires
	\[
	\vr\in C_{\mathrm{weak}}([0,T];L^\gamma(\T^d))
	\quad\text{and}\quad
	\mathbf m\in C_{\mathrm{weak}}([0,T];L^{\frac{2\gamma}{\gamma+1}}(\T^d)).
	\]
	This is recovered from the equations: \eqref{eq:w-ERc} and \eqref{eq:w-ERm} bound $\partial_t\vr_n$ and $\partial_t\mathbf m_n$ uniformly in a negative-order space, so $\displaystyle t\mapsto\int_{\T^d}\vr_n\phi \dx $ and $t\mapsto\int_{\T^d}\mathbf m_n\cdot\pmb\psi$ are equicontinuous; Arzel\`a--Ascoli and density of smooth test functions then place $\vr,\mathbf m$ in the class \eqref{regclass}.
\end{itemize}

\subsubsection{Defect measures}
\label{sec:stability-defect}
Along the subsequence \eqref{app:conv} we recall $\Ov{F(\cdot)}$ for the weak-$*$ limit of $F(\vr_n,\mathbf m_n)$ (§\ref{sec:notation}). The convective ($\vr_n \vu_n \otimes \vu_n$) , pressure / pressure potential ($p(\vr_n)/P(\vr_n)$) and stress terms ($\mathbf{T}[U_n]$) are nonlinear, so they do not commute with $\Ov{\ \cdot\ }$. We work in the conservative variables $(\vr,\mathbf m)$.

\medskip
\noindent\emph{Fluid part.} For each fixed $\xi\in\mathbb R^d$ the scalar map
\begin{equation}\label{conv1}
	(\vr,\mathbf m)\ \longmapsto\ \frac{\mathbf m\otimes\mathbf m}{\vr}:(\xi\otimes\xi) \;=\;
	\begin{cases}
		\dfrac{|\mathbf m\cdot\xi|^2}{\vr}, & \vr>0,\\[4pt]
		0, & \vr=\mathbf m=0,\\[2pt]
		+\infty, & \text{otherwise},
	\end{cases}
\end{equation}
is convex and lower semicontinuous on $\mathbb R\times\mathbb R^d$, and $\vr\mapsto p(\vr)=a\vr^\gamma$, $\vr\mapsto P(\vr)=\tfrac{a}{\gamma-1}\vr^\gamma$ are convex for $\gamma>1$. By weak lower semicontinuity the corresponding gaps are nonnegative, so we define
\begin{equation}\label{fluid-defect}
	\mathcal R_{fluid} := \left( \Ov{\frac{\mathbf m\otimes\mathbf m}{\vr}} - \frac{\mathbf m\otimes\mathbf m}{\vr} \right) + \Big( \Ov{p(\vr)} - p(\vr) \Big)\mathbb I_d
	\ \in\ L^\infty_{\mathrm{weak}\text{-}*}\big(0,T;\mathcal M^+(\mathbb T^d;\mathbb R^{d\times d}_{\mathrm{sym}})\big),
\end{equation}
and 
\begin{equation}\label{fluid-energy-defect}
	\mathcal E_{fluid} := \frac12\left( \Ov{\frac{|\mathbf m|^2}{\vr}} - \frac{|\mathbf m|^2}{\vr} \right) + \Big( \Ov{P(\vr)} - P(\vr) \Big)
	\ \in\ L^\infty_{\mathrm{weak}\text{-}*}\big(0,T;\mathcal M^+(\mathbb T^d)\big).
\end{equation}
Considering $$A:=\Ov{|\mathbf m|^2/\vr}-|\mathbf m|^2/\vr\ge0, \; B:=\Ov{P(\vr)}-P(\vr)\ge0 \text{ and }p=(\gamma-1)P $$ we have 
\[
\operatorname{tr}\mathcal R_{fluid}=A+d(\gamma-1)B,\qquad \mathcal E_{fluid}=\tfrac12 A+B,
\]
and therefore \eqref{dm-compf} holds with $\lambda_1=\min\{\tfrac12,\tfrac1{d(\gamma-1)}\}$, $\lambda_2=\max\{\tfrac12,\tfrac1{d(\gamma-1)}\}$.

\medskip
\noindent\emph{Extension part.} Here the mechanism is entirely different, and it is worth isolating. We recall \eqref{app:conv} that states that $G_n \rightharpoonup G$ weakly-$*$ in $L^\infty\big(0,T;L^2(\Omega;\mathbb R^{d+1})\big)$ where
\begin{equation*}
	G_n := z^{\frac{1-2s}{2}}\,\text{ and } \,\Grad_{x,z}U_n, \qquad G := z^{\frac{1-2s}{2}}\,\Grad_{x,z}U.
\end{equation*}
Hence, we define 
\begin{align}
	\mathcal E_{ext}=\frac{c_{d,s}}{2} \lr{\Ov{|G|^2} - |G|^2}
\end{align}

Since
\[
z^{1-2s}\,\Grad_x U\otimes\Grad_x U = (G\otimes G)_{xx},
\qquad
z^{1-2s}\,|\Grad_{x,z}U|^2 = \operatorname{tr}(G\otimes G),
\]
the tensor $\mathbb{T}_{\rm ER}$ as in \eqref{T} is the image of $G\otimes G$ under a \emph{linear} map:
\begin{equation}\label{T-linear}
	\mathbb T_{\rm ER}[U] = \mathcal L\big( G\otimes G \big),
	\qquad
	\mathcal L : \mathbb R^{(d+1)\times(d+1)}_{\mathrm{sym}} \to \mathbb R^{d\times d}_{\mathrm{sym}},
	\quad
	\mathcal L(A) := A_{xx} - \tfrac12 (\operatorname{tr}A)\, \mathbb I_d .
\end{equation}
Next we consider $G \mapsto G\otimes G$, acting on a sequence that converges weakly-$*$ in a \emph{Hilbert} space. For every $\xi\in\mathbb R^{d+1}$,
\[
(\xi\otimes\xi):(G\otimes G)=|\,\xi\cdot G\,|^2,
\]
and the map $G\mapsto |\xi\cdot G|^2$ is convex. Hence weak lower semicontinuity yields
\[
\Ov{|\,\xi\cdot G\,|^2}-|\,\xi\cdot G\,|^2\ge0,
\]
or equivalently,
\begin{equation}\label{Rfull}
	\mathfrak R := c_{d,s}\Big( \Ov{G\otimes G} - G\otimes G \Big)
	\ \in\ \ L^\infty_{\mathrm{weak}\text{-}*}\big(0,T;\mathcal M^+(\Ov{\Omega};\mathbb R^{{d+1}\times {d+1}}_{\mathrm{sym}})\big)
\end{equation}
Thus $\mathfrak R$ is a symmetric positive semidefinite matrix-valued measure. 

Writing $\mathfrak R = \left(\begin{smallmatrix}\mathfrak R_{xx} & \mathfrak R_{xz}\\ \mathfrak R_{zx} & \mathfrak R_{zz}\end{smallmatrix}\right)$, the block $\mathfrak R_{xx}\in\mathcal M^+(\overline\Omega;\mathbb R^{d\times d}_{\mathrm{sym}})$ and the scalar $\mathfrak R_{zz}\ge0$. The bulk energy defect is half the trace,
\begin{equation}\label{Eext}
	\mathcal E_{ext} := \tfrac12 \operatorname{tr}\mathfrak R = \tfrac12\big( \operatorname{tr}\mathfrak R_{xx} + \mathfrak R_{zz} \big) \ \ L^\infty_{\mathrm{weak}\text{-}*}\big(0,T;\mathcal M^+(\Ov{\Omega} )\big)
\end{equation}
and, since $\mathcal L$ is linear and therefore commutes with weak-$*$ limits,
\begin{equation}\label{Rext}
	c_{d,s}\,\mathcal R_{ext} := c_{d,s}\Big( \Ov{\mathbb T_{\rm ER}[U]} - \mathbb T_{\rm ER}[U] \Big) = \mathcal L(\mathfrak R) = \mathfrak R_{xx} - \mathcal E_{ext}\,\mathbb I_d .
\end{equation}
The map $\mathcal L$ is linear but \emph{not} positivity preserving, so $\mathcal R_{ext}$ is in general a signed measure. 

\subsubsection*{Structure of the extension defect}
\label{sec:stability-defect-ext}

Identity \eqref{Rext} exhibits $c_{d,s}\mathcal R_{ext}$ as the difference of two \emph{nonnegative} matrix-valued measures
\begin{equation}\label{Rext-str}
	c_{d,s}\,\mathcal R_{ext} = \underbrace{\mathfrak R_{xx}}_{\succeq\,0} \;-\; \underbrace{\mathcal E_{ext}\,\mathbb I_d}_{\succeq\,0},
	\qquad
	\mathcal E_{ext} = \tfrac12\operatorname{tr}\mathfrak R ,
\end{equation}
the entire indefiniteness being carried by the second term, we use the notation introcuded in Section \ref{sec:notation}. The following lemma collects what survives the loss of sign; part (iv) is what the relative energy estimate consumes.

\begin{lemma}[Structure of the defect measure]\label{lem:Rext-structure}
	Let $\mathfrak R \in L^\infty_{\mathrm{weak}\text{-}*}\big(0,T;\mathcal M^+(\overline\Omega;\mathbb R^{(d+1)\times(d+1)}_{\mathrm{sym}})\big)$, and let $\mathcal E_{ext},\mathcal R_{ext}$ be given by \eqref{Eext}, \eqref{Rext}. Then, for a.e.\ $t\in(0,T)$, the following hold as measures on $\overline\Omega$:
	\begin{enumerate}[leftmargin=*]
		\item[\rm(i)] \emph{(Trace bound.)} $\ 0 \ \le\ \operatorname{tr}\mathfrak R_{xx} \ \le\ \operatorname{tr}\mathfrak R \;=\; 2\,\mathcal E_{ext}$.
		\item[\rm(ii)] \emph{(Sharp two-sided bound.)} For every $\xi \in \mathbb R^d$ with $|\xi|=1$,
		\begin{equation}\label{defm-do-2}
			-\,\mathcal E_{ext} \ \le\ \big(c_{d,s}\,\mathcal R_{ext}\big) : (\xi\otimes\xi) \ \le\ \mathcal E_{ext} ,
		\end{equation}
		and both bounds are attained.
		\item[\rm(iii)] \emph{(Trace identity.)} $\ \operatorname{tr}\big(c_{d,s}\mathcal R_{ext}\big) = (2-d)\,\mathcal E_{ext} - {\mathfrak R_{zz}}$; in particular this is $\le 0$ for $d\ge2$.
		\item[\rm(iv)] \emph{(Absolute continuity and density bound.)} For a.e.\ $t\in(0,T)$,
		$\mathcal R_{ext}(t)\ll\mathcal E_{ext}(t)$ entrywise, so there is a symmetric density
		$\mathbb D(t)\in L^1\big(\overline\Omega,\mathcal E_{ext}(t);\mathbb R^{d\times d}_{\mathrm{sym}}\big)$
		with $$\mathbb D(t)\,\xi\cdot\xi=\dfrac{d\big((c_{d,s}\mathcal R_{ext}(t)):(\xi\otimes\xi)\big)}{d\,\mathcal E_{ext}(t)}$$
		for every $\xi\in\mathbb R^d$. \\ 
		The bound (ii) upgrades this to
		$\mathbb D(t)\in L^\infty\big(\overline\Omega,\mathcal E_{ext}(t);\mathbb R^{d\times d}_{\mathrm{sym}}\big)$
		with $\|\mathbb D(t)\|_{L^\infty(\overline\Omega,\,\mathcal E_{ext}(t))}\le1$.
		
		\item[\rm(v)] \emph{(Testing against a smooth field.)} For every $\mathbf v$ with $\Grad_x\mathbf v \in L^1(0,T; L^\infty(\mathbb T^d;\mathbb R^{d\times d}))$,
		\begin{equation}\label{Rext-test}
			\Big| \int_0^T \iint_{\overline\Omega} \Grad_x\mathbf v : \big(c_{d,s}\,\mathcal R_{ext}\big) \dt \Big|
			\ \le\ \int_0^T \Big( 2\,\|\Grad_x \mathbf v\|_{L^\infty_x} + \|\Div_x\mathbf v\|_{L^\infty_x} \Big)\, \Big(\iint_{\overline\Omega} d\mathcal E_{ext}\Big) \dt .
		\end{equation}
	
	\end{enumerate}
\end{lemma}

\begin{pf}
	Throughout, $\xi\in\mathbb R^d$ is a unit vector and $\tilde\xi := (\xi,0)\in\mathbb R^{d+1}$, so that
	\begin{equation}\label{Rx-block}
		\mathfrak R_{xx} : (\xi\otimes\xi) \;=\; \mathfrak R : (\tilde\xi\otimes\tilde\xi) \ \ge\ 0
	\end{equation}
	by positivity of $\mathfrak R$; in particular $\mathfrak R_{xx}\succeq0$.
	
	\begin{itemize}[leftmargin=*]
		\item[-] \emph{(i).} Choosing for $\xi$ the coordinate directions in \eqref{Rx-block} and summing gives $\operatorname{tr}\mathfrak R_{xx} \ge 0$. Moreover $\operatorname{tr}\mathfrak R = \operatorname{tr}\mathfrak R_{xx} + \mathfrak R_{zz}$ with $\mathfrak R_{zz} = \mathfrak R : (e_z\otimes e_z) \ge 0$, so $\operatorname{tr}\mathfrak R_{xx}\le\operatorname{tr}\mathfrak R = 2\mathcal E_{ext}$ by \eqref{Eext}.
		\item[-] \emph{(ii).}  At first we observe that
		\[
		\big(c_{d,s}\mathcal R_{ext}\big) : (\xi\otimes\xi) \;=\; \mathfrak R_{xx} : (\xi\otimes\xi) \;-\; \mathcal E_{ext} .
		\]
		The lower bound is \eqref{Rx-block}. For the upper bound, a positive semidefinite matrix satisfies $\mathfrak R_{xx}:(\xi\otimes\xi) \le \operatorname{tr}\mathfrak R_{xx}\le 2\mathcal E_{ext}$. 
		\item[-]\emph{(iii).} Take traces in \eqref{Rext-str} and use $\operatorname{tr}\mathfrak R_{xx} = 2\mathcal E_{ext} - \mathfrak R_{zz}$ from (i). Since $\mathfrak R_{zz}=\mathfrak R:(e_{z}\otimes e_{z})\ge0$ and
		$\mathcal E_{ext}\ge0$, both terms are nonpositive for $d\ge2$.
\item[-]\emph{(iv).} Let $E\subset\overline\Omega$ be Borel with $\mathcal E_{ext}(E)=0$, and for $\xi\in\R^d$
set $\mu_\xi:=(c_{d,s}\mathcal R_{ext}):(\xi\otimes\xi)$. By \eqref{Rext-str},
$\mu_\xi=\mathfrak R_{xx}:(\xi\otimes\xi)-\mathcal E_{ext}|\xi|^2$. Since $\mathfrak R_{xx}\succeq0$ and
$\operatorname{tr}\mathfrak R_{xx}\le2\mathcal E_{ext}$ by~(i), each $\mu_\xi$ is a difference of two
nonnegative measures absolutely continuous w.r.t.\ $\mathcal E_{ext}$, hence $\mu_\xi\ll\mathcal E_{ext}$;
in particular $\mu_\xi(E)=0$ for every $\xi$. Taking $\xi=e_i$ and $\xi=(e_i+e_j)/\sqrt2$ and using
the symmetry of $\mathcal R_{ext}$,
\[
(\mathcal R_{ext})_{ii}=\tfrac1{c_{d,s}}\mu_{e_i},\qquad
(\mathcal R_{ext})_{ij}=\tfrac1{c_{d,s}}\big[\mu_{(e_i+e_j)/\sqrt2}-\tfrac12\mu_{e_i}-\tfrac12\mu_{e_j}\big]
\quad(i\ne j),
\]
so $(\mathcal R_{ext})_{ij}(E)=0$ for every entry; thus $\mathcal R_{ext}\ll\mathcal E_{ext}$ entrywise.
Radon--Nikod\'ym then yields a symmetric density
$\mathbb D:=d(c_{d,s}\mathcal R_{ext})/d\mathcal E_{ext}$, a priori in
$L^1(\overline\Omega,\mathcal E_{ext};\mathbb R^{d\times d}_{\mathrm{sym}})$, with
$\mathbb D\,\xi\cdot\xi=d\mu_\xi/d\mathcal E_{ext}$ $\,\mathcal E_{ext}$-a.e.

To upgrade the bound, write $\mathbb A:=d\mathfrak R_{xx}/d\mathcal E_{ext}$; by \eqref{Rext-str},
$\mathbb D=\mathbb A-\mathbb I_d$ with $\mathbb A\succeq0$ and $\operatorname{tr}\mathbb A\le2$
$\,\mathcal E_{ext}$-a.e.\ (by~(i)). A positive semidefinite matrix with trace $\le2$ has all
eigenvalues in $[0,2]$, so $\mathbb A$ has spectrum in $[0,2]$ and $\mathbb D=\mathbb A-\mathbb I_d$
spectrum in $[-1,1]$; since $\mathbb D$ is symmetric,
$\|\mathbb D\|_{L^\infty(\overline\Omega,\,\mathcal E_{ext})}\le1$, and in particular
$\mathbb D\in L^\infty(\overline\Omega,\mathcal E_{ext};\mathbb R^{d\times d}_{\mathrm{sym}})$.
		\item[-] \emph{(v).} We observe that
		\[
		\Grad_x\mathbf v : \big(c_{d,s}\,\mathcal R_{ext}\big)
		\;=\; \Grad_x \mathbf v : \mathfrak R_{xx} \;-\; \mathcal E_{ext}\ \Div_x\mathbf v ,
		\]
		using $\Grad_x\mathbf v : \mathbb I_d = \Div_x\mathbf v$. The two terms are controlled by different mechanisms. For the first, diagonalise the positive semidefinite density of $\mathfrak R_{xx}$ as $\sum_{i} \lambda_i\, b_i\otimes b_i$ with $\lambda_i \ge 0$ and $\{b_i\}$ orthonormal; then
		\[
		\big| \Grad_x\mathbf v : \mathfrak R_{xx} \big|
		= \Big| \sum_i \lambda_i \, \big(\Grad_x\mathbf v\big) b_i \cdot b_i \Big|
		\ \le\ \|\Grad_x\mathbf v\|_{\mathrm{op}} \sum_i \lambda_i
		\ =\ \|\Grad_x\mathbf v\|_{\mathrm{op}}\ \operatorname{tr}\mathfrak R_{xx}
		\ \le\ 2\,\|\Grad_x\mathbf v\|_{\mathrm{op}}\ \mathcal E_{ext},
		\]
		by (i). Note this step uses only $\mathfrak R_{xx}\succeq0$ --- it is the classical trace bound, unavailable for $\mathcal R_{ext}$ itself. The second term is bounded by $\|\Div_x\mathbf v\|_{L^\infty_x}\,\mathcal E_{ext}$ directly. Integrating in time and adding gives \eqref{Rext-test}.
	\end{itemize}

\end{pf}

\subsubsection{Passage to the limit}
\label{sec:stability-conclusion}

We now let $n\to\infty$ in the approximate system. Each $(\vr_n,\mathbf m_n=\vr_n \vu_n)$ solves \eqref{eq:ERc}--\eqref{eq:ERm} up to the consistency errors of (A3), so for every $\pmb\psi \in C_c^\infty([0,T)\times\mathbb T^d;\mathbb R^d)$,
\begin{align*}
	\int_0^T\!\!\int_{\mathbb T^d} \left[ \mathbf m_n\cdot\partial_t\pmb\psi + \left( \frac{\mathbf m_n\otimes\mathbf m_n}{\vr_n} + p(\vr_n)\mathbb I_d \right) : \Grad_x\pmb\psi \right] dx\,dt
	+ \int_{\mathbb T^d} \mathbf m_{0,n}\cdot\pmb\psi(0,\cdot)\,dx
	\\	= \int_0^T\!\!\int_{\mathbb T^d} \vr_n\, \Grad_x\Lambda_0(U_n)\cdot\pmb\psi \,dx\,dt +\big\langle \mathbf{e}_n^{\,m}, \pmb\psi \big\rangle   .
\end{align*}

\medskip
We write $\vr_n = (\vr_n - \bar\vr) + \bar\vr$ and using \eqref{T} applied to $U_n$ it holds 
\begin{equation*}
	\int_{\mathbb T^d} (\vr_n - \bar\vr)\, \Grad_x\Lambda_0(U_n)\cdot\pmb\psi \,dx
	= c_{d,s} \iint_{\Omega} \mathbb T[U_n] : \Grad_x\pmb\psi \,dx\,dz ,
\end{equation*}
while the mean part is linear, and integrates by parts on $\mathbb T^d$ with no boundary contribution,
\begin{equation}\label{force-mean}
	\bar\vr \int_{\mathbb T^d} \Grad_x\Lambda_0(U_n)\cdot\pmb\psi \,dx
	= -\, \bar\vr \int_{\mathbb T^d} \Lambda_0(U_n)\, \Div_x\pmb\psi \,dx .
\end{equation}

\medskip
\noindent\emph{The limit.} The convective and pressure fluxes converge modulo $\mathcal R_{fluid}$ by \eqref{fluid-defect}; the bulk stress converges modulo $\mathcal R_{ext}$, since $\Ov{\mathbb T_{\rm ER}[U]} = \mathbb T_{\rm ER}[U] + \mathcal R_{ext}$ by \eqref{Rext}. \par 
 The initial data converge by (A1); and the term \eqref{force-mean} is \emph{linear}, so it passes to the limit with no defect at all. Hence invoking the \eqref{app:conv} and continuity of the trace $\Lambda_0 : \dot H^1(\Omega,z^{1-2s}) \to \dot H^s(\mathbb T^d)$ one has $\Lambda_0(U_n) \rightharpoonup \Lambda_0(U)$ in $\dot H^s(\mathbb T^d)$. Collecting the terms and moving the force to the left-hand side, and using $\vm\in C_{\mathrm{weak}}([0,T];L^{\frac{2\gamma}{\gamma+1}}(\mathbb T^d;\mathbb R^d))$, we get
 \begin{align*}
 	 		&\int_0^T\!\!\int_{\mathbb T^d} \Big[ \vm\cdot\partial_t\pmb\psi + \Big( \tfrac{\vm\otimes\vm}{\vr} + p(\vr)\mathbb I_d \Big) : \Grad_x\pmb\psi \Big]\,dx\,dt \\
 	&\quad -\, c_{d,s} \int_0^T\!\!\iint_{\Omega} \mathbb T_{\rm ER}[U] : \Grad_x\pmb\psi \,dx\,dz\,dt
 	\;+\; \bar\vr \int_0^T\!\!\int_{\mathbb T^d} \Lambda_0(U)\, \Div_x\pmb\psi \,dx\,dt \\
 	&\quad +\, \int_0^T\!\!\int_{\mathbb T^d} \Grad_x\pmb\psi : d\mathcal R_{fluid}
 	\;-\; c_{d,s} \int_0^T\!\!\iint_{\overline\Omega} \Grad_x\pmb\psi : d\mathcal R_{ext}
 	= -\int_{\mathbb T^d} \vm_0\cdot\pmb\psi(0,\cdot)\,dx ,
 \end{align*}
 for every $\pmb\psi\in C_c^\infty([0,T)\times\mathbb T^d;\mathbb R^d)$. Since $\vm\in C_{\mathrm{weak}}([0,T];L^{2\gamma/(\gamma+1)})$,
 this identity is equivalent to the form \eqref{eq:w-ERm} carrying the terminal trace $\vm(\tau)$ at a.e.\ $\tau\in(0,T)$.

\medskip
\noindent\emph{The remaining conditions.} The continuity equation \eqref{eq:w-ERc} is obtained as it is linear in $(\vr_n,\mathbf m_n)$ and we pass the limit by directly using \eqref{app:conv}. Using a similar argument, the extension identity \eqref{eq:w-ext} is obtained, as it is linear in $U_n$ and in $\vr_n$. The energy inequality \eqref{eq:ei} follows from by weak lower semicontinuity, the excess kinetic, internal, and bulk energies being recorded exactly by $\mathcal E_{fluid}$ of \eqref{fluid-energy-defect} and $\mathcal E_{ext}$ of \eqref{Eext}. Finally, as a consequence of Subsection \ref{sec:stability-defect} and Lemma \ref{lem:Rext-structure} the compatibility conditions \eqref{dm-compf} and \eqref{dm-compE} holds.  

We conclude this subsection stating the result, that follows from the above discussion: 
\begin{prop}[Sequential stability]\label{prop:sequential-stability}
	Let $(\vr_n,\mathbf m_n)$ be a consistent approximating sequence in the sense of Definition~\ref{def:consistent}, and let $U_n$ be the extension of $\vr_n - \bar\vr$. Then, up to a subsequence, the limits \eqref{app:conv} exist, the defect measures \eqref{fluid-defect}, \eqref{fluid-energy-defect}, \eqref{Rext}, \eqref{Eext} are well defined, and
	\[
	\big(\vr,\ \mathbf m,\ U,\ \mathcal R_{fluid},\ \mathcal R_{ext},\ \mathcal E_{fluid},\ \mathcal E_{ext}\big)
	\]
	is a dissipative weak solution in the sense of Definition~\ref{def:dissipative_weak_sol}.
\end{prop}

In particular, a dissipative weak solution exists whenever a consistent, energy-stable approximation does --- which is precisely what justifies the definition.


\subsection{Compatibility of dissipative solutions}
\label{sec:compatibility}

We now show that the class of dissipative weak solutions introduced in Definition~\ref{def:dissipative_weak_sol} is \emph{compatible} with the classical theory: a dissipative weak solution which happens to be smooth is a classical solution of the Euler--Riesz system, and its defect measures vanish identically. This is the minimal consistency requirement on any notion of generalized solution .

The proof proceeds in two steps, corresponding to the two layers of the formulation:
\begin{enumerate}[leftmargin=*]
	\item[(1)] \emph{The defects vanish.} A Gr\"onwall argument, driven by the energy inequality \eqref{eq:ei} and the compatibility conditions \eqref{dm-compf}--\eqref{dm-compE} give $\mathcal E_{fluid}=\mathcal E_{ext}=0$ and hence $\mathcal R_{fluid}=\mathcal R_{ext}=0$.
	\item[(2)] \emph{The extended system reduces to the classical one.} The extension identity \eqref{eq:w-ext} then gives $\Lambda_0 U=\Phi$ with $(-\Delta)^s\Phi=\vr-\bar\vr$.
\end{enumerate}

\begin{thm}[Compatibility]\label{thm:compatibility}
	Let $d=2,3$, $\gamma>1$, $s\in(0,1)$ (equivalently $\beta\in (0,2)$), and let
	\[
	\big(\vr,\ \mathbf m,\ U,\ \mathcal R_{fluid},\ \mathcal R_{ext},\ \mathcal E_{fluid},\ \mathcal E_{ext}\big)
	\]
	be a dissipative weak solution on $[0,T]$ in the sense of Definition~\ref{def:dissipative_weak_sol}, with initial data $(\vr_0,\mathbf m_0) \in C^1(\T^d)$. Further we assume that $
	\underline \vr:=\min_{[0,T]\times\mathbb T^d} \vr>0 $ and $\vr \in C^1\big([0,T]\times\mathbb T^d\big)$ and $ \vv \in C^1\big([0,T]\times\mathbb T^d;\R^d\big)$.
	Then
	\begin{equation}\label{defect0}
		\mathcal E_{fluid} = \mathcal E_{ext} = 0
		\qquad\text{and}\qquad
		\mathcal R_{fluid} = \mathcal R_{ext} = 0
		\qquad\text{on } [0,T],
	\end{equation}
	the total energy is conserved,
	\begin{equation}\label{comp-eid}
		\mathcal E(\tau) = \mathcal E(0) \qquad\text{for all } \tau\in[0,T],
	\end{equation}
	and $(\vr,\vu=\frac{m}{\vr})$ is a classical solution of the Euler--Riesz system \eqref{eq:ERc}--\eqref{eq:ERr}, with $\Phi = \Lambda_0 U$.
\end{thm}

\begin{pf}
	\emph{Step 1: the defect measures vanish.}
	
	Since $(\vr,\vm,U)$ is smooth, using integrations by parts in the momentum balance \eqref{eq:w-ERm} and the identity \eqref{w-T} for the stress term, in the sense of distributions it holds that 
	\begin{equation}\label{compt-mom}
		\partial_t \mathbf m + \Div_x\!\left( \frac{\mathbf m\otimes\mathbf m}{\vr} \right) + \Grad_x p(\vr) + \vr\,\Grad_x \Lambda_0 U
		\;=\; \Div_x \mathcal R_{\mathrm{tot}},
		\qquad
		\mathcal R_{\mathrm{tot}} := \mathcal R_{fluid} \;-\; c_{d,s}\!\int_0^\infty \mathcal R_{ext}\,dz.
	\end{equation}
 Therefore a smooth dissipative solution solves the system forced by the divergence of the defects. In particular its energy need not be conserved, and this is the circularity the following argument must break. 
	Note that from extension PDE \eqref{eq:w-ext}, we already have  
	\begin{align*}
			\iint_{\Omega} z^{1-2s}\, \Grad_{x,z} U \cdot \Grad_{x,z} \Psi \,dx\,dz = c_{d,s}^{-1} \int_{\mathbb{T}^d} (\vr - \bar{\vr})\, \Lambda_0(\Psi) \,dx.
	\end{align*}
	Testing \eqref{compt-mom} with $\vu:=\tfrac{\vm}{\vr}$, using the continuity equation \eqref{eq:w-ERc} and along the extension PDE \eqref{eq:w-ext} with $ \Psi=U$, a standard computation gives the following energy balance.
	\begin{align}\label{smooth-EB}
		\begin{split}
		&	\frac{d}{dt}\,\mathcal E(t)
		\;=\; -\int_{\mathbb T^d} \Grad_x\vu : d\mathcal R_{fluid}
		\;+\; c_{d,s}\iint_{\overline\Omega} \Grad_x\vu : d\mathcal R_{ext},\\
		\qquad
	& \text{with }	\mathcal E(t) = \int_{\mathbb T^d}\!\Big(\tfrac12\vr|\vu|^2 + P(\vr)\Big)dx + \tfrac{c_{d,s}}{2}\iint_\Omega z^{1-2s}|\Grad_{x,z}U|^2 \dx\;  \text{d}z.
	\end{split}
	\end{align}

	Introduce the total defect mass
	\begin{equation}\label{D-def}
		D(\tau) := \int_{\mathbb T^d} d\mathcal E_{fluid}(\tau) \;+\; \iint_{\overline\Omega} d\mathcal E_{ext}(\tau) \ \ge\ 0 ,
	\end{equation}
	which is non-negative because $\mathcal E_{fluid}, \mathcal E_{ext} \in \mathcal M^+$. The energy inequality \eqref{eq:ei}, combined with \eqref{smooth-EB} integrated in time, yields for a.e.\ $\tau\in(0,T)$
	\begin{equation}\label{D-bound}
		D(\tau)
		\ \le\ \mathcal E(0) - \mathcal E(\tau)
		\ =\ \int_0^\tau\!\!\int_{\mathbb T^d} \Grad_x\vu : d\mathcal R_{fluid}\,dt
		\;-\; c_{d,s}\int_0^\tau\!\!\iint_{\overline\Omega} \Grad_x\vu : d\mathcal R_{ext}\,dt .
	\end{equation}
	It remains to bound the right-hand side by $\displaystyle \int_0^\tau D(\tau) \dt $. Here, the compatibility conditions for defect measures (\eqref{dm-compf} and \eqref{dm-compE}) play an important role
	
	For the fluid term, $\mathcal R_{fluid}\succeq0$, so the classical trace bound applies, and \eqref{dm-compf} gives
	\begin{equation}\label{compt-f1}
		\left| \int_{\mathbb T^d} \Grad_x\vu : d\mathcal R_{fluid} \right|
		\ \le\ \|\Grad_x\vu\|_{L^\infty_x}\, \operatorname{tr}\mathcal R_{fluid}\big(\mathbb T^d\big)
		\ \le\ \frac{\|\Grad_x\vu\|_{L^\infty_x}}{\lambda_1}\, \int_{\mathbb T^d} d\mathcal E_{fluid} .
	\end{equation}
	For the extension term, we use \eqref{dm-compE}, more explicitely Lemma \ref{lem:Rext-structure} to get
	\begin{equation}\label{compt-f2}
		\left| c_{d,s}\iint_{\overline\Omega} \Grad_x\vu : d\mathcal R_{ext} \right|
		= \left| \iint_{\overline\Omega} \Grad_x\vu : \mathbb D \ d\mathcal E_{ext} \right|
		\ \le\ \sqrt d\,\|\Grad_x\vu\|_{L^\infty} \iint_{\overline\Omega} d\mathcal E_{ext},
	\end{equation}
	using $|A : \mathbb D| \le |A|_F\,|\mathbb D|_F \le \sqrt d\,|A|_F\,\|\mathbb D\|_{\mathrm{op}}$ for $\mathbb D \in \mathbb R^{d\times d}_{\mathrm{sym}}$. Substituting \eqref{compt-f1}--\eqref{compt-f2} into \eqref{D-bound},
	\begin{equation}\label{compt-f}
		D(\tau) \ \le\ C_{\vu} \int_0^\tau D(t)\,dt,
		\qquad
		C_{\vu} := C(d,s,\lambda_1) \big\| \Grad_x\vu \big\|_{L^\infty((0,T)\times\mathbb T^d)} \ <\ \infty,
	\end{equation}
	the constant being finite precisely by the hypothesis $\Grad_x\vu\in L^\infty((0,T)\times \T^d)$. Since $D\ge0$ and the inequality \eqref{compt-f} has no source term, Gr\"onwall's lemma gives $D\equiv 0$ on $[0,T]$, i.e. $	\mathcal E_{fluid} = 0, \qquad \mathcal E_{ext} = 0 $. As a simple consequence of Lemma \ref{lem:Rext-structure}, we have  $ 	 \mathcal R_{fluid} = 0 \text{ and } \mathcal R_{ext} = 0$.

	\medskip
	\emph{Step 2: the extended system reduces to the classical one.} With $\mathcal R_{fluid} = \mathcal R_{ext} = 0$, the forced equation \eqref{compt-mom} becomes
	\[
	\partial_t \mathbf m + \Div_x\!\left( \frac{\mathbf m\otimes\mathbf m}{\vr} \right) + \Grad_x p(\vr) + \vr\,\Grad_x\Lambda_0 U = 0
	\]
	in the sense of distributions; since all fields are of class $C^1$ and $\displaystyle \vr\ge\underline\vr>0$, the equation holds pointwise, and $\mathbf m = \vr\vu$ gives the classical momentum equation \eqref{eq:ERm}. Likewise \eqref{eq:w-ERc} is the classical continuity equation \eqref{eq:ERc}.
	
	It remains to eliminate the extension. Since, $U(t)$ is the unique weak solution of \eqref{CS_weak} with datum $\vr(t)-\bar\vr$, so by Proposition~\ref{prop:CS_weak} its trace satisfies
	\[
	\Lambda_0 U = \Phi, \qquad (-\Delta)^s \Phi = \vr - \bar\vr \quad\text{on } \mathbb T^d ,
	\]
	which is \eqref{eq:ERr}. Substituting $\Grad_x\Lambda_0 U = \Grad_x\Phi$ into the momentum equation recovers \eqref{eq:ERm}. Hence $(\vr,\vu,\Phi)$ is a classical solution of \eqref{eq:ERc}--\eqref{eq:ERr}.
\end{pf}

\begin{rmk}[Relation to weak--strong uniqueness]\label{rem:compat-vs-ws}
	Theorem~\ref{thm:compatibility} is the diagonal case of the weak--strong uniqueness principle: it compares a dissipative solution with \emph{itself}, under the additional hypothesis that it is smooth. The general statement replaces $\vu$ in \eqref{smooth-EB} by an independent strong solution $\mathbf v$ and $\mathcal E$ by the corresponding relative energy; the structure of the argument --- energy inequality, then \eqref{compt-f1}--\eqref{compt-f2}, then Gr\"onwall --- is unchanged, and in particular the extension defect is handled by exactly the same estimate \eqref{compt-f2}. 
\end{rmk}

	\section{Relative Energy and Weak--Strong Uniqueness}
	\label{sec:relative}
	
\subsection{Relative energy: formal derivation}
	\subsubsection*{Property of tensor $\T_{\rm ER}[\;\cdot\;]$}
In relative energy, the goal is to compute $\T_{\rm ER}[U-V]$ and its relation with $\T[U]$ and $T[V]$, for suitable $U,V$.
Keeping in mind the definition of $\T_{\rm ER}[\;\cdot\;]$ from \eqref{T}, we define
\begin{equation}\label{TUV-def}
	\T_{\rm ER}(W,\widetilde{W}) := z^{1-2s}\Big(
	\tfrac12\big(\Gradx W\otimes\Gradx \widetilde{W} + \Gradx \widetilde{W}\otimes\Gradx W\big)
	- \tfrac12\big(\Gradxz W\cdot\Gradxz \widetilde{W}\big)\,\mathbb I_d \Big),
\end{equation}
so that $\T_{\rm ER}(W,W)=\T_{\rm ER}[W]$ and the following identity holds
\begin{equation}\label{Tpolar}
	\T_{\rm ER}[W-\widetilde{W}] = \T_{\rm ER}[W] - 2\,\T_{\rm ER}(W,\widetilde{W}) + \T_{\rm ER}[\widetilde{W}] .
\end{equation}
Next, we state a key identity, that we will use extensively in next subsection. 
\begin{lemma}[Extension identity]\label{lem:stress-id}
	Let $W, \widetilde{W}$ be the extensions (following \eqref{CS_weak}) of mean-zero $f, \widetilde{f} \in \dot H^{-s}(\mathbb T^d)$. If $\Gradx\Gradxz W$ and $\Gradx\Gradxz \widetilde{W}$ belong to $L^2(\Omega, z^{1-2s})$, then
	\begin{equation}\label{stress-id}
		\int_{\mathbb T^d} f\,\pmb\psi\cdot\Gradx\Lambdaz(\widetilde{W})\,dx
		\;+\; \int_{\mathbb T^d} \widetilde{f}\,\pmb\psi\cdot\Gradx\Lambdaz(W)\,dx
		\;=\; 2\,c_{d,s}\iint_\Omega \T_{\rm ER}(W,\widetilde{W}) : \Gradx\pmb\psi \,dx\,dz,
	\end{equation}
	where $\pmb\psi\in C^1(\mathbb T^d;\mathbb R^d)$.
	In particular, taking $W=\widetilde{W}$,
	\begin{equation}\label{stress-id2}
		\int_{\mathbb T^d} f\,\pmb\psi\cdot\Gradx\Lambdaz(W)\,dx
		\;=\; c_{d,s}\iint_\Omega \T_{\rm ER}[W] : \Gradx\pmb\psi \,dx\,dz .
	\end{equation}
\end{lemma}

\begin{pf}
	Consider the weak formulation of \eqref{CS-ext} for $W,f$,
	\[ \int_{\mathbb T^d} f\,\Lambdaz(\Psi)\,\dx
	= c_{d,s}\iint_\Omega z^{1-2s}\,\Gradxz W\cdot\Gradxz \Psi \,\dx\,\dz .\]
	Choose the test function $\Psi := \pmb\psi\cdot\Gradx\widetilde W = \psi_k\,\partial_{x_k}\widetilde W$,
	with $\pmb\psi\in C^1(\mathbb T^d;\mathbb R^d)$. Since $\pmb\psi$ is $z$-independent and the trace
	commutes with $\partial_{x_k}$, its trace is $\Lambdaz(\Psi)=\pmb\psi\cdot\Gradx\Lambdaz(\widetilde W)$,
	and the weak form gives
	\[
	\int_{\mathbb T^d} f\,\pmb\psi\cdot\Gradx\Lambdaz(\widetilde W)\,\dx
	= c_{d,s}\iint_\Omega z^{1-2s}\,\Gradxz W\cdot\Gradxz\big(\psi_k\,\partial_{x_k}\widetilde W\big)\,\dx\,\dz .
	\]
	Using $\partial_z\psi_k = 0$,
	\begin{align*}
		\Gradxz W\cdot\Gradxz\big(\psi_k \partial_{x_k}\widetilde{W}\big)
		&= \big(\partial_{x_j}W\,\partial_{x_k}\widetilde{W}\big)\,\partial_{x_j}\psi_k
		\;+\; \psi_k\,\Gradxz W\cdot \partial_{x_k}\Gradxz \widetilde{W} \\
		&= \big(\Gradx W\otimes \Gradx \widetilde{W}\big) : \Gradx\pmb\psi + \psi_k\,\Gradxz W\cdot\partial_{x_k}\Gradxz \widetilde{W}.
	\end{align*}
	Substituting back,
	\begin{equation*}
		\int_{\mathbb T^d} f\,\pmb\psi\cdot\Gradx\Lambdaz(\widetilde{W}) \,\dx 
		= c_{d,s}\iint_\Omega z^{1-2s}\Big[ \big(\Gradx W\otimes\Gradx\widetilde{W}\big):\Gradx\pmb\psi
		+ \psi_k\,\Gradxz W\cdot\partial_{x_k}\Gradxz\widetilde{W} \Big] \,\dx\,\dz .
	\end{equation*}
	Exchanging the roles of $W,\widetilde{W}$ and adding, the two second terms combine into a perfect
	$x$-derivative $\psi_k\,\partial_{x_k}\big(\Gradxz W\cdot\Gradxz\widetilde W\big)$, giving
	\begin{align}
		&\int_{\mathbb T^d}\!\big( f\,\pmb\psi\cdot\Gradx\Lambdaz(\widetilde W) + \widetilde f\,\pmb\psi\cdot\Gradx\Lambdaz(W)\big)\,\dx \nonumber\\
		&\quad= c_{d,s}\iint_\Omega z^{1-2s}\Big[ \big(\Gradx W\otimes\Gradx\widetilde W + \Gradx\widetilde W\otimes\Gradx W\big):\Gradx\pmb\psi
		+ \psi_k\,\partial_{x_k}\big(\Gradxz W\cdot\Gradxz\widetilde W\big)\Big] \,\dx\,\dz. \nonumber
	\end{align}
	For the second term on the right, integrating by parts in $x$,
	\begin{align*}
	\iint_\Omega z^{1-2s}\,\psi_k\,\partial_{x_k}\big(\Gradxz W\cdot\Gradxz\widetilde W\big) \dx \, \dz
	&= -\iint_\Omega z^{1-2s}\,\big(\partial_{x_k}\psi_k\big)\big(\Gradxz W\cdot\Gradxz\widetilde W\big)  \dx \, \dz\\
	&= -\iint_\Omega z^{1-2s}\,\big(\Gradxz W\cdot\Gradxz\widetilde W\big)\,\Divx\pmb\psi  \dx \, \dz\\
	&= -\iint_\Omega z^{1-2s}\,\big(\Gradxz W\cdot\Gradxz\widetilde W\big)\,\mathbb I_d:\Gradx\pmb\psi  \dx \, \dz,
\end{align*}
	since $z^{1-2s}$ is independent of $x$ and $\mathbb T^d$ has no boundary. Substituting into the
	previous display, the bracket becomes
	\[
	\big(\Gradx W\otimes\Gradx\widetilde W + \Gradx\widetilde W\otimes\Gradx W\big):\Gradx\pmb\psi
	- \big(\Gradxz W\cdot\Gradxz\widetilde W\big)\,\mathbb I_d:\Gradx\pmb\psi
	= 2\,\T_{\rm ER}(W,\widetilde W):\Gradx\pmb\psi
	\]
	by \eqref{TUV-def}, which is \eqref{stress-id}. Taking $W=\widetilde W$ and $f=\widetilde f$ halves
	both sides and yields \eqref{stress-id2}. 
\end{pf}

\begin{rmk}\label{rem:stress-id-scope}
Formally, identity \eqref{stress-id} with $f=\vr-\bar\vr$ and $W=\widetilde{W}=U$ is precisely
\eqref{w-T}. In the statement of Lemma~\ref{lem:stress-id}, the regularity hypothesis
$\Gradx\Gradxz W\in L^2(\Om,z^{1-2s})$ is what makes $\Psi=\pmb\psi\cdot\Gradx\widetilde{W}$
admissible. It holds for $W=V$ (since $r\in C^1$; see Proposition~\ref{prop:CS_regularity}), but
\emph{fails} for $W=U$ at the energy level, where only $\Gradxz U\in L^2(\Om,z^{1-2s})$ is available.
In the proof of Lemma~\ref{lem:reduced} we justify this step through a regularization. 
\end{rmk}

\subsubsection*{Formal computation of the relative energy inequality}
Before comparing the dissipative solution with the strong solution, we give a formal
computation highlighting why the relative energy considered here is well suited to the task.
Let $(\vr,\vu)$ and $(r,\vv)$ be a strong solution of \eqref{eq:ERc}--\eqref{eq:ERr} with 
	\begin{equation}
	\vr,r,\ \vu, \vv\in C^1\big([0,T]\times\mathbb T^d\big),\;
	\vr  \geq  \underline{\vr}>0 \text{ and } r \geq \underline{r}>0 \text{ in } [0,T]\times \T^d,
\end{equation}
 emanating from the same initial data, thus $\bar \vr=\bar r$. 
 Let us denote
 \begin{align}\label{UV-formal}
 	U=\mathfrak{E}(\vr-\Ov{\vr}),\; V=\mathfrak{E}( r-\Ov{\vr})  \text{ and } U- V=\mathfrak{E}( \vr-r) .
 \end{align}
  Hence we have $\Psi:=\Lambdaz(V)=(-\Delta)^{-s}(r-\bar\vr)$ and $\Phi=\Lambdaz(U)$ likewise, set
\[
R:=\vr-r ,\qquad \Phi-\Psi=\Lambdaz(U-V)=(-\Delta)^{-s}R .
\]
We notice that at formal level the relative energy \eqref{RE:Def} is equivalent to
\begin{equation}\label{RE}
\Erel(\tau):=\int_{\mathbb T^d}\Big(\tfrac12\,\vr\,|\vu-\vv|^2+P(\vr\,|\,r)\Big)(\tau)\,dx
	\;+\;\tfrac12\,\big\|R(\tau)\big\|_{\dot H^{-s}(\mathbb T^d)}^2 ,
\end{equation}
with relative internal energy and pressure
\[
P(\vr\,|\,r):=P(\vr)-P(r)-P'(r)(\vr-r),\qquad
p(\vr\,|\,r):=p(\vr)-p(r)-p'(r)(\vr-r).
\]
All three terms are non-negative, and since $\gamma>1$ makes $P$ strictly convex and $r \geq \underline{r}>0 \text{ in } [0,T]\times \T^d$ implies $\Erel(\tau)=0$ iff $\vr(\tau)=r(\tau)$, $\vu(\tau)=\vv(\tau)$ $\vr$-a.e., and $\Phi(\tau)=\Psi(\tau)$.
\begin{itemize}[leftmargin=*]
	\item[-] \textbf{Kinetic and Potential energy:}
	These contributions are  similar to the classical barotropic ones; the only new feature is the body force $\vr\Gradx\Phi$. Using \eqref{eq:ERc}--\eqref{eq:ERm} for $(\vr,\vu)$ and $(r,\vv)$, the gradients $\Gradx\big(P'(\vr)-P'(r)\big)$ produced by the kinetic and internal parts cancel, and the two forces enter only through their difference:
	\begin{align}\label{f:f}
		&\frac{d}{dt}\int_{\mathbb T^d}\!\Big(\tfrac12\vr|\vu-\vv|^2+P(\vr\,|\,r)\Big) \dx
		\nonumber \\ 
		&=-\int_{\mathbb T^d}\!\vr\,(\vu-\vv)\otimes(\vu-\vv):\Gradx\vv
		-\int_{\mathbb T^d}\! p(\vr\,|\,r)\,\Divx\vv
		-\int_{\mathbb T^d}\!\vr\,(\vu-\vv)\cdot\Gradx(\Phi-\Psi).
	\end{align}
	\item[-] \textbf{Interaction energy:}
	For the interaction term we use only self-adjointness of $(-\Delta)^{-s}$ and the two continuity equations. Since $\partial_t R=-\Divx(\vr\vu-r\vv)$ and $(-\Delta)^{-s}R=\Phi-\Psi$,
	\begin{equation}\label{f:nl}
		\frac{d}{dt}\Big[\tfrac12\|R\|_{\dot H^{-s}}^2\Big]
		=\int_{\mathbb T^d}\partial_t R\;\big((-\Delta)^{-s}R\big)\,dx
		=\int_{\mathbb T^d}\big(\vr\vu-r\vv\big)\cdot\Gradx(\Phi-\Psi)\,dx .
	\end{equation}
\end{itemize}
Adding \eqref{f:f} and \eqref{f:nl} and using $ -\vr(\vu-\vv)+\big(\vr\vu-r\vv\big)=(\vr-r)\,\vv=R\,\vv ,$ it yields
\begin{equation}\label{mid}
	\frac{d}{dt}\Erel
	=-\int_{\mathbb T^d}\vr(\vu-\vv)\otimes(\vu-\vv):\Gradx\vv
	-\int_{\mathbb T^d} p(\vr\,|\,r)\,\Divx\vv
	+\int_{\mathbb T^d} R\,\vv\cdot\Gradx(\Phi-\Psi)\,dx .
\end{equation}
The nonlocal force has reduced to a single term pairing the density difference $R$ --- the \emph{source} of $\Phi-\Psi$ --- against $\vv\cdot\Gradx(\Phi-\Psi)$.

\emph{From the kernel to the extension.}
Using the extension \eqref{CS-ext} as described in \eqref{UV-formal} we have  $\Lambdaz(U-V)=\Phi-\Psi$. Moreover using \eqref{CS-prop1} it holds
\begin{equation}\label{iso}
	\tfrac{c_{d,s}}2\iint_\Omega z^{1-2s}\big|\Gradxz(U-V)\big|^2\,dx\,dz
	=\tfrac12\,\|R\|_{\dot H^{-s}}^2.
\end{equation}
On the other hand the stress identity (Lemma~\ref{lem:stress-id}) helps us to obtain
\begin{equation}\label{nonloc}
	\int_{\mathbb T^d} R\,\vv\cdot\Gradx\Lambdaz(U-V)\,dx
	=c_{d,s}\iint_\Omega \T_{\rm ER}[U-V]:\Gradx\vv\,dx\,dz,	
\end{equation}
Substituting \eqref{nonloc} into \eqref{mid} yields the formal relative energy identity
\begin{equation}\label{REI}
	\begin{aligned}
		\frac{d}{dt}\Erel
		&=-\int_{\mathbb T^d}\vr\,(\vu-\vv)\otimes(\vu-\vv):\Gradx\vv\,dx
		-\int_{\mathbb T^d} p(\vr\,|\,r)\,\Divx\vv\,dx \\
		&\quad +c_{d,s}\iint_\Omega \T_{\rm ER}[U-V]:\Gradx\vv\,dx\,dz .
	\end{aligned}
\end{equation}
The nonlocal force is thus absorbed \emph{exactly} into the $\mathbb{T}[U-V]$: relative energy \eqref{RE} and relative stress \eqref{REI} are built from the same difference.
We note that now, the relative entropy is
\begin{equation*}
	\Erel(\tau)
	:= \int_{\mathbb T^d}\Big( \tfrac12\,\vr\,|\vu - \vv|^2 + P(\vr\,|\,r) \Big)(\tau)\,dx
	\;+\; \frac{c_{d,s}}{2}\iint_\Omega z^{1-2s}\,\big|\Gradxz (U-V)\big|^2(\tau)\,dx\,dz,
\end{equation*} 

\emph{The relative energy inequality.}
Each term on the right of \eqref{REI} is now controlled by $\Erel$. 
First we have
\[
\Big|\int_{\mathbb T^d}\vr(\vu-\vv)\otimes(\vu-\vv):\Gradx\vv\Big|
\le\|\Gradx\vv\|_{L^\infty}\!\int_{\mathbb T^d}\vr|\vu-\vv|^2\le \Cv\,\Erel .
\]
Then, by the comparison estimate for the relative pressure ($0\le p(\vr\,|\,r)\le C(\underline r,\overline r)\,P(\vr\,|\,r)$ for $\gamma>1$, $r\in[\underline r,\overline r]$), we obtain
\[
\Big|\int_{\mathbb T^d} p(\vr\,|\,r)\,\Divx\vv\Big|
\le\|\Divx\vv\|_{L^\infty}\!\int_{\mathbb T^d} p(\vr\,|\,r)
\le C(\underline r,\overline r)\,\Cv\,\Erel .
\]
Finally, For the stress term, the pointwise bound $|\Gradx W|\le|\Gradxz W|$ gives $$|\T_{\rm ER}[W]:\Gradx\vv|\le z^{1-2s}|\Gradxz W|^2\big(\|\Gradx\vv\|_{L^\infty_x}+\tfrac12\|\Divx\vv\|_{L^\infty_x}\big),$$ so with \eqref{iso}, it yields
\[
\Big|c_{d,s}\iint_\Omega \T_{\rm ER}[U-V]:\Gradx\vv\Big|
\le \tfrac12\Cv\, c_{d,s}\iint_\Omega z^{1-2s}|\Gradxz(U-V)|^2
= \tfrac12\Cv\,\|R\|_{\dot H^{-s}}^2\le \Cv\,\Erel .
\]
No fractional interpolation enters: the bulk term is estimated algebraically, uniformly for every $s\in(0,1)$ and $\gamma>1$. Collecting the three bounds, we deduce 
\begin{equation}\label{gron}
	\frac{d}{dt}\Erel(\tau)\ \le\ C\big(\underline r,\overline r,\Cv\big)\,,\Erel(\tau),
	\qquad \tau\in[0,T),
\end{equation}
which is of Gr\"onwall form, whence $\Erel(\tau)\le \Erel(0)\,\exp\!\big(C\Cv\,\tau\big)$ on $[0,T)$. In particular, if two strong solutions share the same data then $\Erel(0)=0$, so $\Erel\equiv0$ and $\vr=r$, $\vu=\vv$, $\Phi=\Psi$. The rigorous version against a dissipative weak solution where the defect measures of Definition~\ref{def:dissipative_weak_sol} enter \eqref{REI} with a favourable sign, is carried out next. 
\begin{rmk}
	One import observation each term in the R.H.S of \eqref{REI} is controlled by Kinetic, Potential and interaction (in extended variable) energy respectively.This one of the key point in this method that allows us to give state the Theorem \ref{thm:ws} independently of $\gamma$ and $\beta$.
\end{rmk}

	\subsection{The relative energy inequality and weak--strong uniqueness}
	\label{sec:REI-rigorous}

	\begin{df}[Relative energy]\label{def:relenergy}
		For $(\vr,\mathbf m, U)$ and $(r,\vv,V)$ as above, set
		\begin{equation}\label{eq:RE}
			\Erel(\tau)
			:= \int_{\mathbb T^d}\Big( \frac12\,\vr\,\left| \frac{\vm}{\vr} - \vv\right|^2 + P(\vr\,|\,r) \Big)(\tau)\,dx
			\;+\; \frac{c_{d,s}}{2}\iint_\Omega z^{1-2s}\,\big|\Gradxz (U-V)\big|^2(\tau)\,dx\,dz,
		\end{equation}
		with the relative internal energy and relative pressure
		\[
		P(\vr\,|\,r) := P(\vr) - P(r) - P'(r)(\vr - r).
		\]
		The kinetic term is understood in conservative variables as $|\mathbf m - \vr\vv|^2/(2\vr)$, with value $0$ where $\vr = \mathbf m = 0$.
	\end{df}
	
\begin{rmk}[Relation between $P(\cdot|\cdot)$ and $p(\cdot| \cdot)$ ]\label{rem:relP} Recall that 
	\[	p(\vr\,|\,r) := p(\vr) - p(r) - p'(r)(\vr - r).\]
	For $P(\vr)=\tfrac{a}{\gamma-1}\vr^\gamma$ and $p(\vr)=a\vr^\gamma$ one has $p=(\gamma-1)P$,
	so the relative quantities are proportional, $p(\vr\,|\,r)=(\gamma-1)\,P(\vr\,|\,r)$. In
	particular $P(\vr\,|\,r)\ge0$, with equality iff $\vr=r$ (strict convexity of $P$ for
	$\gamma>1$).
\end{rmk}

\begin{rmk}[Coercivity]\label{rem:coercivity}
	Each of the three terms in \eqref{eq:RE} is nonnegative, so $\Erel(\tau)\ge0$. By linearity
	of the extension, $U-V$ is the extension of $\vr-r$, and the isometry \eqref{CS-prop1}
	identifies the bulk term with the negative Sobolev distance of the densities,
	\begin{equation}\label{eq:BI}
		\frac{c_{d,s}}{2}\displaystyle\iint_{\Omega} z^{1-2s}\,\big|\Gradxz(U-V)\big|^2\,dx\,dz
		\;=\; \frac12\,\|\vr-r\|_{\dot H^{-s}(\mathbb T^d)}^2 .
	\end{equation}
	If $\vr(\tau)=r(\tau)$ and $\vu(\tau)=\vv(\tau)$, then $U(\tau)=V(\tau)$ by
	linearity, and every term in \eqref{eq:RE} vanishes; hence $\Erel(\tau)=0$.
	Conversely, suppose $\Erel(\tau)=0$. Then each nonnegative term vanishes separately. From
	$\displaystyle\int_{\mathbb T^d} P(\vr\,|\,r)\,dx=0$ and the strict convexity of $P$
	(Remark~\ref{rem:relP}), we first obtain
	\[
	\vr(\tau)=r(\tau)\qquad\text{a.e. on }\mathbb T^d .
	\]
	Since the strong solution satisfies $r \geq \underline{r}>0 \text{ in } [0,T]\times \T^d$, the density is now bounded below, and
	$\displaystyle\int_{\mathbb T^d}\tfrac12\,\vr\,|\vu-\vv|^2\,dx=0$ then forces
	\[
	\vu(\tau)=\vv(\tau)\qquad\text{a.e. on }\mathbb T^d ,
	\]
	and hence $U(\tau)=V(\tau)$ by \eqref{eq:BI}. Thus $\Erel(\tau)=0$ if and only if
	$(\vr,\vu,U)(\tau)=(r,\vv,V)(\tau)$.
\end{rmk}

	\subsubsection{The reduced relative energy identity}
	
	We first expand $\Erel$ against arbitrary smooth comparison fields, using only the weak identities of the dissipative solution. 
	
	\begin{lemma}[Relative energy inequality I]\label{lem:generic}
		Let $r\in C^1([0,T]\times\mathbb T^d)$ with $r\ge\underline r>0$, $\vv\in C^1([0,T]\times\mathbb T^d;\mathbb R^d)$, $\Ov{r}=\Ov{\vr}$ and let $V=\mathfrak{E}(r-\bar\vr)$ following \eqref{CS-ext} and \eqref{Not:CSext}. Then for a.e.\ $\tau\in(0,T)$ it holds
	\begin{equation}\label{REI1}
		\begin{aligned}
			\Erel(\tau)+&\int_{\mathbb T^d}d\mathcal E_{fluid}(\tau)+\iint_{\overline\Omega}d\mathcal E_{ext}(\tau)
			\\
			\ \le\ \Erel(0)&-\int_0^\tau\!\!\int_{\mathbb T^d}\vr \lr{\tfrac{\vm}{\vr}-\vv}\otimes\lr{\tfrac{\vm}{\vr}-\vv}:\Gradx\vv \dx\dt 
			-\int_0^\tau\!\!\int_{\mathbb T^d}p(\vr\,|\,r)\Divx\vv\dx \dt \\
			&\quad-\int_0^\tau\!\!\int_{\mathbb T^d}\vr\lr{\tfrac{\vm}{\vr}-\vv} \cdot\Big(\partial_t\vv+\vv\cdot\Gradx\vv+\tfrac1r\Gradx p(r)+\Gradx\Lambdaz(V)\Big) \dx \dt \\
			&\quad+\int_0^\tau\!\!\int_{\mathbb T^d}\Big(\tfrac{r-\vr}{r}p'(r)+\Lambdaz(V)-\Lambdaz(U)\Big)\big(\partial_t r+\Divx(r\vv)\big) \dx \dt \\
			&\quad+c_{d,s}\int_0^\tau\!\!\iint_\Omega\T[U]:\Gradx\vv \dx\, \dz\,  \dt
			-\bar\vr\int_0^\tau\!\!\int_{\mathbb T^d}\Lambdaz(U)\Divx\vv \dx \dt \\
			&\quad-\int_0^\tau\!\!\int_{\mathbb T^d}(\vr-r)\,\vv\cdot\Gradx\Lambdaz(V) \dx \dt
			+\int_0^\tau\!\!\int_{\mathbb T^d} \Divx\lr{r\,\vv} \Lambdaz(U) \dx \dt \\
			&\quad-\int_0^\tau\!\!\int_{\mathbb T^d}\Gradx\vv:d\mathcal R_{fluid} \dx \dt 
			+c_{d,s}\int_0^\tau\!\!\iint_{\overline\Omega}\Gradx\vv:d\mathcal R_{ext} \dx \dt = \sum_{i=1}^{11}\mathcal{T}_i
		\end{aligned}
	\end{equation}
	\end{lemma}
	
	\begin{pf}
		For a.e. $\tau \in (0,T)$, using $p(r)=rP'(r)-P(r)$ we expand \eqref{eq:RE} and it gives
		\begin{align*}
					\Erel(\tau)
				&=\int_{\mathbb T^d}\lr{\!\tfrac12 \frac{|\vm|^2}{\vr} +P(\vr)} \dx 
				-\int_{\mathbb T^d}\!\mathbf m\cdot\vv dx \dt 
				+\int_{\mathbb T^d}\!\vr\,\phi\;  \dx 
				+\int_{\mathbb T^d}\!p(r)\dx  \\
				&	\quad	+\tfrac{c_{d,s}}2\iint_\Omega\! z^{1-2s}|\Gradxz U|^2 \dx \dz 
				-c_{d,s}\!\iint_\Omega\! z^{1-2s}\Gradxz U\cdot\Gradxz V \dx \dz \\
				&\quad
				+\tfrac{c_{d,s}}2\iint_\Omega\! z^{1-2s}|\Gradxz V|^2 \dx \dz=\sum_{i=1}^{7} \mathcal{B}_i (\tau)
		\end{align*}
		with $\phi:=\tfrac12|\vv|^2-P'(r)$, all at time $\tau$.

		\smallskip
		\noindent\emph{The terms $\mathcal B_1+\mathcal B_5$.} Their sum is the total energy of the weak solution at time $\tau$, so the energy inequality \eqref{eq:ei} applies verbatim:
		\begin{equation}\label{B15}
			\mathcal B_1(\tau)+\mathcal B_5(\tau)+\int_{\mathbb T^d}d\mathcal E_{fluid}(\tau)+\iint_{\overline\Omega}d\mathcal E_{ext}(\tau)
			\ \le\ \mathcal B_1(0)+\mathcal B_5(0).
		\end{equation}
		This is the only inequality in the computation; all remaining steps are identities.

		\noindent\emph{The term $\mathcal B_2(\tau)$.} We use the momentum identity \eqref{eq:w-ERm} with $\pmb\psi=\vv$, admissible since $\vv\in C^1_{t,x}$.  This yields
		\begin{equation}\label{B2}
			\begin{aligned}
				\big[\mathcal B_2\big]_0^\tau
				&=-\int_0^\tau\!\!\int_{\mathbb T^d}\lr{\vm\cdot\partial_t\vv+\tfrac{\vm\otimes\vm}{\vr}:\Gradx\vv+p(\vr)\Divx\vv}\dx\dt\\
				&\quad+c_{d,s}\int_0^\tau\!\!\iint_\Omega\T_{\rm ER}[U]:\Gradx\vv \dx\dz\dt
				-\bar\vr\int_0^\tau\!\!\int_{\mathbb T^d}\Lambdaz(U)\Divx\vv \dx\dt\\
				&\quad-\int_0^\tau\!\!\int_{\mathbb T^d}\Gradx\vv:d\mathcal R_{fluid} \dx\dt
				+c_{d,s}\int_0^\tau\!\!\iint_{\overline\Omega}\Gradx\vv:d\mathcal R_{ext} \dx\dt .
			\end{aligned}
		\end{equation}

		\noindent\emph{The terms $\mathcal B_3+\mathcal B_4+\mathcal B_4$.} The weak continuity equation \eqref{eq:w-ERc} with $\phi=\tfrac12|\vv|^2-P'(r)\in C^1_{t,x}$, together with $\displaystyle \left[\int_{\T^d} p(r) \dx  \right]_{t=0}^{t=\tau}=\int_0^\tau\!\int_{\T^d}\partial_t p(r) \dx \dt$, gives
		\begin{equation}\label{B34}
			\begin{aligned}
			\big[\mathcal B_3+\mathcal B_4\big]_0^\tau
			=&\int_0^\tau\!\!\int_{\mathbb T^d}\lr{\vr\,\partial_t\tfrac12|\vv|^2+\vm\cdot\Gradx\tfrac12|\vv|^2}\dx\dt \\
			&+\int_0^\tau\!\!\int_{\mathbb T^d}\lr{-\vr\,\partial_tP'(r)-\vm\cdot\Gradx P'(r)+\partial_tp(r)}\dx\dt .
		\end{aligned}
		\end{equation}

		Next, We put a bit more focus on the Terms $\mathcal{B}_6$ and $\mathcal{B}_7$

		Here we use that $V=\mathfrak{E}(r-\bar\vr)$ with $\bar\vr=\bar r$. Note that, no equation for $(r,\vv)$ is invoked at this stage.

\emph{The term $\mathcal B_6(\tau)$.} Applying \eqref{eq:w-ext} with $\Psi=V$, and using
$\int_{\mathbb T^d}\Lambdaz(V)\,\dx=0$,
\begin{equation*}
	\mathcal B_6(\tau) = -\int_{\mathbb T^d}(\vr-\bar\vr)\,\Lambdaz(V)\,\dx
	= -\int_{\mathbb T^d}\vr\,\Lambdaz(V)\,\dx .
\end{equation*}
Since $r\in C^1((0,T)\times\mathbb T^d)$, the trace $\Lambdaz(V)\in C^1((0,T)\times\mathbb T^d)$.
Testing the weak continuity equation \eqref{eq:w-ERc} with $\phi=\Lambdaz(V)$ gives
\begin{equation}\label{B6-dt}
	\mathcal B_6(\tau)-\mathcal B_6(0)
	= -\int_0^\tau\!\!\int_{\mathbb T^d}\Big(\vr\,\partial_t\Lambdaz(V)
	+ \vm\cdot\Gradx\Lambdaz(V)\Big)\,\dx\,\dt .
\end{equation}

\smallskip
\noindent\emph{The term $\mathcal B_7(\tau)$.} Testing \eqref{eq:w-ext} with $\Psi=V$ once more
--- now reading it as the energy identity --- yields
\begin{equation}\label{B7-static}
	\mathcal B_7 = \tfrac{c_{d,s}}2\iint_\Omega z^{1-2s}\,|\Gradxz V|^2\,\dx\,\dz
	= \tfrac12\int_{\mathbb T^d}(r-\bar\vr)\,\Lambdaz(V)\,\dx .
\end{equation}
Since $E$ is linear and $t$-independent, $\partial_t V=E(\partial_t r)$, and $\partial_t r$ is
mean-zero as $\bar\vr$ is constant; differentiating \eqref{B7-static} in time and testing
\eqref{eq:w-ext} with $\Psi=\partial_t V$,
\begin{equation}\label{B7-symm}
	\frac{d}{dt}\,\mathcal B_7(t)
	= c_{d,s}\iint_\Omega z^{1-2s}\,\Gradxz V\cdot\Gradxz\partial_t V\,\dx\,\dz
	= \int_{\mathbb T^d}\partial_t r\,\Lambdaz(V)\,\dx ,
\end{equation}
the factor $\tfrac12$ absorbed because the two symmetric contributions of the product rule
coincide. Hence
\begin{equation}\label{B7-dt}
	\mathcal B_7(\tau)-\mathcal B_7(0)
	= \int_0^\tau\!\!\int_{\mathbb T^d}\partial_t r\,\Lambdaz(V)\,\dx\,\dt .
\end{equation}

It is convenient to combine $\mathcal B_6$ and $\mathcal B_7$. Since $\mathfrak{E}$ is linear and
$t$-independent, $\partial_t V=\mathfrak{E}(\partial_t r)$ is the extension of the mean-zero source
$\partial_t r$. Testing \eqref{eq:w-ext} for $U$ (source $\vr-\bar\vr$) with $\Psi=\partial_t V$,
and \eqref{eq:w-ext} for $\partial_t V$ (source $\partial_t r$) with $\Psi=U$, both have the same
left-hand side $\iint_\Omega z^{1-2s}\Gradxz U\cdot\Gradxz\partial_t V\,\dx\dz$; equating the
right-hand sides,
\[
\int_{\mathbb T^d}(\vr-\bar\vr)\,\partial_t\Lambdaz(V)\,\dx
= \int_{\mathbb T^d}\partial_t r\,\Lambdaz(U)\,\dx ,
\]
and since $\int_{\mathbb T^d}\partial_t\Lambdaz(V)\,\dx=0$, also
$\int_{\mathbb T^d}\vr\,\partial_t\Lambdaz(V)\,\dx=\int_{\mathbb T^d}\partial_t r\,\Lambdaz(U)\,\dx$.
Adding \eqref{B6-dt} and \eqref{B7-dt}, the two time derivatives of $\Lambdaz(V)$ cancel and there
remains
\begin{equation}\label{B67}
	\big[\mathcal B_6(s)+\mathcal B_7(s)\big]_{s=0}^{s=\tau}
	= -\int_0^\tau\!\!\int_{\mathbb T^d}\vm\cdot\Gradx\Lambdaz(V)\,\dx\,\dt
	+ \int_0^\tau\!\!\int_{\mathbb T^d}\partial_t r\,\big(\Lambdaz(V)-\Lambdaz(U)\big)\,\dx\,\dt .
\end{equation}

		We rewrite the second term in the R.H.S of \eqref{B67} as 

		\begin{align*}
			&	\int_{\mathbb T^d}\partial_t r\,\big(\Lambdaz(V)-\Lambdaz(U)\big) \dx \\
			&= \int_{\mathbb T^d}\big(\Lambdaz(V)-\Lambdaz(U)\big)\big(\partial_t r+\Divx(r\vv)\big) \dx
			+ \int_{\mathbb T^d}\big( r\,\vv\cdot \Gradx \Lambdaz(V)+ \Div \lr{ r\,\vv}\Lambdaz(U)\big) \dx .
		\end{align*}
The two terms are treated differently on purpose: $\Lambdaz(V)\in C^1_{t,x}$ carries a derivative, whereas $\Lambdaz(U)\in\dot H^{s}(\mathbb T^d)$ does not, and the pairing is kept in the form in which it is a Lebesgue integral.

	\smallskip
	
		\noindent\emph{Regrouping.} All identities below are pointwise a.e., with the convention that expressions carrying $\vr$ in the denominator vanish on $\{\vr=0\}$, where also $\vm=0$; note $\vm-\vr\vv=0$ there as well, so both sides vanish on the vacuum set. 
	
	The terms of \eqref{B2}, \eqref{B34} carrying derivatives of $\vv$ combine as follows. From $\vr\,\partial_t\tfrac12|\vv|^2=\vr\,\vv\cdot\partial_t\vv$,
	\[
	-\vm\cdot\partial_t\vv+\vr\,\partial_t\tfrac12|\vv|^2=-(\vm-\vr\vv)\cdot\partial_t\vv ,
	\]
	while $\vm\cdot\Gradx\tfrac12|\vv|^2=(\vm\otimes\vv):\Gradx\vv$ gives
	\begin{align*}
		-\tfrac{\vm\otimes\vm}{\vr}:\Gradx\vv+\vm\cdot\Gradx\tfrac12|\vv|^2
		&=-\tfrac{\vm\otimes(\vm-\vr\vv)}{\vr}:\Gradx\vv\\
		&=-\tfrac{(\vm-\vr\vv)\otimes(\vm-\vr\vv)}{\vr}:\Gradx\vv-(\vm-\vr\vv)\cdot\lr{\vv\cdot\Gradx\vv},
	\end{align*}
	where the last step used $\vm=(\vm-\vr\vv)+\vr\vv$ and $\big(\vv\otimes(\vm-\vr\vv)\big):\Gradx\vv=(\vm-\vr\vv)\cdot(\vv\cdot\Gradx\vv)$. Hence
	\begin{equation}\label{grp-kin}
		\begin{aligned}
			&	-\vm\cdot\partial_t\vv-\tfrac{\vm\otimes\vm}{\vr}:\Gradx\vv
			+\vr\,\partial_t\tfrac12|\vv|^2+\vm\cdot\Gradx\tfrac12|\vv|^2 \\
			&	=-\tfrac{(\vm-\vr\vv)\otimes(\vm-\vr\vv)}{\vr}:\Gradx\vv
			-(\vm-\vr\vv)\cdot\lr{\partial_t\vv+\vv\cdot\Gradx\vv} .			
		\end{aligned}
	\end{equation}
	The first term in the R.H.S of \eqref{grp-kin} is exactly leads to the term $\mathcal{T}_2$ in \eqref{REI1}.
	For the pressure terms we use $rP''(r)=p'(r)$, i.e.\ $\partial_tP'(r)=\tfrac{p'(r)}{r}\partial_tr$ and $\Gradx P'(r)=\tfrac1r\Gradx p(r)$, and split $\vm=(\vm-\vr\vv)+\vr\vv$ in the flux term. Then
	\begin{equation}\label{grp-pres}
		\begin{aligned}
			&\int_{\mathbb T^d}\lr{-p(\vr)\Divx\vv-\vr\,\partial_tP'(r)-\vm\cdot\Gradx P'(r)+\partial_tp(r)}\dx\\
			&\qquad=\int_{\mathbb T^d}\lr{-p(\vr\,|\,r)\Divx\vv-(\vm-\vr\vv)\cdot\tfrac1r\Gradx p(r)
				+\tfrac{r-\vr}{r}p'(r)\big(\partial_tr+\Divx(r\vv)\big)}\dx .
		\end{aligned}
	\end{equation}
	Clearly the first term in the  R.H.S of \eqref{grp-pres} is $\mathcal{T}_3$ in \eqref{REI1}.

	\smallskip

	 \noindent\emph{Conclusion.} In \eqref{B67} we split $\vm=(\vm-\vr\vv)+\vr\vv$ in the first term. Together with the identity above for $\partial_t r$, the contribution of $\mathcal B_6+\mathcal B_7$ thus consists of $-(\vm-\vr\vv)\cdot\Gradx\Lambdaz(V)$, of the continuity residual with coefficient $\Lambdaz(V)-\Lambdaz(U)$, and of the remainder
	\begin{align*}
			 -\int_{\mathbb T^d}\vr\,\vv\cdot\Gradx\Lambdaz(V)\dx 
		 + \int_{\mathbb T^d}\big( r\,\vv\cdot \Gradx \Lambdaz(V)+ \Div \lr{ r\,\vv}\Lambdaz(U)\big) \dx\\
		=-\int_{\mathbb T^d}(\vr-r)\,\vv\cdot\Gradx\Lambdaz(V)\dx+\int_{\mathbb T^d} \Div\lr{r\,\vv} \Lambdaz(U)\dx ,
	\end{align*}
	
	 which is the equals to $\mathcal{T}_8+\mathcal{T}_9$ in \eqref{REI1}. \par 
	 Next, adding now \eqref{B15}, \eqref{B2}, \eqref{B34} in the form \eqref{B67}, \eqref{grp-kin} and \eqref{grp-pres}, the three terms proportional to $\vm-\vr\vv$ assemble the momentum residual $\mathcal{T}_4$ in \eqref{REI1}, i.e. 
	 \[
	 -\int_0^\tau\!\!\int_{\mathbb T^d}(\vm-\vr\vv)\cdot\Big(\partial_t\vv+\vv\cdot\Gradx\vv+\tfrac1r\Gradx p(r)+\Gradx\Lambdaz(V)\Big)\dx\dt .
	 \]
	 Meanwhile the residual terms of \eqref{grp-pres} and of \eqref{B67} assemble the continuity residual with coefficient $\tfrac{r-\vr}{r}p'(r)+\Lambdaz(V)-\Lambdaz(U)$, that is the term $\mathcal{T}_5$ in \eqref{REI1} . Hence, collecting all the information from above we obtain the relative energy inequality \eqref{REI1}.

	\end{pf}

	\medskip

	Our next goal is to substitute the strong solution $(r,\vv)$ in \eqref{REI1}.
\begin{lemma}[Relative energy inequality II]\label{lem:reduced}
	Let $(r,\vv,V)\in\mathcal C_{\mathrm{str}}(T)$ and $\Ov{\vr}=\int_{\T^d} \vr_0 \dx = \int_{\T^d} r_0 \dx=\Ov{r}$. Then for a.e.\ $\tau\in(0,T)$,
	\begin{equation}\label{REI2}
		\begin{aligned}
			&\Erel(\tau)+\int_{\mathbb T^d}d\mathcal E_{fluid}(\tau)+\iint_{\overline\Omega}d\mathcal E_{ext}(\tau)
			\ \le\ \Erel(0)\\
			&\quad-\int_0^\tau\!\!\int_{\mathbb T^d}\vr \lr{\tfrac{\vm}{\vr}-\vv} \otimes \lr{\tfrac{\vm}{\vr}-\vv}:\Gradx\vv \dx \dt
			-\int_0^\tau\!\!\int_{\mathbb T^d}p(\vr\,|\,r)\Divx\vv \dx \dt
			\;+\;\mathcal N\\
			&\quad-\int_0^\tau\!\!\int_{\mathbb T^d}\Gradx\vv:d\mathcal R_{fluid}
			+c_{d,s}\int_0^\tau\!\!\iint_{\overline\Omega}\Gradx\vv:d\mathcal R_{ext},
		\end{aligned}
	\end{equation}
	where $\mathcal{N}$ the contribution coming due to extension PDE is given by
	\begin{equation}\label{N}
		\begin{aligned}
			\mathcal{N}
			&=c_{d,s}\int_0^\tau\!\!\iint_\Omega\T_{\rm ER}[U]:\Gradx\vv \dx \dz \dt
			+ \int_0^\tau\!\!\int_{\mathbb T^d}\Divx\lr{(r-\bar\vr)\,\vv}\Lambdaz(U) \dx \dt  \\
			&\quad
			-\int_0^\tau\!\!\int_{\mathbb T^d}(\vr-\bar\vr)\,\vv\cdot\Gradx\Lambdaz(V) \dx \dt  +\int_0^\tau\!\!\int_{\mathbb T^d}(r-\bar\vr)\,\vv\cdot\Gradx\Lambdaz(V) \dx \dt.
		\end{aligned}
	\end{equation}
	Moreover, it holds that 
	\begin{align}\label{N-final}
		\mathcal N \;=\; c_{d,s}\int_0^\tau\!\!\iint_\Omega \T_{\rm ER}[U-V] : \Gradx\vv \,dx\,dz\,dt .
	\end{align}
\end{lemma}

\begin{pf}
	Since $(r,\vv)$ is a strong solution, it holds
	\begin{align*}
		&\partial_t r+\Divx(r\vv)=0 ,\\
		&\partial_t\vv+\vv\cdot\Gradx\vv+\tfrac1r\Gradx p(r)+\Gradx\Lambdaz(V)=0 ,
	\end{align*}
	so in \eqref{REI1} we have $\mathcal{T}_4=0$ and $\mathcal{T}_5=0$. Note that using $\Ov{\vr}=\Ov{r}$ and Remark \ref{rem:trace-mean-zero} it holds 
	\begin{align*}
		\mathcal{T}_6+\mathcal{T}_7+\mathcal{T}_8+\mathcal{T}_9= \mathcal{N},
	\end{align*}
	as defined in \eqref{N} and complete the proof of \eqref{REI2}. Next to identify $\mathcal N$, we will use \eqref{TUV-def} and Lemma  \ref{lem:stress-id}.

	The first term in the R.H.S of \eqref{N} is already in the required form. The last term of \eqref{N} is the identity \eqref{stress-id} with $f=r-\bar\vr$, $W=V$, $\pmb\psi=\vv$, admissible since $\Gradx\Gradxz V\in L^2(\Omega,z^{1-2s})$ by the regularity class of $v$ assumed in \eqref{strong-class} and the Proposition \ref{prop:CS_regularity}: 
	\begin{equation}\label{eq:N-TV}
		\int_0^\tau\!\!\int_{\mathbb T^d}(r-\bar\vr)\,\vv\cdot\Gradx\Lambdaz(V) \dx 
		=c_{d,s}\int_0^\tau\!\!\iint_\Omega\T_{\rm ER}[V]:\Gradx\vv \dx \dz.
	\end{equation}

\emph{The cross terms $\T_{\rm ER}(U,V)$.} The two middle terms of \eqref{N} form the left-hand side of the identity \tcbb{\eqref{stress-id}} with $(f,W)=(\vr-\bar\vr,U)$ and $(\widetilde f,\widetilde W)=(r-\bar\vr,V)$. Unfortunately, Lemma~\ref{lem:stress-id} does not apply as it stands: its proof tests the extension equation with the field $\pmb\psi\cdot\Gradx U$, whose full gradient involves $\Gradx\Gradxz U$, and at the energy level we control only $\Gradxz U\in L^2(\Omega,z^{1-2s})$ (Remark~\ref{rem:stress-id-scope}).

We therefore regularize in space (the variable $x$), and we do so \emph{once}, at the level
of the density. Let $\eta_\epsilon$ be a mollifier on $\mathbb T^d$ with
$\int_{\mathbb T^d}\eta_\epsilon\,dx=1$, and set
\[
\vr_\epsilon:=\eta_\epsilon\ast\vr,\qquad U_\epsilon:=\mathfrak E[\vr_\epsilon-\bar\vr].
\]
Mollification preserves the mean, $\overline{\vr_\epsilon}=\bar\vr$, so $\vr_\epsilon-\bar\vr$
is mean-zero and $U_\epsilon$ is well defined. Only the density carries $\epsilon$: $U_\epsilon$
is not regularized on its own but \emph{defined} as an extension, so tangential smoothness (in the $x$-variable) of
$\vr_\epsilon$ is inherited by $U_\epsilon$ and $\Gradx\Gradxz U_\epsilon\in L^2(\Omega,z^{1-2s})$,
which is what Lemma~\ref{lem:stress-id} requires. The isometry \eqref{CS-prop1} controls the regularization error through the negative Sobolev norm of the density difference, which tends to zero as $\epsilon\to0$; the stress identity for $(U,V)$ is thereby recovered.

Concretely, since $\mathfrak E$ is linear, $U_\epsilon-U=\mathfrak E[\vr_\epsilon-\vr]$, and
\eqref{CS-prop1} gives, for a.e.\ $t\in(0,\tau)$,
\begin{equation}\label{eq:eps-energy}
	c_{d,s}\big\|\Gradxz(U_\epsilon-U)\big\|_{L^2(\Omega,z^{1-2s})}^2
	=\big\|\Lambdaz(U_\epsilon)-\Lambdaz(U)\big\|_{\dot H^{s}(\mathbb T^d)}^2
	=\big\|\vr_\epsilon-\vr\big\|_{\dot H^{-s}(\mathbb T^d)}^2
	\xrightarrow[\epsilon\to0]{}0 .
\end{equation}
Additionally, we also have $ 	|| \vr- \vr_\epsilon ||_{L^p_t L^\gamma_x}  	\xrightarrow[\epsilon\to0]{}0  \text{ for } 1\leq p<\infty $.
Rewriting the $\Lambdaz(U)$ term of \eqref{N} and integrating by parts in $x$,
\begin{align*}
	\int_0^\tau\!\!\int_{\mathbb T^d}\Divx\!\big((r-\bar\vr)\,\vv\big)\,\Lambdaz(U)\,dx\,dt
	&=\int_0^\tau\!\!\int_{\mathbb T^d}\Divx\!\big((r-\bar\vr)\,\vv\big)\big(\Lambdaz(U)-\Lambdaz(U_\epsilon)\big)\,dx\,dt\\
	&\quad-\int_0^\tau\!\!\int_{\mathbb T^d}(r-\bar\vr)\,\vv\cdot\Gradx\Lambdaz(U_\epsilon)\,dx\,dt .
\end{align*}
For fixed $\epsilon$ one has $\Gradx\Gradxz U_\epsilon\in L^2(\Omega,z^{1-2s})$, so
Lemma~\ref{lem:stress-id} applies to the pair $(U_\epsilon,V)$ ($V$ being smooth):
\begin{equation}\label{eq:polar-eps}
	\int_{\mathbb T^d}(\vr_\epsilon-\bar\vr)\,\vv\cdot\Gradx\Lambdaz(V) \dx
	+\int_{\mathbb T^d}(r-\bar\vr)\,\vv\cdot\Gradx\Lambdaz(U_\epsilon) \dx 
	=2\,c_{d,s}\iint_\Omega\T_{\rm ER}(U_\epsilon,V):\Gradx\vv \dx \dz.
\end{equation}
Combining the two displays gives the regularized identity
\begin{align}\label{stress-reg}
	&\int_0^\tau\ \int_{\mathbb T^d}(\vr-\bar\vr)\,\vv\cdot\Gradx\Lambdaz(V) \dx \dt 
	+\int_0^\tau\!\!\int_{\mathbb T^d}\Divx\!\big((r-\bar\vr)\,\vv\big)\Lambdaz(U)\,dx\,dt \nonumber\\ 
	&	-2c_{d,s}\int_0^\tau\!\!\iint_\Omega\T_{\rm ER}(U,V):\Gradx\vv\,dx\,dz\,dt\nonumber \\
	&=\int_0^\tau\!\!\int_{\mathbb T^d}\Divx\!\big((r-\bar\vr)\,\vv\big)\big(\Lambdaz(U)-\Lambdaz(U_\epsilon)\big)\,dx\,dt
	-2c_{d,s}\int_0^\tau\!\!\iint_\Omega\T_{\rm ER}(U_\epsilon-U,V):\Gradx\vv\,dx\,dz\,dt \nonumber \\
	&\qquad+\int_0^\tau\!\!\int_{\mathbb T^d}(\vr_\epsilon-\vr)\,\vv\cdot\Gradx\Lambdaz(V)\,dx\,dt = \sum_{i=1}^{3} \ \mathcal{Z}_i.
\end{align}
As $\epsilon\to0$ each $\mathcal Z_i$ in the RHS of \eqref{stress-reg} vanishes: using \eqref{eq:eps-energy} and $r,\vv\in C^1_{t,x}$ it holds
\begin{align*}
	&	\left| \mathcal{Z}_1\right|  \leq C \lr{ || r ||_{C^1_{t,x}} , || \vv||_{C^1_{t,x}} } \int_0^\tau \big\|\Lambdaz(U_\epsilon)-\Lambdaz(U)\big\|_{\dot H^{s}_x}  \dt,\\
	&	\left| \mathcal{Z}_2\right|  \leq C \lr{  || \vv||_{C^1_{t,x}} } \int_0^\tau \|\Gradxz(U_\epsilon-U)\|_{L^2(\Om,z^{1-2s})}\|\Gradxz V\|_{L^2(\Om,z^{1-2s})}\tcbb{\,}\dt  \\
	&\left| \mathcal{Z}_3\right|  \leq C \lr{ || r ||_{C^1_{t,x} }, || \vv||_{C^1_{t,x}} } \int_0^\tau \|\vr_\epsilon-\vr\|_{\dot H^{-s}} || \Gradx\Lambdaz(V) ||_{\dot H^{s}}  \dt ,
\end{align*}
where we used 
$\|\T_{\rm ER}(U_\epsilon-U,V)\|_{L^1(\Omega)} \leq C  \|\Gradxz(U_\epsilon-U)\|_{L^2(\Om,z^{1-2s})}\|\Gradxz V\|_{L^2(\Om,z^{1-2s})}$,\;
$\Gradx\vv\in L^\infty_{t,x}$ and 
\[
\|\ \vv\cdot\Gradx\Lambdaz(V) - \langle \vv\cdot\Gradx\Lambdaz(V) \rangle  \|_{\dot H^s}
\le
C\,\|\vv \|_{W^{1,\infty}}
\|\Gradx\Lambdaz(V)\|_{\dot H^s},
\qquad 0<s<1.
\] Each right-hand factor tends to $0$ as $\epsilon\to0$, so
\begin{align}\label{eq:N-cross}
	&\int_0^\tau \int_{\mathbb T^d}(\vr-\bar\vr)\,\vv\cdot\Gradx\Lambdaz(V) \dx \dt
	+\int_0^\tau\!\!\int_{\mathbb T^d}\Divx\!\big((r-\bar\vr)\,\vv\big)\Lambdaz(U)\,dx\,dt \nonumber\\
	&\quad 		=2c_{d,s}\int_0^\tau\!\!\iint_\Omega\T_{\rm ER}(U,V):\Gradx\vv\,dx\,dz\,dt .
\end{align}
Substituting \eqref{eq:N-TV} and \eqref{eq:N-cross} into \eqref{N} and using \eqref{Tpolar}, it holds
\begin{align*}
	\mathcal N
	&=c_{d,s}\int_0^\tau\!\!\iint_\Omega\big(\T_{\rm ER}[U]-2\,\T_{\rm ER}(U,V)+\T_{\rm ER}[V]\big):\Gradx\vv \dx \,\dz \dt \\
	&=c_{d,s}\int_0^\tau\!\!\iint_\Omega\T_{\rm ER}[U-V]:\Gradx\vv \dx \,\dz\dt.
\end{align*}

\end{pf}

\begin{rmk}\label{rem:reg-scope}
	The regularization enters only in \eqref{eq:polar-eps} and is removed as $\epsilon\to0$ through
	\eqref{eq:eps-energy}; it touches neither the momentum equation nor the defect measures. In the
	final identity \eqref{N-final}, $U$ appears only through $\Gradxz U\in L^\infty(0,T;L^2(\Omega,z^{1-2s}))$,
	i.e.\ at the energy level.
\end{rmk}

\subsubsection{Proof of the main Theorem \ref{thm:ws}}

\begin{pf}
Throughout, $(\vr,\vm,U,\mathcal R_{fluid},\mathcal R_{ext},\mathcal E_{fluid},\mathcal E_{ext})$ is a
dissipative weak solution and $(r,\vv)$ the strong solution \eqref{strong-class} with the same
initial data. Set
\[
\mathcal Z(\tau):=\Erel(\tau)+\int_{\mathbb T^d}d\mathcal E_{fluid}(\tau)
+\iint_{\overline\Omega}d\mathcal E_{ext}(\tau)\ \ge 0 ,
\]
and let $C>0$ be a generic constant depending only on
$\|\Gradx\vv\|_{L^\infty_{t,x}}, \| r\|_{C^1_{t,x}}\ \underline r,\ \overline r,\ \gamma,\ d,\ \lambda_1,\ \lambda_2, \ c_{d,s}$
--- all finite since $\vv\in C^1_{t,x},\; r\in C^1_{t,x}$ and $0<\underline r\le r\le\overline r$ --- whose value may
change from line to line.

\emph{Step 1: Gr\"onwall inequality.} By Lemma~\ref{lem:reduced}, for a.e. $\tau\in(0,T)$,
\begin{equation}\label{REI3}
	\begin{aligned}
		\mathcal Z(\tau)\ \le\ \Erel(0)
		&- \int_0^\tau\!\!\int_{\mathbb T^d} \vr\lr{\tfrac{\vm}{\vr}-\vv}\otimes\lr{\tfrac{\vm}{\vr}-\vv}:\Gradx\vv \dx \dt
		- \int_0^\tau\!\!\int_{\mathbb T^d} p(\vr\,|\,r)\,\Divx\vv \dx \dt\\
		&+ c_{d,s}\int_0^\tau\!\!\iint_\Omega \T_{\rm ER}[U-V]:\Gradx\vv \dx \,\dz \dt \\
		&
		- \int_0^\tau\!\!\int_{\mathbb T^d} \Gradx\vv : d\mathcal R_{fluid}
		+ c_{d,s}\int_0^\tau\!\!\iint_{\overline\Omega} \Gradx\vv : d\mathcal R_{ext} .
	\end{aligned}
\end{equation}
We bound the last five integrals in the R.H.S of \eqref{REI3}by $\displaystyle C\,\int_0^\tau \mathcal Z(t) \dt$. Since $\Gradx\vv\in L^\infty_{t,x}$ and
$0\le p(\vr\,|\,r)\le C\,P(\vr\,|\,r)$ for $\gamma>1$, $\underline r\le r\le\overline r$
(Remark~\ref{rem:relP}), the convective and pressure terms are controlled by the kinetic and
internal parts of $\Erel$:
\[
\Big|\int_0^\tau\!\!\int_{\mathbb T^d}\vr\lr{\tfrac{\vm}{\vr}-\vv}\otimes\lr{\tfrac{\vm}{\vr}-\vv}:\Gradx\vv \dx \dt  \Big| 
+\Big|\int_0^\tau\!\!\int_{\mathbb T^d}p(\vr\,|\,r)\,\Divx\vv \dx \dt  \Big| 
\ \le\ C\int_0^\tau\Erel \dt .
\]
For the stress term, $|\Gradx W|\le|\Gradxz W|$ gives
$|\T_{\rm ER}[W]:\Gradx\vv|\le C\,z^{1-2s}|\Gradxz W|^2$, so by the isometry \eqref{eq:BI}, we have
\[
\Big| \int_0^\tau\!\!\iint_\Omega\T_{\rm ER}[U-V]:\Gradx\vv \dx\, \dz \dt\Big|
\ \le\ C\,\int_0^\tau\!\!\iint_\Omega z^{1-2s}|\Gradxz(U-V)|^2 \dx\, \dz \dt
\ \le\ C\int_0^\tau\Erel \dt .
\]

This is where retaining the extended energy term in $\Erel$ pays off: the bound is purely algebraic, uniform in $s\in(0,1)$ and independent of $\gamma>1$. For the fluid defect, \eqref{dm-compf} bounds its trace by
$\mathcal E_{fluid}$, and since $\mathcal R_{fluid}\succeq0$ we may pair it with $\Gradx\vv$ through
that trace; for the extension defect, Lemma \ref{lem:Rext-structure} ( more precisely \eqref{Rext-test} bounds the pairing directly by
$\mathcal E_{ext}$. Hence
\[
\Big|\int_0^\tau\!\!\int_{\mathbb T^d}\Gradx\vv:d\mathcal R_{fluid}\Big|
+\Big|c_{d,s}\int_0^\tau\!\!\iint_{\overline\Omega}\Gradx\vv:d\mathcal R_{ext}\Big|
\ \le\ C\int_0^\tau\!\!\int_{\mathbb T^d}d\mathcal E_{fluid}+C\int_0^\tau\!\!\iint_{\overline\Omega}d\mathcal E_{ext}.
\]
Inserting these into \eqref{REI3},
\begin{equation}\label{WS-final}
	\mathcal Z(\tau)\ \le\ \Erel(0)+C\int_0^\tau\mathcal Z(t)\,dt
	\qquad\text{for a.e. }\tau\in(0,T).
\end{equation}

\emph{Step 2: conclusion.} The initial data coincide, so the fluid part of $\Erel(0)$ vanishes; moreover
$U_0$ and $V(0)$ are both the extension of $\vr_0-\bar\vr$, hence $U_0=V(0)$ by
Proposition~\ref{prop:CS_weak}, and the bulk part vanishes too. Thus $\Erel(0)=0$, and
\eqref{WS-final} with $\mathcal Z\ge0$ and vanishing source gives, by Gr\"onwall,
$\mathcal Z\equiv0$ on $[0,T]$. The three nonnegative constituents of $\mathcal Z$ then vanish
separately: $\Erel\equiv0$ and $\mathcal E_{fluid}=\mathcal E_{ext}=0$. The tensor defects follow in
the direction of annihilation: $\operatorname{tr}\mathcal R_{fluid}\le\lambda_1^{-1}\mathcal E_{fluid}=0$
with $\mathcal R_{fluid}\succeq0$ forces $\mathcal R_{fluid}=0$, while
$\mathcal R_{ext}\ll\mathcal E_{ext}=0$ forces $\mathcal R_{ext}=0$. Finally, $\Erel\equiv0$ and
Remark~\ref{rem:coercivity} identify the fields,
\[
\vr=r,\qquad \vm=r\vv,\qquad U=V\qquad\text{a.e. on }[0,T]\times\mathbb T^d .
\]
Hence the dissipative weak solution coincides with the strong one on $[0,T]$; all defects vanish, so
the energy inequality \eqref{eq:ei} reduces to $\mathcal E(\tau)\le\mathcal E(0)$, and since
$(\vr,\vm,U)=(r,r\vv,V)$ is the classical solution, the energy is conserved. 
\end{pf}

	\appendix
	\section{Formal energy identity for Euler-Riesz system}
\label{sec:eng-formal}

We record the formal computation behind \eqref{en_def_ext}, for smooth solutions of \eqref{eq:ERc}--\eqref{eq:ERm}. 

Test the momentum equation \eqref{eq:ERm} with $\vu$. 
Using the following identities 
\begin{align*}
	&\partial_t(\vr\vu)\cdot\vu = \partial_t(\tfrac12\vr|\vu|^2) + \tfrac12|\vu|^2\partial_t\vr,\\ 
	&\Divx(\vr\vu\otimes\vu)\cdot\vu = \Divx(\tfrac12\vr|\vu|^2\vu) + \tfrac12|\vu|^2\Divx(\vr\vu), \\
	&\vu\cdot\Gradx p(\vr)=\Divx(p(\vr)\vu)-p(\vr)\Divx\vu,
\end{align*}
we have
\begin{equation}\label{ap-test}
	\partial_t\big(\tfrac12\vr|\vu|^2\big)
	+ \Divx\big[\big(\tfrac12\vr|\vu|^2 + p(\vr)\big)\vu\big]
	+ \tfrac12|\vu|^2\big(\partial_t\vr + \Divx(\vr\vu)\big)
	- p(\vr)\Divx\vu
	= -\,\vr\,\vu\cdot\Gradx\Phi.
\end{equation}
The third term vanishes by the continuity equation \eqref{eq:ERc}. For the pressure, substitute \eqref{eq:ERc} once more: since $P''(\vr)=p'(\vr)/\vr$,
\begin{equation*}
	\partial_t P(\vr) + \Divx(P(\vr)\vu) = -\,p(\vr)\Divx\vu,
\end{equation*}
so that $-p(\vr)\Divx\vu$ in \eqref{ap-test} is $\partial_t P(\vr) + \Divx(P(\vr)\vu)$. Collecting, the internal energy joins the flux and
\begin{equation*}
	\partial_t\Big(\tfrac12\vr|\vu|^2 + P(\vr)\Big)
	+ \Divx\Big[\big(\tfrac12\vr|\vu|^2 + P(\vr) + p(\vr)\big)\vu\Big]
	= -\,\vr\,\vu\cdot\Gradx \Phi.
\end{equation*}
Integrating over space, we have the following balance
\begin{equation}\label{K-mech}
	\frac{d}{dt}\int_{\mathbb T^d}\Big(\tfrac12\vr|\vu|^2 + P(\vr)\Big) \dx
	= -\int_{\mathbb T^d}\vr\,\vu\cdot\Grad_x\Phi  \dx .
\end{equation}
Before extending, we record the energy in its intrinsic form, in terms of the interaction kernel $K$. Recall $\Phi=K\ast(\vr-\bar\vr)$ and, on $\mathbb T^d$, $\Grad_x(K\ast\vr)=\Grad_x\Phi$.

For the interaction term, $K$ is even, so the bilinear form $(f,g)\mapsto\int f\,(K\ast g)$ is symmetric, and
\begin{equation}\label{K-int}
	\frac{d}{dt}\left[\frac12\int_{\mathbb T^d}(\vr-\bar\vr)\,K\ast(\vr-\bar\vr) \dx\right]
	= \int_{\mathbb T^d}\partial_t\vr\,\big(K\ast(\vr-\bar\vr)\big)  \dx
	= \int_{\mathbb T^d}\partial_t\vr\,\Phi  \dx
	= \int_{\mathbb T^d}\vr\,\vu\cdot\Grad_x\Phi  \dx,
\end{equation}
using $\partial_t\vr=-\Div_x(\vr\vu)$ and integration by parts. Adding \eqref{K-mech} and \eqref{K-int}, the interaction terms cancel and the total energy
\begin{equation}\label{K-eng}
	\mathcal E(t) = \underbrace{\int_{\mathbb T^d}\tfrac12\vr|\vu|^2 \dx}_{\text{kinetic}}
	+ \underbrace{\int_{\mathbb T^d} P(\vr) \dx}_{\text{internal}}
	+ \underbrace{\frac12\int_{\mathbb T^d}(\vr-\bar\vr)\,K\ast(\vr-\bar\vr) \dx\,\dz}_{\text{interaction}}
\end{equation}
is conserved, $\tfrac{d}{dt}\mathcal E(t)=0$.\par 

Now in Extended variable, using  the isometry \eqref{CS-prop1}, it holds 
\begin{equation*}
	\frac12\int_{\mathbb T^d}(\vr-\bar\vr)\,K\ast(\vr-\bar\vr)  \dx
	= \frac{1}{2c_{d,s}}\big\|\vr-\bar\vr\big\|_{\dot H^{-s}}^2
	= \frac{c_{d,s}}2\iint_\Omega z^{1-2s}\,|\Grad_{x,z}U|^2 \dx \, \dz  ,
\end{equation*}
with $\Lambda_0 U=\Phi$. The total energy \eqref{K-eng} thus takes the local form
\begin{equation}\label{en_def_ext}
	\mathcal E(t) = \int_{\mathbb T^d}\Big(\tfrac12\vr|\vu|^2 + P(\vr)\Big) \dx 
	+ \frac{c_{d,s}}2\iint_\Omega z^{1-2s}\,|\Grad_{x,z}U(t)|^2 \dx \,\dz ,
\end{equation}
every term now local in $(x,z)$. \par 

\begin{rmk}[Coercivity in the repulsive case]\label{rem:K-coercive}
	In the repulsive case $K=+C_{d,s}|x|^{-\alpha}$ the interaction term is non-negative. Indeed $\widehat K(k)=c_{d,s}^{-1}|k|^{-2s}\ge0$, so by Plancherel
	\begin{equation}\label{K-post}
		\frac12\int_{\mathbb T^d}(\vr-\bar\vr)\,K\ast(\vr-\bar\vr) \dx
		= \frac12\sum_{k\neq0}\widehat K(k)\,|\widehat\vr(k)|^2
		= \frac{1}{2c_{d,s}}\big\|\vr-\bar\vr\big\|_{\dot H^{-s}(\mathbb T^d)}^2 \ \ge\ 0,
	\end{equation}
	and $\mathcal E$ controls kinetic, internal and interaction energy simultaneously. In the attractive case the same quantity is non-positive, and coercivity must be recovered separately (see the introduction).
\end{rmk}
\begin{rmk}[Comment on the term $ \vr \Grad \Phi$ via Tensor $   \T $] \label{rem:stress-meaning}
	We use the stress structure to give the force meaning. Written via \eqref{T} as $ c_{d,s}\iint_\Omega\mathbb T_{\rm ER}[U]:\Grad_x\pmb{\psi} \dx \dz ,$ with $\mathbb T[U]$ quadratic in $\Grad_{x,z}U$ and hence in $L^\infty_tL^1(\Omega,z^{1-2s})$ by the energy alone, the force is a well-defined element of $L^\infty_tW^{-1,1}(\mathbb T^d)$ --- for every $s\in(0,1)$ and every $\gamma>1$, with no further hypothesis.
\end{rmk}

\section{A discussion on the Caffarelli-Silvestre Extension}	\label{sec:extension-appendix}

This appendix serves two purposes. First, we give the explicit construction of
the extension on the torus $\mathbb{T}^d$. Next, we state a proposition on the
regularity of the extension in the extended variable $z$
(Proposition~\ref{prop:CS_regularity}).

Throughout, $K_\nu$ denotes the modified Bessel function of the second kind of order $\nu$: the
decaying solution of $r^2 K_\nu''+r K_\nu'-(r^2+\nu^2)K_\nu=0$ on $(0,\infty)$, with the two
asymptotics we use
\begin{equation}\label{eq:bessel-asymp}
	K_\nu(r)\sim \tfrac12\Gamma(\nu)\big(\tfrac{2}{r}\big)^{\nu}\quad(r\to0^+)
	\quad \text{and} \quad
	K_\nu(r)\sim \sqrt{\tfrac{\pi}{2r}}\,e^{-r}\quad(r\to\infty).
\end{equation}
so $K_\nu$ is singular of order $r^{-\nu}$ at the origin and decays exponentially at infinity.

\paragraph{Explicit construction on the torus.}
On $\mathbb T^d$ the extension is diagonal in Fourier, in subordinated semigroup form \cite[eq.~(3.4)]{RS2016}, which is how we fix the constant $c_{d,s}$.%
\footnote{Roncal and Stinga write this extension in subordinated heat-semigroup
	form (the torus case of their Theorem~1.1): with $\sigma=2s$, $\tau=z$, and
	boundary datum $g:=\Lambda_0 U$,
	\[
	U(x,z)=\frac{z^{2s}}{4^{s}\Gamma(s)}\int_0^\infty e^{t\Delta}g(x)\,
	e^{-z^2/4t}\,\frac{dt}{t^{1+s}}
	=\sum_{k\neq0}\hat g_k\,e^{ik\cdot x}\,\frac{z^{2s}}{4^{s}\Gamma(s)}
	\int_0^\infty e^{-t|k|^2}e^{-z^2/4t}\,\frac{dt}{t^{1+s}} .
	\]
	Evaluating the $t$-integral by
	$\int_0^\infty t^{\nu-1}e^{-at-b/t}\,dt=2(b/a)^{\nu/2}K_\nu(2\sqrt{ab})$
	($\nu=-s$, $a=|k|^2$, $b=z^2/4$) yields $\psi(|k|z)$, so \eqref{eq:torus-U}
	coincides with \cite[eq.~(3.4)]{RS2016}; the Bessel profile is the closed form
	of their semigroup kernel. Their normalizing constant is $c_{d,s}^{-1}$: we fix
	the constant by the Dirichlet energy \eqref{CS-prop1}, they by the co-normal
	derivative.}
Equivalently, the extension can be written as a single Fourier series in the modes. Since $f=\vr-\bar\vr$ is mean-zero, only nonzero modes appear: writing $\vr-\bar\vr=\sum_{k\neq0}\hat\rho_k e^{ik\cdot x}$,
\begin{equation}\label{eq:torus-U}
	U(x,z)=\sum_{k\neq0}|k|^{-2s}\hat\rho_k\,\psi(|k|z)\,e^{ik\cdot x},
\end{equation}
where $\psi$ is the unique decaying solution of
\begin{equation}\label{eq:profile}
	\psi''+\frac{1-2s}{r}\psi'-\psi=0,\qquad \psi(0)=1,\quad \psi(\infty)=0,
\end{equation}
namely $\psi(r)=\tfrac{2^{1-s}}{\Gamma(s)}r^{s}K_s(r)$; the normalization $\psi(0)=1$ follows from
the small-$r$ asymptotic in \eqref{eq:bessel-asymp}, and for $s=\tfrac12$ it reduces to
$\psi(r)=e^{-r}$. The co-normal condition then fixes the constant. Using $\psi'(r)\sim -\tfrac{2^{1-2s}\Gamma(1-s)}{\Gamma(s)}\,r^{2s-1}$ as $r\to0$, so that $-\lim_{z\to0}z^{1-2s}\partial_z U=\tfrac{2^{1-2s}\Gamma(1-s)}{\Gamma(s)}\sum_{k\neq0}|k|^{2s}\hat\Phi_k e^{ik\cdot x}=\tfrac{2^{1-2s}\Gamma(1-s)}{\Gamma(s)}(-\Delta)^s\Phi$, matching \eqref{CS-ext} gives
\begin{align}\label{constCDS}
	c_{d,s}=\frac{\Gamma(s)}{2^{1-2s}\,\Gamma(1-s)} .
\end{align}

The isometry \eqref{CS-prop1} follows by Parseval identity: $$c_{d,s}\iint_\Omega z^{1-2s}|\Grad_{x,z}U|^2=\sum_{k\neq0}|k|^{2s}|\hat\Phi_k|^2=\|\Phi\|_{\dot H^s}^2=\|f\|_{\dot H^{-s}}^2.$$

\subsubsection*{Higher Sobolev and H\"older regularity of the extension}

The extension lifts the regularity of the boundary datum by $2s$ derivatives. Part~(i) below is a
direct Fourier-side computation, which we detail as it is the admissibility we use in
Lemma~\ref{lem:stress-id}; part~(ii) is the periodic fractional Schauder theory of
\cite[Thm.~1.4]{RS2016}.

\begin{prop}[Regularity lifting]\label{prop:CS_regularity}
	Let $f\in\dot H^{-s}(\mathbb T^d)$ have zero mean, $\int_{\mathbb{T}^d} f\,dx = 0$,
	and let $U=\mathfrak{E}(f)$ be its extension following \eqref{CS-ext}.
	\begin{enumerate}[leftmargin=*]
		\item If $f\in H^{k_0}(\mathbb T^d)$, $k_0\ge0$, then for every multi-index $\alpha$ and
		$m\in\{0,1\}$ with $|\alpha|+m\le k+1$,
		\[
		\|\partial_x^\alpha\partial_z^m U\|_{L^2(\Omega,\,z^{1-2s})}\le C(d,s)\,\|f\|_{H^{k_0}},
		\]
		and in fact $\|\partial_x^\alpha\partial_z^m U\|_{L^2(\Omega,z^{1-2s})}
		\le C(d,s)\,\|f\|_{\dot H^{\,|\alpha|+m-1-s}}$. 
		Moreover, $\Lambda_0 U=(-\Delta)^{-s}f$ gains exactly $2s$ derivatives: since
		$\widehat{\Lambda_0 U}(k)=|k|^{-2s}\hat f(k)$,
		\[
		\|\Lambda_0 U\|_{\dot H^{\sigma}}=\|f\|_{\dot H^{\sigma-2s}}\qquad(\sigma\in\mathbb R).
		\]
		On the mean-zero class the homogeneous and inhomogeneous norms are comparable, so
		$\Lambda_0 U\in H^{k_0+2s}(\mathbb T^d)$ with $\|\Lambda_0 U\|_{H^{k_0+2s}}\le C(d,s)\,\|f\|_{H^{k_0}}$.
		\item If $f \in C^{k_0,\tilde{\beta}}(\mathbb{T}^d)$, $k_0\ge0$ integer, $\tilde{\beta}\in(0,1]$,
		$k_0+\tilde{\beta}+2s\notin\mathbb Z$, then $\Lambda_0 U \in C^{k_0,\tilde{\beta}+2s}(\mathbb{T}^d)$
		and $U\in C^\infty(\mathbb T^d\times(0,\infty))$; cf.\ \cite[Thm.~1.4]{RS2016}.
	\end{enumerate}
\end{prop}

The restriction $m\le1$ in (i) is necessary: $\partial_z^2 U\notin L^2(\Omega,z^{1-2s})$ in general
(for $s\ne\tfrac12$).

\begin{pf}[Proof of (i)]
By \eqref{eq:torus-U}, $\partial_{x_j}$ multiplies the $k$-th mode by $ik_j$ and leaves the profile
$\psi(|k|z)$ untouched, so tangential derivatives commute with the extension and
$\Grad_x\Grad_{x,z}U$ is again diagonal in Fourier. Fix a multi-index $\alpha$ and $m\in\{0,1\}$.
and differentiate \eqref{eq:torus-U} to obtain
\[
\partial_x^\alpha\partial_z^m U
= \sum_{k\neq0} (ik)^\alpha\,|k|^{-2s+m}\,\hat\rho_k\,\psi^{(m)}(|k|z)\,e^{ik\cdot x},
\]
so that, by using Parseval identity in $x$, it holds
\[
\|\partial_x^\alpha\partial_z^m U\|_{L^2(\Omega,z^{1-2s})}^2
= \sum_{k\neq0} |k|^{2|\alpha|-4s+2m}\,|\hat\rho_k|^2
\int_0^\infty z^{1-2s}\,\psi^{(m)}(|k|z)^2\,dz .
\]
The rescaling $r=|k|z$ gives $z^{1-2s}\,dz=|k|^{2s-2}\,r^{1-2s}\,dr$, so the $z$-integral equals
$|k|^{2s-2}A_s^{(m)}$ with $A_s^{(m)}:=\int_0^\infty r^{1-2s}\,\psi^{(m)}(r)^2\,dr$, independent of $k$.
Hence
\begin{equation}\label{eq:bulk-Hk}
	\|\partial_x^\alpha\partial_z^m U\|_{L^2(\Omega,z^{1-2s})}^2
	= A_s^{(m)}\sum_{k\neq0} |k|^{2|\alpha|+2m-2-2s}\,|\hat\rho_k|^2
	= A_s^{(m)}\sum_{k\neq0} |k|^{2(|\alpha|+m-1-s)}\,|\hat\rho_k|^2 .
\end{equation}
The constant $A_s^{(m)}=\int_0^\infty r^{1-2s}\,\psi^{(m)}(r)^2\,dr$ is finite for $m\in\{0,1\}$
and every $s\in(0,1)$. Since $\psi(r)=\tfrac{2^{1-s}}{\Gamma(s)}r^sK_s(r)$ and $K_s$ is smooth and
positive on $(0,\infty)$, the integrand $r^{1-2s}\psi^{(m)}(r)^2$ is continuous on $(0,\infty)$,
so we only check the two ends.
\begin{itemize}[leftmargin=*]
	\item \emph{As $r\to\infty$.} By \eqref{eq:bessel-asymp}, $K_s(r)\sim\sqrt{\pi/2r}\,e^{-r}$,
	so $\psi$ and $\psi'$ decay exponentially; the integrand does too, and is integrable at $\infty$
	for both $m=0,1$.
	\item \emph{As $r\to0^+$.} Here $\psi(r)\to1$ and $\psi'(r)\sim c\,r^{2s-1}$ (from the small-$r$
	expansion of $K_s$), so
	\[
	r^{1-2s}\psi(r)^2\sim r^{1-2s}\quad(m=0),
	\qquad
	r^{1-2s}\psi'(r)^2\sim r^{1-2s}\cdot r^{2(2s-1)}=r^{2s-1}\quad(m=1).
	\]
	The exponents $1-2s$ and $2s-1$ are negatives of each other, and for $s\in(0,1)$ both lie in
	$(-1,1)$; in particular both are $>-1$, so each integrand is integrable at the origin.
\end{itemize}
Hence $A_s^{(m)}<\infty$. By \eqref{eq:bulk-Hk} the right-hand side is
$A_s^{(m)}\|f\|_{\dot H^{|\alpha|+m-1-s}}^2$, and since $|k|^{|\alpha|+m-1-s}\le|k|^{k_0}$ on the
nonzero modes when $|\alpha|+m\le k_0+1$, it is bounded by $A_s^{(m)}\|f\|_{H^{k_0}}^2$.

In particular, taking $|\alpha|+m=2$ (one tangential and one full
derivative), a mean-zero $f\in H^1(\mathbb T^d)$ yields $\Grad_x\Grad_{x,z}U\in L^2(\Omega,z^{1-2s})$
--- exactly the admissibility required in \eqref{stress-id}. Finally, the trace bound
$\|\Lambdaz U\|_{\dot H^\sigma}=\|f\|_{\dot H^{\sigma-2s}}$ is immediate from
$\widehat{\Lambdaz U}(k)=|k|^{-2s}\hat f(k)$. 
\end{pf}

\section*{Concluding remarks}

We close with the scope and limitations of the approach.
\begin{itemize}[leftmargin=*]
	\item \emph{Two divergence-form representations of the Riesz term.} The nonlocal force
	$\varrho\nabla K*\varrho$ can be written as a divergence in more than one way. On $\mathbb R^d$ with
	vanishing far field, Alves et. al \cite[Eq. 1-4]{AlGT2026} use the two-point representation
	$\varrho\nabla K_\alpha*\varrho=\Div S_\alpha(\varrho)$,
	\[
	S_\alpha(\varrho)(x)=\tfrac12\int_0^1\!\!\int_{\mathbb R^d}
	\varrho\big(x+(\theta-1)y\big)\,\varrho(x+\theta y)\,|y|^{\alpha-d-2}\,y\otimes y\,\dd y\,\dd\theta,
	\]
	built directly from the kernel, whereas here the same force is $\Div\!\big(\int_0^\infty
	\mathbb T_{ER}[U]\,dz\big)$, a boundary flux of a local stress in one higher dimension. The two
	tensors have the same divergence --- hence represent the same force and agree against any gradient
	test field --- but are not equal: the same singular interaction is localized through different
	auxiliary variables, a separation vector $y$ for $S_\alpha$ in \cite{AlGT2026} and an extension
	height $z$ for $\mathbb T_{ER}$ here. A systematic comparison, and a full-space version of the
	present approach, seem worth investigating.
	\item \emph{General monotone pressure.} The power law $p(\vr)=a\vr^\gamma$ enters our
	argument only through the convexity of the pressure potential and the proportionality
	$p=(\gamma-1)P$. The analysis extends to a general monotone pressure in the sense of
	Abbatiello, Feireisl and Novotn\'y~\cite{AFN2021}: $p\in C[0,\infty)\cap C^2(0,\infty)$,
	$p(0)=0$, $p'(\vr)>0$ for $\vr>0$, with pressure potential determined by
	$\vr P'(\vr)-P(\vr)=p(\vr)$, $P(0)=0$, and such that $P-\underline a\,p$ and
	$\bar a\,p-P$ are convex for some constants $\underline a,\bar a>0$. As shown in~\cite{AFN2021}, these assumptions imply $P(\vr)\geq a \vr^\gamma$ for some $a>0$ and $\gamma=1+\frac{1}{\bar{a}}$ for $\vr>1$. This two-sided convexity
	is exactly what renders the fluid defect $\mathcal R_{fluid}$ controllable by the energy
	defect $\mathcal E_{fluid}$ (cf.\ \eqref{dm-compf}); the coupling with the extension and the
	nonlocal stress is unaffected, so the relative-energy inequality carries over under this
	weaker hypothesis.
	
	\item \emph{The linear (isothermal) pressure.} The endpoint $p(\vr)=\vr$ ($\gamma=1$) is
	more delicate and is not covered above. Here $P(\vr)=\vr\log\vr$ (up to an affine term),
	convex, so the a priori bound places $\vr$ in the
	Orlicz space $L\log L(\mathbb T^d)$ rather than in $L^\gamma$, $\gamma>1$. The integrability
	gain that $L^\gamma$ supplies for free is then absent, instead the weak convergence of density must be drawn from equi-integrability (de la Vall\'ee-Poussin). Adapting the present framework to
	this case requires this additional work, which we do not pursue here; see Basari\'c \cite{B2021} for
	the linear-pressure theory in the viscous setting.
\end{itemize}
\paragraph{Acknowledgment}
The work is supported by the NAWA ULAM grant BPN/SEL/2025/1/ 00005/U/00001. The author would like to thank Piotr Mucha and Jan Peszek for organizing the reading seminar, which greatly benefited this work. The author also thanks Maja Szlenk for fruitful discussions on this topic.

	\bibliographystyle{abbrv}
	\bibliography{biblionil}
	
	\end{document}